\documentclass[11pt]{article}
\usepackage[letterpaper, portrait, margin=1in]{geometry}
\usepackage[T1]{fontenc}
\usepackage[utf8]{inputenc}
\usepackage[export]{adjustbox}
\usepackage{graphicx, subcaption, float}
\usepackage{amsmath, amsthm, amssymb, mathtools, mathrsfs, thmtools, bbm}
\usepackage{gensymb, mdframed, listings, xcolor, soul, indentfirst, verbatim, enumerate, enumitem, comment, longtable, hyperref}
\usepackage{tikz}
\usetikzlibrary{arrows.meta,positioning}

\theoremstyle{plain}
\newtheorem{theorem}{Theorem}[section]  % Numbering by section
\newtheorem{lemma}[theorem]{Lemma}  % Shares numbering with theorem
\newtheorem{definition}[theorem]{Definition} 
\newtheorem{corollary}[theorem]{Corollary} 
\newtheorem{proposition}[theorem]{Proposition} 
\newtheorem{remark}[theorem]{Remark} 
\newtheorem{assumption}{Assumption}

\definecolor{codegreen}{rgb}{0,0.6,0}
\definecolor{codegray}{rgb}{0.5,0.5,0.5}
\definecolor{codepurple}{rgb}{0.58,0,0.82}
\definecolor{backcolour}{rgb}{0.95,0.95,0.92}
\lstdefinestyle{mystyle}{
    backgroundcolor=\color{backcolour},   
    commentstyle=\color{codegreen},
    keywordstyle=\color{magenta},
    numberstyle=\tiny\color{codegray},
    stringstyle=\color{codepurple},
    basicstyle=\ttfamily\footnotesize,
    breakatwhitespace=false,         
    breaklines=true,                 
    captionpos=b,                    
    keepspaces=true,                 
    numbers=left,                    
    numbersep=5pt,                  
    showspaces=false,                
    showstringspaces=false,
    showtabs=false,                  
    tabsize=2
}
\title{Approximation of Singular-Stopping Control Driven by Hawkes Processes via Rescaled MDPs}
\author{Isabel \textsc{AGOSTINO}\footnote{\texttt{iagostino@berkeley.edu}} \quad Thibaut \textsc{MASTROLIA}\footnote{\texttt{mastrolia@berkeley.edu}}\\[0.5em]
UC Berkeley, Department of Industrial Engineering and Operations Research}
\date{\today}

\begin{document}

\maketitle

\begin{abstract}
   We investigate a singular–optimal stopping stochastic control problem driven by self-exciting dynamics governed by a Hawkes process. In the continuous-time setting, we show that the optimization problem reduces to solving a variational partial differential equation  with gradient constraints. We then introduce its discrete-time counterpart, modeled as a Markov Decision Process. We prove that, under an appropriate rescaling procedure, the value function of the discrete-time problem converges to its continuous-time equivalent, implying that the discrete-time optimizers are asymptotically optimal for the continuous-time problem. Finally, we apply these results to an Ornstein-Uhlenbeck stochastic differential equation driven by a Hawkes process with singular control, motivated by optimal power plant investment under cyber threat and we illustrate the theoretical findings through numerical simulations. 
\end{abstract}

\section{Introduction}
\subsection{Cyber Risk Management as a Motivating Framework}

Singular stochastic control emerged in the 1960s with the foundational work of Bather and Chernoff \cite{bather1967sequential}, who introduced a stylized spaceship-approach problem to study optimal fuel management under uncertainty. In their formulation, the controller must continuously balance two competing effects: early interventions, which can be made when the remaining time is large but rely on imprecise information, and late interventions, which benefit from improved information but become increasingly ineffective as the remaining time $T-t$ diminishes. The resulting finite-horizon control problem consists of minimizing a quadratic terminal cost together with a running control cost of the form $(T-t)^{-1}\, d\xi_t$, where $\xi$ denotes the cumulative fuel usage. The control process $\xi$ is modeled as a nondecreasing, bounded-variation process, thereby extending classical stochastic control problems in which $\xi$ is typically assumed to be absolutely continuous with respect to the Lebesgue measure. This seminal example motivated an extensive development of the mathematical theory of singular control; see, among many others, \cite{alfonsi2013capacitary,alvarez2001singular,benevs1980some,cadenillas1994stochastic,dufour2004singular,haussmann1995singularI, haussmann1995singularII,karatzas1983class,karatzas2000finite,karoui1988probabilistic,shreve1988introduction}. Singular control techniques have since been applied in a wide variety of domains. In queueing systems, singular controls model instantaneous adjustments of the workload or queue length through admission control, server-speed modulation, or enforced idleness that pushes the state process; see \cite{harrison1987brownian,iglehart1970multiple,martins1990routing}. In biology and renewable-resource management, such controls arise in optimal harvesting and population regulation problems \cite{hening2019harvesting}. In mathematical finance, they play a central role in models of optimal investment, consumption, and dividend distribution \cite{baldursson1996irreversible,reppen2020optimal,shreve1994optimal}.While in energy economics, singular controls are used to describe irreversible or partially reversible investment decisions for power-plant capacity planning \cite{aid2015explicit,koch2021optimal}. For a comprehensive treatment of singular control, its analytical foundations, and its applications, we refer to the monograph \cite{fleming2006controlled} and the recent Ph.D. thesis \cite{sun2024optimal}.\vspace{1em}

Recalling the examples of the spaceship fuel problem and power-plant investment in energy management, one notes that these sectors have become increasingly exposed to cyberattacks in recent years; see, among others, \cite{hemmati2022identification,stergiopoulos2020cyber,yildiz2024enhancing}. These empirical observations highlight the need to extend existing studies on singular control by incorporating cyber threats into the underlying system dynamics. A recent empirical and theoretical article \cite{baldwin2017contagion} has demonstrated by using real data from the Wannacry attacks that cyberattacks exhibit self-exciting features, and Hawkes processes, as a class of self-exciting counting processes, successfully reproduce the clustering and contagion patterns observed in practice. In recent years, stochastic control problems involving self-exciting dynamics have gained increasing attention due to their capacity to model systems in which the occurrence of an event modifies the likelihood of future events. Such feedback effects are naturally captured by Hawkes processes, whose intensities evolve in response to their past jump activity. Self-exciting dynamics arise in various applications, including finance, insurance, and cyber–physical systems. In the context of power-plant investment under cyber risk, for instance, attacks or failures may temporarily increase vulnerability, leading to clusters of incidents. Optimal management of such systems requires balancing continuous operational adjustments with discrete, possibly abrupt interventions aimed at mitigating risk and maintaining stability. Within this framework, we study a singular-mixed stochastic control problem driven by Hawkes dynamics. The controller acts through both regular (absolutely continuous) and singular (bounded-variation) components, reflecting a combination of gradual and instantaneous decisions. The presence of a self-exciting jump mechanism introduces a non-Markovian dependence structure, which renders the optimization problem analytically challenging. In the continuous-time setting, we show that the problem can be characterized by a variational Hamilton-Jacobi-Bellman equation with gradient constraints, providing a natural extension of classical singular control theory to systems with endogenous jump intensities.

\subsection{State of the art and contribution}
This work contributes to the intersection of singular control and self-exciting stochastic systems. Early studies on stochastic control and optimal stopping for jump-diffusion processes, such as \cite{alvarez2009singular,an2010combined,ceci2004mixed}, established verification theorems and regularity properties of the value function while numerical approximations based on partial differential equation (PDE) discretization schemes have been developed, for instance, in \cite{dumitrescu2021approximation}. Our formulation extends this line of work by incorporating Hawkes-driven jumps, which induce both state-dependent and history-dependent jump intensities. In contrast to classical approaches relying on PDE discretization, we investigate a purely discrete-time control model that is provably equivalent to the continuous-time optimization problem, thereby circumventing the need for direct numerical treatment of the variational PDE.\vspace{0.5em}

Starting with the seminal work of Haussmann and Suo \cite{haussmann1995singularI, haussmann1995singularII}, the theory of singular stochastic control has been deeply connected with optimal stopping problems. This relationship has been further clarified in \cite{guo2008connections, guo2009class,bovo2025variational}, which established analytical links between free-boundary formulations and variational inequalities. Comprehensive treatments of singular control theory and its connections to stochastic calculus and viscosity solutions can be found in the monographs \cite{fleming2006controlled,shreve1988introduction}. These foundational results provide the analytical background for our continuous-time formulation and motivate the study of discrete approximations under self-exciting dynamics.\vspace{0.5em}

The control of systems influenced by Hawkes processes has been investigated in \cite{bensoussan2024stochastic,aubert2025optimal}, where the authors analyzed stochastic control problems with self-exciting jumps, and in \cite{khabou2025markov}, which developed a Markov approximation framework for controlled jump processes. On the discrete-time side, the theory of Markov decision processes (MDPs) is well established, see for example the classical references \cite{howard1972risk, puterman2014markov} and has been widely applied to decision-making, dynamic games, and stochastic control \cite{bauerle2011markov, feinberg2012handbook}. Moreover, MDPs have been successfully employed as approximation schemes for singular stochastic control problems under various structural settings. Convergence and consistency results have been obtained for state-constrained systems \cite{budhiraja2007convergent}, multi-species population models \cite{hening2019harvesting}, scaling limits in population dynamics \cite{jusselin2023scaling}, and switching jump-diffusion models with capital injections in finance \cite{jin2013numerical, tran2016numerical}.\vspace{1em}

The main contributions of this paper are threefold. 

\begin{enumerate}
    \item We propose a continuous and discrete time formulation of Singular-mixed stopping stochastic control problem in continuous time driven by a Hawkes process in a general framework. We derive the associated variational Hamilton-Jacobi-Bellman partial differential equation  with gradient constraints in the continuous time setting. 
\item Discrete-time counterpart modeled as an MDP.  We establish, via a suitable rescaling argument, that the discrete-time value function converges to its continuous-time analog. This convergence result implies that the discrete-time optimal controls are asymptotically optimal for the continuous-time problem. We propose in addition a new discretization of the controlled Hawkes process adapted to our MDP framework.
\item Application to a singular controlled Ornstein-Uhlenbeck process driven by Hawkes jumps motivated by power plant investment under cyber threat. This example illustrates how the general results specialize to a tractable model and demonstrates the accuracy of the discrete-time approximation through numerical experiments. Overall, the paper extends existing Markov approximation techniques to non-Markovian jump environments in systems exhibiting self-excitation.
\end{enumerate}

The structure of this work is the following. In Section \ref{sec:continuous} we introduce the continuous time problem formulation, including the spaces and notations used. Section \ref{sec:hawkes} recall the definition of Hawkes processes with exponential kernels and properties of the intensity process. Section \ref{sec:controlSDE} investigates integrability properties of the solution to a stochastic differential equation driven by a Hawkes process with singular control. Section \ref{sec:controlcontinuous} presents the stopping-singular optimization \eqref{controlpb-continuous}, the variational gradient-constrained PDE \eqref{VPDE} associated to this problem, and verifications theorems, see Theorems \ref{thm:cont_supersolution} and \ref{thm:verif}. Section \ref{sec:bounded} introduces reflection process to contain the state process into a bounded domain, as a preliminary step to the MDP framework and first convergence results in Proposition \ref{prop:VL_to_V}. Section \ref{sec:MDP} presents the discrete time corresponding framework, starting with the discretization of the Hawkes process in Section \ref{sec:hawkesdisc}, then the MDP formulation of the problem in Section \ref{sec:MCdisc}. Convergence results are stated in Section \ref{sec:cv} starting with the convergence of the MDP to the controlled SDE in Section \ref{sec:MDPtoSDE} using time rescaling and tightness properties. The main convergence results are provided in Section \ref{sec:convvalue}. We prove in particular that the value function for the discrete time problem converges to the continuous time value function of the continuous time optimization on bounded domain from the time rescaling, which itself converge to the value function of the continuous time unbounded domain optimization problem. We also show that any optimizers of the discrete time problem provide an $\varepsilon$-optimizer for the continuous time limit problem. Finally, Section \ref{secnumeric} illustrates the results for an Ornstein–Ulhenbeck process driven by Hawkes processes.

\section{The continuous time model}\label{sec:continuous}
In this section, we introduce the continuous time version of the optimization problem we are interested into. The controlled SDE is given by \eqref{SDE:singular} below, while the optimization problem is \eqref{controlpb-continuous}. 
\subsection{Spaces and notations} 
We follow the framework introduced in \cite{hillairet2024chaotic,khabou2025markov} in the context of marked processes combined with a continuous diffusion process. 
In all this work, $T>0$ denotes a future horizon fixed and we denote by $(\Omega,\mathcal F,\mathbb P)$ a probability space, where $\Omega=\Omega^c\times \Omega^\Pi$ is composed of the Wiener space $\Omega^c:=\mathcal C([0,T];\mathbb R)$ and the space of jump configurations, $\Omega^\Pi$ defined by:
$$\Omega^\Pi = \Bigg\{ \omega^\Pi = \sum_{i=1}^n \delta_{y_i}, \ y_i=(t_i, \theta_i, z_i)\in[0,T]\times\mathbb{R}_+\times\mathbb{R}_+, \ i = 1,\dots, n, \ 0 = t_0<T_1<\cdots <t_n=T  \Bigg\}.$$
We assume that this space is endowed with a Brownian motion $W$. Let $\omega:=(\omega^c,\omega^\Pi)\in \Omega$. We have $W_t(\omega)=\omega^c_t$. We denote by $\mathbb P^\Pi$ the Poisson measure under which the counting process $\Pi$ defined by \[\Pi([0,t]\times [0,\theta]\times (-\infty,z]):=\omega^\Pi([0,t]\times [0,\theta]\times [0,z]),\; (t,\theta,z)\in [0,T]\times \mathbb R_+\times \mathbb R_+,\]
is a homogeneous Poisson process with intensity measure $dtd\theta \text{m}(dz)$ and where $\text{m}(dz)$ is a probability measure on $\{\mathbb{R}_+, \mathcal{B}(\mathbb{R}_+)\}$ with  support in a compact subset $\mathfrak P\subsetneq \mathbb R_+$ and $\text{m}(\{0\}) = 0$.
We assume moreover that $\int_{\mathbb R}|z|\text{m}(dz)<\infty$. We recall that the natural filtration associated with $(W,\Pi)$ is given by
\[\mathcal F_t:=\sigma(W_t)\vee \mathcal F_t^\Pi,\quad \mathcal F_t^\Pi:=\sigma\Big(\Pi\big(\mathcal B([0,t])\times\mathcal B(\mathbb R_+)\times \mathcal B(\mathbb R_+)\big)\Big).\]

We denote by $\mathcal{T} _{[0,T]}$ the set of $\mathbb F$-stopping time on $[0,T]$, where $\mathbb F:=(\mathcal F_t)_{t\in [0,T]}$. We set $\mathcal D[0,\infty)$ the space of c\`adl\`ag functions that are defined on $[0,\infty)$ and take values in $\mathbb R$, and $\mathcal I[0,\infty)$ (resp. $\mathcal {BV}[0,\infty)$) its restriction to nondecreasing (resp. bounded variations) functions.\\

Note that we view $\mathrm m$ as a random element taking values in the space $\mathcal M(\mathbb R_+)$
of finite Borel measures on $\mathbb R_+$, endowed with the topology of weak convergence (or, equivalently, the topology induced by bounded continuous test functions).

\subsection{Hawkes process and exponential kernels}\label{sec:hawkes}

\subsubsection{Hawkes' formulation}
First developed by Hawkes in 1971 \cite{hawkes1971spectra}, a Hawkes process is a self-exciting time-inhomogeneous counting process with a predictable intensity.
That is, the intensity $\lambda_t$, can be determined from $\mathcal{F}_{t-}$.Hawkes processes are traditionally used to model the clustering of earthquakes and aftershocks in seismology. The key feature is that the occurrence of an event increases the likelihood of future events in the near term, creating temporal dependence. In the past few decades, Hawkes processes have found broader applications in finance (modeling trade arrivals and price jumps), neuroscience, social networks, and more recently in cybersecurity and risk modeling, where they capture the clustering and contagion effects of cyberattacks.
Here we recall the concept of a marked Hawkes process and its counting process. The version presented below is a special case of the
generalized marked Hawkes process introduced in \cite{bielecki2022construction}. Let $(\Omega,\mathcal{F},\mathbb{P})$  be a probability space and consider a marked point process $N$ with  mark space $(\mathbb{R}_+,\mathcal{B}(\mathbb{R}_+ ))$ that is, a sequence of random elements $
((T_n, Z_n))_{n \geq 1},$
where $(T_n)_n$ and $(Z_n)_n$ are sequences of non-decreasing non-negative random variables respectively. We associate with the process $N$ an integer-valued
random measure, also denoted by
$N$ and defined as
\begin{equation*}\label{eq:NH-G}
N(dt,dz) := \sum_{n \geq 1 } \delta_{(T_n, Z_n)} (dt, dz) \mathbf 1_{\{T_n < \infty\}}, \text{ so that }
N((0,t],A)=\sum_{n\geq 1} \mathbf 1_{\{ T_n \leq t,\, Z_n \in A \}},\; A\in\mathcal{B}(\mathbb{R}_+),
\end{equation*}
where $\delta$ denotes the Dirac measure on $\mathbb R_+\times \mathbb R_+$.
The corresponding Hawkes kernel is given as
\begin{equation*}\label{eq:kappa-G}
\psi(t,A)=\left (\lambda^\infty_t +\int_{(0, t)\times \mathbb{R}_+ } \phi(t-s,z)N (ds,dz)\right )\mathbb{Q}(A),
\end{equation*}
for $t\geq 0$ and $A\in \mathcal{B}(\mathbb{R}_+ ) $, where $\mathbb{Q}$ is a probability measure on $(\mathbb{R}_+,\mathcal{B}(\mathbb{R}_+ ))$,
$\eta$ is a positive and predictable process, and $\phi^\circ(\cdot,\cdot)$ is a non-negative, bounded measurable function. 
\begin{remark}\label{rem:exp}
    If $\phi(t)=\alpha e^{-\beta t}$ for any $t>0$ then the couple $(N_t,\lambda_t)$ is a
Markov process and in particular $\lambda$ is solution to the following SDE
\[
d\lambda_t = \beta\big(\lambda_{\infty}-\lambda_t\big)\,dt + \alpha\, dN_t,
\] where we use the abuse of notation $N_t=\int_{(0,t) \times \mathbb{R}_+ } \phi(t-s,z)N (ds,dz)$ with solution given by

\[
\lambda_t =e^{-\alpha t}\lambda_0
+ \bigl(1 - e^{-\alpha t}\bigr)\lambda_\infty
+ \beta \int_{0}^{t} e^{-\alpha (t-s)} \, dN_s .
\]
\end{remark}

In this paper, we consider a Hawkes process with an exponential kernel as an extension of Remark \ref{rem:exp} above; specifically, define \[\varphi(t) := \sum_{j=1}^n d_j e^{-q_jt},\; (d_1,\dots,d_n)\in \mathbb R^n,\; (q_1,\dots,q_n)\in (0,\infty)^n.\]

\subsubsection{Br\'emaud-Massouli\'e's formulation}

One drawback of the historical formulation proposed by Hawkes is that the definition of the process is not self-contained in the sense that the intensity $\lambda$ depends on the process $N$ which depends itself on $\lambda$.
The thinning method has been an effective and very elegant approach to define a Hawkes process as a system of self-contained SDE solved by $(N,\lambda)$.  Originally proposed by Ideka \cite[Page 93, (7.27)]{ikeda2014stochastic} the thinning method was rigorously defined by Br\'emaud and Massouli\'e in \cite{bremaud1996stability} and also appears in early works for simulation in \cite{lewis1976simulation,ogata1981lewis}. The main intuition consists of considering a Poisson random measure $\mu$ on $ \mathbb R_+\times \mathbb R_+\times \mathbb R_+$ with intensity measure $dtd\theta \nu(dz)$ providing an infinite cloud of potential events scattered on the plane $(s,u)$; then at each time $s$ an event is accepted if the vertical coordinate $u$ lies below the current intensity $\lambda_{s-}$. This method ensures that $(N,\lambda)$ are solution to a coupled system of SDEs, represented via the Poisson embedding method developed by Brémaud and Massoulié \cite{bremaud1996stability}:
\begin{equation}\label{eq: lambda_t}
    \lambda_t = \lambda_\infty(t) + \int_0^{t-}\int_{\mathbb{R}_+}\int_{\mathbb{R}_+} \varphi(t-s)\varrho(z) \textbf{1}_{\theta \leq \lambda_s}\Pi(ds, d\theta, dz) = \lambda_\infty(t) + \sum_{j=1}^n d_j\varsigma^{(j)}_t,
\end{equation}
where each component $\varsigma_t^{(j)}$ is defined by $\varsigma_t^{(j)} = \int_0^{t-}\int_{\mathbb{R}_+}\int_{\mathbb{R}_+} e^{-q_j(t-s)}\varrho(z)\textbf{1}_{\theta\leq\lambda_s}\Pi(ds, d\theta, dz)$ and $\lambda_\infty(\cdot)$ is deterministic. We assume from now on that the following assumption is enforced along the study. 
\begin{assumption}\label{as: L_inft_bound} There exists $\Lambda_\infty<\infty$ such that
    $\sup_{s\in[0,T]}\lambda_\infty(s) \leq \Lambda_\infty$.
\end{assumption}

From \cite{khabou2025markov} and Assumption \ref{as: L_inft_bound}, we can now write $\lambda$ as a first order SDE of the following form:
$$\begin{cases}
    d\varsigma^{(1)}_t = -q_1\varsigma^{(1)}_tdt + \int_{\mathbb{R}_+}\int_{\mathbb{R}_+} \varrho(z)\textbf{1}_{\theta\leq \lambda_\infty(t) + \sum_{j=1}^n d_j\varsigma^{(j)}_t}\Pi(dt,d\theta,dz) \\
    \vdots \\
    d\varsigma^{(n)}_t = -q_n\varsigma^{(n)}_tdt + \int_{\mathbb{R}_+}\int_{\mathbb{R}_+} \varrho(z)\textbf{1}_{\theta\leq \lambda_\infty(t) + \sum_{j=1}^n d_j\varsigma^{(j)}_t}\Pi(dt,d\theta,dz) .
\end{cases}$$
The Poisson embedding of our marked Hawkes process is $N(dt, dz) = \int_0^\infty \textbf{1}_{\theta\leq \lambda_s} \Pi(dt,d\theta,dz)$.
Note that since $\Pi$ is a homogeneous Poisson process, the intensity measure is $dtd\theta \text{m}(dz)$.

\begin{lemma}[Doob-Meyer decomposition - Theorem 2.3.7. in \cite{applebaum2009levy}]\label{lemmaDM}
For any $t\in [0,T]$ we have 
    \begin{equation*}
    \int_0^t\int_{\mathbb{R}_+}N(ds,dz) = \int_0^t\int_{\mathbb{R}_+}\int_{\mathbb{R}_+} \textbf{1}_{\theta\leq \lambda_s} \Pi(dt,d\theta,dz) = \bar{N}_t + \int_0^t\int_{\mathbb{R}_+}\int_{\mathbb{R}_+} \textbf{1}_{\theta\leq \lambda_s}dsd\theta \text{m}(dz)
\end{equation*}
where $\bar{N}_t$ is a martingale and $\int_0^t\int_{\mathbb{R}_+}\int_{\mathbb{R}_+} \textbf{1}_{\theta\leq \lambda_s}dsd\theta \text{m}(dz)$ is the compensator of the Hawkes process. Moreover, for any random process $H(z)$ with $z\in \mathbb{R}_+$ such that 
\begin{equation*}
    \mathbb{E}\Bigg[ \int_0^T\int_{\mathbb{R}_+}\int_{\mathbb{R}_+} |H_s(z)| \textbf{1}_{\theta\leq\lambda_s}\ \text{m}(dz)d\theta ds \Bigg] < \infty,
\end{equation*}
the process $M^H$ defined by
\[M^H_t:=  \int_0^t\int_{\mathbb{R}_+} H_s(z) N(ds,dz)-\int_0^t\int_{\mathbb R_+}\int_{\mathbb R_+} H_s(z) \mathbf 1_{\theta\leq \lambda_s} dsd\theta \text{m}(dz), \] is an $(\mathbb P^\Pi,\mathbb F^\Pi)$-martingale. 
\end{lemma}

\subsubsection{Properties of the intensity process}
The analysis relies on the following assumption, enforced throughout.

\paragraph{Assumption $\mathbf{(R)}$.} There exists $\rho\geq 2$ such that  $\varrho^* = \int_{\mathbb{R}_+}|\varrho(z)|^\rho\text{m}(dz) < \infty$ together with the stability condition: $\varrho^* \frac{d_j}{q_j} < 1$ for all $j$. 

\begin{lemma}Under Assumption $\mathbf{(R)}$ there exists $C_T < \infty$ such that $\mathbb{E}\Big[ \int_0^T \big(\lambda_\infty(t) + \sum_{j=1}^n d_j\varsigma_t^{(j)}\big)dt \Big]\leq C_T$. 
\end{lemma}
\begin{proof}
We begin by taking the expectation of each $\varsigma^{(j)}$. Then, 
\begin{equation*}
    \frac{d}{dt}\mathbb{E}\big[\varsigma_t^{(j)}\big] = -q_j\mathbb{E}\big[\varsigma_t^{(j)}\big] + \mathbb{E}\Bigg[ \int_{\mathbb{R}_+}\int_{\mathbb{R}_+} \varrho(z)\textbf{1}_{\theta \leq \lambda_\infty(t) + \sum_{j=1}^n d_j\varsigma_t^{(j)}} \ d\theta \ \text{m}(dz) \Bigg],
\end{equation*}
which we can rewrite as
\begin{equation*}
    \frac{d}{dt}\mathbb{E}\big[\varsigma_t^{(j)}\big] = -q_j\mathbb{E}\big[\varsigma_t^{(j)}\big] + \varrho^*\Bigg( \lambda_\infty(t) + \sum_{j=1}^n d_j\mathbb{E}\big[\varsigma_t^{(j)}\big]\Bigg).
\end{equation*}
Assuming any pre-history contributions to $\varsigma^{(j)}$ are incorporated into the baseline intensity $\lambda_\infty$, we can solve the above linear ODE as:
\begin{equation*}
    \mathbb{E}\big[\varsigma_t^{(j)}\big] = e^{-q_jt}\mathbb{E}\big[\varsigma_0^{(j)}\big] + \varrho^*\int_0^te^{-q_j(t-s)}\Big(\lambda_\infty(s) + \sum_{j=1}^n d_j\mathbb{E}\big[\varsigma^{(j)}_s\big]\Big)ds.
\end{equation*}
From Assumption (\textbf{R}) together with $\mathbb{E}[\lambda_t] = \lambda_\infty(s) + \sum_{j=1}^n d_j\mathbb{E}\big[\varsigma^{(j)}_s\big]$ we get the following convolutional Volterra equation
\begin{equation*}
    \mathbb{E}[\lambda_t] = \lambda_\infty(t) + \int_0^t I(t-s)\mathbb{E}[\lambda_s]ds, \quad I(u) := \varrho^*\sum_{j=1}^n d_je^{-q_ju}.
\end{equation*}
Define the $n$-fold convolution powers $I^{*0}(u) = \delta_0$, $I^{*1}(u) = I(u)$, and $I^{*(n+1)}(u) = (I*I^{*n})(u) = \int_0^t I(t-s)I^{*n}(s)ds$—where $\delta_0$ is the Dirac measure at 0. Define the renewal kernel as $R(t) = \sum_{n=0}^\infty I^{*n}(t)$. Notice that the renewal kernel converges in $L^1$ if $\|I\|_{L^1} < 1$ where $\|\cdot\|_{L^1}$ is the convolution norm.
\begin{equation*}
    \|I\|_{L^1} = \int_0^\infty I(u)du = \varrho^*\sum_{j=1}^nd_j\int_0^\infty e^{-q_ju}du = \varrho^*\sum_{j=1}^n \frac{d_j}{q_j}.
\end{equation*}
Under Assumption \textbf{(R)}, $\|R\|_{L^1} = \frac{1}{1- \varrho^*\sum_{j=1}^n \frac{d_j}{q_j}}$.
Now, by the definition of a convolution from chapter 2 of \cite{gripenberg1990volterra}, we can write 
$$\mathbb{E}[\lambda_t] = (R * \lambda_\infty)(t) = \int_0^t R(t-s)\lambda_\infty(s)ds.$$
Integration over time and Fubini's theorem gives
\begin{align*}
    \int_0^T \mathbb{E}[\lambda_t]dt &= \int_0^T\int_0^t R(t-s)\lambda_\infty(s)dsdt = \int_0^T\lambda_\infty(s)\Bigg( \int_s^T R(t-s)dt\Bigg)ds \\
    &= \int_0^T\lambda_\infty(s)\Bigg(\int_0^{T-s} R(u)ds\Bigg )ds \leq \frac{1}{1- \varrho^*\sum_{j=1}^n \frac{d_j}{q_j}}\int_0^T\lambda_\infty(s)ds.
\end{align*}
Finally, via Assumption \ref{as: L_inft_bound}, we have $\int_0^T \mathbb{E}[\lambda_t]dt \leq \frac{\Lambda_\infty T}{1- \varrho^*\sum_{j=1}^n \frac{d_j}{q_j}}$. \end{proof}

\begin{lemma}\label{lm: lam_mom_bound}
    For any $\ell\geq2$ such that Assumption \textbf{(R)} holds, $\mathbb{E}[\sup_{t\leq T}\lambda_t^\ell] < \infty$, where $\lambda_t$ is defined as in \ref{eq: lambda_t}.
\end{lemma}
\begin{proof}
By applying It\^o's formula to $\lambda^{\ell}$ we get
\begin{multline*}
    d\big(\lambda_t^{\ell}\big) = \ell\lambda_t^{\ell-1}\Bigg( -\sum_{j=1}^n q_j\int_0^t\int_{\mathbb{R}_+}\int_{\mathbb{R}_+}d_je^{-q_j(t-s)}\textbf{1}_{\theta\leq\lambda_s}\varrho(z)\Pi(ds,d\theta,dz)\Bigg)dt \\
    + \int_{\mathbb{R}_+}\int_{\mathbb{R}_+}\Bigg(\lambda_t + \Big(\sum_{j=1}^n d_j\varrho(z)\Big)^{\ell} - \lambda_t^{\ell}\Bigg)\textbf{1}_{\theta\leq\lambda_t}\Pi(dt,d\theta,dz) .
\end{multline*}
Taking expectations, we have
\begin{multline*}
    \frac{d}{dt}\mathbb{E}[\lambda_t^{\ell}] = -\ell\mathbb{E}\Bigg[  \lambda_t^{\ell-1}\sum_{j=1}^n q_j\int_0^t\int_{\mathbb{R}_+}\int_{\mathbb{R}_+}d_je^{-q_j(t-s)}\textbf{1}_{\theta\leq\lambda_s}\varrho(z)\Pi(ds,d\theta,dz) \Bigg] + \ell\mathbb{E}[\lambda_\infty(t)\lambda_t^{\ell - 1}] \\
    +\mathbb{E}\Bigg[ \lambda_t\int_{\mathbb{R}_+} \Bigg(\lambda_t + \Big(\sum_{j=1}^n d_j\varrho(z)\Big)^{\ell} - \lambda_t^{\ell}\Bigg)\text{m}(dz)\Bigg].
\end{multline*}\label{eq:moment_diffeq}

Focusing on the drift term, we notice that
\begin{align*}
    \sum_{j=1}^n q_j\int_0^t\int_{\mathbb{R}_+\times\mathbb{R}_+}d_je^{-q_j(t-s)}\textbf{1}_{\theta\leq\lambda_s}\varrho(z)\Pi(ds,d\theta,dz) &\geq q_{\min}  \int_0^t\int_{\mathbb{R}_+\times\mathbb{R}_+}\sum_{j=1}^nd_je^{-q_j(t-s)}\textbf{1}_{\theta\leq\lambda_s}\varrho(z)\Pi(ds,d\theta,dz) \\
    &= q_{\min} (\lambda_t - \lambda_\infty(t)).
\end{align*}
Therefore,
\begin{align*}
    -\ell\mathbb{E}\Bigg[ \lambda_t^{\ell-1} \sum_{j=1}^n q_j\int_0^t\int_{\mathbb{R}_+}\int_{\mathbb{R}_+}d_je^{-q_j(t-s)}\textbf{1}_{\theta\leq\lambda_s}\varrho(z)\Pi(ds,d\theta,dz) \Bigg] &\leq -\ell q_{\min}\mathbb{E}[\lambda_t^{\ell-1}(\lambda_t - \lambda_\infty(t)] \\
    &= -\ell q_{\min}\mathbb{E}[\lambda^{\ell}] + \ell q_{\min}\mathbb{E}[\lambda_\infty(t)\lambda_t^{\ell-1}].
\end{align*}
Applying Young's inequality gives, for any $\epsilon \in(0,1)$,
$$\ell q_{\min}\mathbb{E}[\lambda_\infty(t)\lambda_t^{\ell -1}] \leq \ell q_{\min}\Lambda_\infty (\epsilon\mathbb{E}[\lambda_t^{\ell}] + C_\epsilon).$$
The drift term is bounded above by $-\ell q_{\min}(1-\tilde{\epsilon})\mathbb{E}[\lambda_t^{\ell}] + C_1$ for some small $\tilde{\epsilon}$ that depends on $\epsilon, \Lambda_\infty$. Now let $\Phi(\lambda) = \int_{\mathbb{R}_+} \Big(\lambda_t + \Big(\sum_{j=1}^n d_j\varrho(z)\Big)^{\ell} - \lambda_t^{\ell}\Big)\text{m}(dz)$. From a binomial expansion and the integrability of the moments of $\sum_{j=1}^n d_j\varrho(z)$ we note that $$\Phi(\lambda) \leq C\sum_{k=0}^\ell \lambda^k \int_{\mathbb{R}_+}\Big(\sum_{j=1}^n d_j\varrho(z)\Big)^{\ell-k}\text{m}(dz).$$
Hence, $\lambda\Phi(\lambda) \leq C\sum_{k=0}^\ell \lambda^{k+1} \int_{\mathbb{R}_+}\Big(\sum_{j=1}^n d_j\varrho(z)\Big)^{\ell-k}\text{m}(dz)$. Then, there exist constants, $C_2\text{ and } C_3$ such that
$$\mathbb{E}[\lambda_t\Phi(\lambda_t)] \leq C_2 + C_3\mathbb{E}[\lambda_t^{\ell}].$$

Combining our two bounds we see $\frac{d}{dt}\mathbb{E}[\lambda_t^{\ell}] \leq A + B\mathbb{E}[\lambda_t^{\ell}]$. From Gronwall's lemma on $[0,T]$, we deduce that $\sup_{0\leq t\leq T} \mathbb{E}[\lambda_t^{m+1}]<\infty$. \end{proof}

\begin{lemma}\label{lemmavarsigma}
    For any $\ell\geq 2$ such that Assumption \textbf{(R)} holds, $\mathbb{E}[\sup_{t\leq T}(\varsigma^{(j)}_t)^{\ell}] < \infty$.
\end{lemma}
\begin{proof}
Let $I_T = \int_0^T\int_{\mathbb{R}_+}\int_{\mathbb{R}_+} \varrho(z)\textbf{1}_{\theta \leq \lambda_s}\Pi(ds, d\theta, dz)$. 
Then $\sup_{t\leq T}\varsigma_t^{(j)} \leq I_T$ for any path $\omega$. So, it suffices to show $\mathbb{E}[I_T^{m+1}] < \infty$.
We can decompose $I_T$ into the the compensated martingale and the compensator as follows:
\begin{equation*}
   I_T = \underbrace{\int_0^T\int_{\mathbb{R}_+}\int_{\mathbb{R}_+} \varrho(z)(\Pi(dt,d\theta,dz) - dtd\theta\text{m}(dz))}_{\mathcal{M}_T} + \underbrace{\varrho^*\int_0^T \Big( \lambda_\infty(t) + \sum_{j=1}^n d_j\varsigma_t^{(j)} \Big)dt}_{\mathcal{A}_T}.
\end{equation*}
Note that $\varsigma^{(j)}_t = \sum_{s\leq t}e^{-q_j(t-s)}\varrho(z_s) \leq \sum_{s\leq T}e^{-q_j(t-s)}\varrho(z_s) =I_T,\; \mathbb P-a.s.$. Hence, $\sup_{t\leq T} \varsigma^{(j)}_t \leq I_T$, so it suffices to show $\mathbb{E}[I_T^{\ell}]<\infty$.
Then from Lemma \ref{lm: lam_mom_bound} we have
$$\mathbb{E}[\mathcal{A}_T^{\ell}] = (\varrho^*)^{\ell}\mathbb{E}\left[ \left( \int_0^T \lambda_\infty(t) + \sum_{j=1}^n d_j\varsigma_t^{(j)} \ dt \right)^{\ell} \right] \leq (\varrho^*)^{\ell}T^{\ell-1}\int_0^T \mathbb{E}[\lambda_t^{\ell}]dt < \infty.$$

Now, by using Theorem 2.11 in \cite{kunita2004stochastic}, we get for some constant $C_\ell>0$
\begin{equation*}
    \mathbb{E}[|\mathcal{M}_T|^{\ell}] \leq C_{\ell}\Bigg( \mathbb{E}\left[ \left(\int_0^T \lambda_sds \cdot \int_{\mathbb{R}_+}\varrho(z)^2\text{m}(dz) \right)^{\ell/2} \right] + \mathbb{E}\left[ \int_0^T \lambda_s ds \cdot \int_{\mathbb{R}_+}\varrho(z)^{\ell+1}\text{m}(dz) \right] \Bigg).
\end{equation*}
From Lemma \ref{lm: lam_mom_bound} and Assumption \textbf{(R)}, the right hand side is finite and consequently for some $C>0$
$$\mathbb{E}\left[\sup_{t\leq T} \big(\varsigma_t^{(j)}\big)^{\ell}\right] \leq \mathbb{E}[I_T^{\ell}] \leq C\Big( \mathbb{E}\big[|\mathcal{M}_T|^{\ell}\big] + \mathbb{E}[\mathcal{A}_T^{\ell}] \Big) < \infty.$$
\end{proof}

\subsection{Controlled SDE and singular control}\label{sec:controlSDE}
We consider the solution to an SDE with drift, volatility and cyber risk cost with singular control and driven by a Hawkes process. Let $B$ and $\Gamma$ be compact subset of $\mathbb R\times \mathbb R$. 
We define the drift, volatility, and Hawkes jump size of our controlled process by 
\[\mu : \mathbb R\longrightarrow \mathbb R, \; \sigma: \mathbb R\times B\longrightarrow \mathbb R,\; \chi :  \mathbb{R}\times{\mathbb{R}_+}\times \Gamma\longrightarrow \mathbb{R},\; \varrho:\mathbb R_+\longrightarrow \mathbb R,
\]
assumed to be continuous functions such that for any $(x,z,b,g)\in \mathbb R\times \mathbb R_+\times B\times \Gamma$
\[ 
\mathbf{(L)}\begin{cases}
  & |\mu(x)-\mu(\tilde x)|+|\sigma(x,b)-\sigma(\tilde x,b)| + |\chi(x, z,g) - \chi( \tilde{x},  z,g) |\leq c|x-\tilde x|,\\
  &  |\mu(x)|+  |\sigma(x,b)| \leq c(1+|x|) \\
&|\chi(x, z,g)|\leq c(1+\tilde \chi(x) +|z|^p),\; \text{ for some }p>1,
\end{cases}
\]
where $\tilde \chi$ is a bounded function from $\mathbb R$ into $\mathbb R$.

\begin{definition}[Admissible controls]
    An admissible control is defined as a pair of progressively measurable processes $\nu:=(\beta,\gamma)$ with values in $B\times \Gamma$. We denote by $\mathcal V$ the set of admissible controls $\nu_t=(\beta_t, \gamma_t)$.
\end{definition}

We now introduce the SDE system without singular control, see \cite{bensoussan2024stochastic,khabou2025markov}. 
\begin{equation}\label{eq:non-sing_dyn}
    \begin{cases}
        d Y_t =  \mu( Y_t)dt+ \sigma( Y_t,\beta_t)dW_t + \int_{\mathbb{R}_+}\int_{\mathbb{R}_+}\chi( Y_t, z,\gamma_t)\textbf{1}_{\theta\leq \lambda_\infty(t) + \sum_{j=1}^n d_j\varsigma^{(j)}_t}\Pi(dt,d\theta, dz), \\
    d\varsigma^{(1)}_t = -q_1\varsigma^{(1)}_tdt + \int_{\mathbb{R}_+}\int_{\mathbb{R}_+} \varrho(z)\textbf{1}_{\theta\leq \lambda_\infty(t) + \sum_{j=1}^n d\varsigma^{(1)}_t} \Pi(dt,d\theta,dz),\\
    \vdots \\
    d\varsigma^{(n)}_t = -q_n\varsigma^{(n)}_tdt + \int_{\mathbb{R}_+}\int_{\mathbb{R}_+} \varrho(z)\textbf{1}_{\theta\leq \lambda_\infty(t) + \sum_{j=1}^n d\varsigma^{(n)}_t}\Pi(dt,d\theta,dz),\\
       N(dt,dz)=\int_0^\infty \mathbf 1_{\theta\leq \lambda_\infty(t) + \sum_{j=1}^n d_j\varsigma^{(j)}_t} \Pi(dt,d\theta,dz).
    \end{cases}
\end{equation}

\begin{remark}[Exponential Kernel dimension 1]
    Assume that $n=1$ so that $\varphi(t)=d e^{-qt}$. The SDE system is given by

    \begin{equation*}
    \begin{cases}
        dY_t = \mu( Y_t)dt + \int_0^t\sigma( Y_t,\beta_t)dW_t + \int_{\mathbb{R}_+}\int_{\mathbb{R}_+}\chi( Y_t, z,\gamma_t)\textbf{1}_{\theta\leq \lambda_t}\Pi(dt,d\theta, dz), \\
       \lambda_t = \lambda_\infty(t) + \int_0^{t-}\int_{\mathbb{R}_+}\int_{\mathbb{R}_+}  d e^{-q(t-s)}\varrho(z) \textbf{1}_{\theta \leq \lambda_s}\Pi(ds, d\theta, dz),\\
       N(dt,dz)=\int_0^\infty \mathbf 1_{\theta\leq \lambda_t} \Pi(dt,d\theta,dz).
    \end{cases}
\end{equation*}

\end{remark}
\begin{definition}[Solution to \eqref{eq:non-sing_dyn}]
    A solution to \eqref{eq:non-sing_dyn} is a tuple $(Y,\varsigma_1,\dots,\varsigma_n,N)$ such that $(Y,\varsigma_1,\dots,\varsigma_n,N)$ satisfies the SDE system \eqref{eq:non-sing_dyn} and $N$ is a marked Hawkes process with compensator $dsd\theta\emph{m}(dz)$.
\end{definition}

\noindent We now extend the definition of this SDE with a singular control. 

\begin{definition}[Singular control]
    Let $\Xi$ be the set of nonnegative, nondecreasing, and left-continuous right-limited functions. Any $\xi\in\mathbb{R}_+$ is called a singular control if it is progressively measurable with respect to $\mathcal{F}$ and its sample paths are in $\Xi$.
\end{definition}

The singular controlled SDE becomes:
\begin{equation}\label{SDE:singular}
    \begin{cases}
        &dX_t =\mu( X_t)dt + \sigma( X_t,\beta_t)dW_t + d\xi_t+\int_{\mathbb{R}_+}\int_{\mathbb{R}_+}\chi( X_t, z,\gamma_t)\textbf{1}_{\theta\leq \lambda_\infty(t) + d^\top \bar\varsigma_t}\Pi(dt,d\theta, dz) \\
    &d\varsigma^{(1)}_t = -q_1\varsigma^{(1)}_tdt + \int_{\mathbb{R}_+}\int_{\mathbb{R}_+} \varrho(z)\textbf{1}_{\theta\leq \lambda_\infty(t) + \sum_{j=1}^n d_j\varsigma^{(j)}_t} \Pi(dt, d\theta, dz)\\
   & \vdots \\
   & d\varsigma^{(n)}_t = -q_n\varsigma^{(n)}_tdt + \int_{\mathbb{R}_+}\int_{\mathbb{R}_+} \varrho(z)\textbf{1}_{\theta\leq \lambda_\infty(t) + \sum_{j=1}^n d_j\varsigma^{(j)}_t}\Pi(dt, d\theta, dz),\\
       &N(dt,dz)=\int_0^\infty \mathbf 1_{\theta\leq \lambda_\infty(t) + \sum_{j=1}^n d_j\varsigma^{(j)}_t} \Pi(dt,d\theta,dz),
    \end{cases}
\end{equation}
where set $\overline \varsigma:=(\varsigma^{(1)},\dots,\varsigma^{(n)})$ and $\bar{d} = (d_1,\dots,d_n)$, so that $\sum_{j=1}^nd_j\varsigma_t^{(j)} = d^\top \bar{\varsigma}_t$. 
\begin{lemma}\label{lm: X_moments}
Let Assumption \textbf{(R)} and Assumption \textbf{(L)} be satisfied. For any $\ell\geq 1$, we have   $\sup_{t\leq T} \mathbb{E}|X_t|^{\ell} < \infty$.
\end{lemma}
\begin{proof}
This proof extends \cite[Theorem 2.2]{khabou2025markov} from the setting of square integrability to the broader framework of singular controls and arbitrary integrability orders. Applying It\^o's formula for jump-diffusions to $|X_t|^\ell$ we get
\begin{equation*}
    d|X_t|^{\ell} = \ell X_{t-}|X_{t-}|^{\ell-2}dX_t + \frac{1}{2}\ell(\ell-1)|X_{t-}|^{\ell-2}d\langle X^c, X^c\rangle_t + \sum_{\Delta X_t>0} (|X_t|^\ell - |X_{t-}|^\ell - \ell X_{t-}|X_{t-}|^{\ell-1} \Delta X_t)
\end{equation*}
From Assumption \textbf{(L)}, note that
\begin{equation*}
    \ell X_{t-}|X_{t-}|^{\ell-2} \mu( X_t) \leq \ell|X_t|^{\ell-1}|\mu( X_t)| \leq \ell|X_t|^{\ell-1} c(1 + |X_t|).
\end{equation*}

Hence,
\begin{equation*}
    \ell|X_t|^{\ell-1}|\mu( X_t)|  \leq c (1 + |X_t|^\ell ).
\end{equation*}
Now note that
$$\frac{\ell(\ell-1)}{2}|X_t|^{\ell-2}|\sigma(X_t,\beta_t)|^2 \leq C(1 + |X_t| )^2 \leq C\big(1 + |X_t|^\ell \big).$$
Therefore,
\begin{align*}
    |X_t|^\ell - |X_{t-}|^\ell &\leq \ell|X_{t-}|^{\ell-1}\Bigg|\int_{\mathbb{R}_+} \chi( X_t,  z,\gamma_t)dz \Bigg| + C\Bigg|\int_{\mathbb{R}_+} \chi( X_t,  z,\gamma_t)\text{m}(dz) \Bigg|^\ell \\
    &\leq C \Bigg(1 + |X_{t-}|^\ell + \Bigg|\int_{\mathbb{R}_+} \chi( X_t,  z,\gamma_t)\text{m}(dz) \Bigg|^\ell\Bigg).
\end{align*}

Consequently $\mathbb{E}\big[|X_t|^\ell\big] \leq \mathbb{E}\big[|X_0|^\ell\big] + C\int_0^t \mathbb{E}\big[1 + |X_s|^\ell \big]ds < \infty.$
By applying Gronwall Lemma, we deduce that $\sup_{t\leq T}\mathbb{E}[|X_t|^\ell]<\infty$.
\end{proof}

\subsection{Stochastic optimal mixed stopping-singular control problem}\label{sec:controlcontinuous}

Motivated by project management and investment problems, we assume that if the value of the project hits a boundary given by a function of its current value, $F$, a firm decides to stop the project; that is, the project stops at time $\tau$ if the value falls below $F(X_\tau)$. Otherwise, the project continues until maturity $T>0$ and the firm receives a value of $G(X_T)$.
We assume that the agent may inject additional monetary support, $\xi$, throughout time at their discretion, incurring cost $\phi$ each time. 
Further, the agent is subject to a running $K$ related to the controls $\beta \text{ and }\gamma$. 
Finally, there is a cost $\kappa$ incurred every time the Hawkes process jumps. 
Specifically, we define the cost functions as follows:
$$F \ : \ \mathbb{R} \to \mathbb{R}, \quad G \ : \ \mathbb{R} \to \mathbb{R}, \quad \phi \ : [0, T] \to \mathbb{R}_+, \quad K \ : \ [0,T]\times\mathbb{R}\times B\times \Gamma\to\mathbb{R}, \quad \kappa \ : \ [0,T]\times\mathbb{R}\times\mathbb{R}_+\to\mathbb{R}$$
such that for any $(t,x,b,g,z)\in[0,T]\times\mathbb{R}\times B\times\Gamma\times\mathbb{R}_+$
\begin{align*}
    \textbf{(Polx)} \ : \quad &F,G\in \mathbb{R}[x], \\
    \textbf{(Poly)} \ : \quad & |K(t,x,b,g)|+|\kappa(t,x,z)|\leq C(1+|x|^p), \ \ p\geq 1, \\
    \textbf{(Coef)} \ : \quad & \forall\mathfrak{K}>0 \sup_{(t,x)\in[0,T]\times[-\mathfrak{K},\mathfrak{K}]} \int_0^\infty |\kappa(t,x,z)|\text{m}(dz) < \infty,
\end{align*}
where $\mathbb R[x]$ denotes the space of real valued polynomials. The optimization problem is 
\begin{equation}\label{controlpb-continuous}
    V_0(x, \overline\varsigma):=\sup_{(\nu;\tau,\xi)\in  \mathcal V\times \mathcal{T}_{[0,T]}\times \Xi} J_0(x, \overline \varsigma;\nu,\tau,\xi),
\end{equation}
where for any $(x, \bar{\varsigma};\nu,\vartheta)\in \mathbb R\times \mathbb R^n\times \mathcal V\times \mathcal T_{[0,T]}\times \Xi$ 
\begin{align*}
    J_0(x, \bar{\varsigma};\nu,\vartheta)&:=
    \mathbb{E}\Big[ e^{-r\tau} F(X_\tau)\mathbf{1}_{\tau<T} 
    +e^{-rT}G(X_T)\mathbf{1}_{\tau\geq T} \\
    &\qquad-\int_0^{\tau\wedge T} e^{-rs}K(s,X_s,\beta_s,\gamma_s)ds -\int_0^{\tau\wedge T}e^{-rs} \phi(s) d\xi_s\\
    &\qquad- \int_0^{\tau\wedge T}\int_{\mathbb{R}_+} e^{-rs}\kappa(s, X_s, z)\mathbf{1}_{\theta\leq \lambda_\infty(t) + \bar{d}\bar{\varsigma}^\top_s}\Pi(ds,d\theta,dz)
    |X_0=x,\bar{\varsigma}_0=\bar{\varsigma}\Big],
\end{align*}
with $\vartheta=(\tau,\xi)\in\mathcal{T}_{[0,T]}\times \Xi$.
\paragraph{Interpretations.} The first and second terms represent the profit made with or without early stopping of the project, respectively. The third term represents the continuous time cost generated by the choice of controls ($\beta,\gamma$). The final two terms represent the costs related to singular actions of the agent and accidents in the project value (as soon as $\chi$ is for example a negative function), respectively.\newline

Using H\"older Inequalities together with Lemmas \ref{lm: lam_mom_bound} and \ref{lm: X_moments}, note that \[\mathbb{E}\Bigg[\int_0^T\int_0^\infty\int_{\mathbb{R}_+} |\kappa(s, X_s, z)|\textbf{1}_{\theta\leq\lambda_\infty(s) + \bar{d}\bar{\varsigma}^\top_s}\ \text{m}(dz)d\theta ds \Bigg] < \infty.\] Therefore, the objective function $J_0$ can be simplified as follows

\begin{align*}
    J_0(x, \bar{\varsigma};\nu,\vartheta)&:=
    \mathbb{E}\Big[ e^{-r\tau} F(X_\tau)\mathbf 1_{\tau<T} 
    +e^{-rT}G(X_T)\mathbf{1}_{\tau\geq T}  \\
    &\qquad -\int_0^{\tau\wedge T}e^{-rs} \phi(s) d\xi_s-\int_0^{\tau\wedge T} e^{-rs}\mathcal{K}(s,X_s, \bar{\varsigma}_s,\text{m},\nu_s)ds \ | \ X_0=x,\bar{\varsigma}_0=\bar{\varsigma}\Big],
\end{align*}
where $\mathcal{K}(s, X_s, \bar{\varsigma}_s, \text{m}, \nu_s) =  \int_{\mathbb{R}_+} \kappa(s, X_s, z)(\lambda_\infty(s) + \bar{d}\bar{\varsigma}^\top_s) \ \text{m}(dz)- K(s,X_s, \nu_s) $.

\begin{definition}[Admissible control]
    A control $\nu=(\beta,\gamma)\in \mathcal V$ and $\vartheta:=(\tau,\xi)\in \mathcal T_{[0,T]}\times \Xi$ is said admissible if 
    \[\mathbb E\Big[ \int_0^{T\wedge \tau} e^{-rs}[\phi(s)d\xi_s +|\mathcal{K}(s,X_s,\bar{\varsigma}_s,\emph{m},\nu_s)|ds]\Big]<\infty.\]
\end{definition}
We introduce the following variational inequality with gradient constraint:
\begin{equation*}\label{VPDE}
\textbf{(VPDE)}\begin{cases}
        \min\left\{ \partial_tv- rv + \sup_{b,g\in B\times \Gamma}\mathcal L^{b,g}[v];\partial_xv- \phi(t); v - F(x)\right\}=0,\; t<T \\
    v(T,x, \bar{\varsigma}) = G(x),\; (x,\bar{\varsigma})\in \mathbb R\times \mathbb R^n_+.
\end{cases}\end{equation*}
where 
\begin{align*}
    \mathcal{L}^{b,g}[v](t, x, \bar{\varsigma},\text{m})&:=\mu(x)\partial_xv(t,x,\bar{\varsigma}) + \tfrac{1}{2}|\sigma(x,b)|^2\partial_{xx}v(t,x,\bar{\varsigma}) - \mathcal{K}(t,x,\bar{\varsigma}, \text{m},b,g) \\
    &- \sum_{k=1}^n q_k\varsigma^k\partial_{\varsigma^k}v(t,x,\bar{\varsigma})\\
    &+ (\lambda_\infty(t) + \bar{d}\bar{\varsigma}^\top) \int_{\mathbb{R}_+} [v(t, x + \chi(x, z,g), \bar{\varsigma}+\varrho(z)\textbf{e}_n) - v(t, x, \bar{\varsigma})] \ \text{m}(dz)
\end{align*}
and $\textbf{e}_n = (1,\cdots,1)\in\mathbb{R}^n$.

\begin{remark}[Exponential Kernel dimension 1]
    Assume that $n=1$. Then \textbf{(VPDE)} becomes
$$(VPDE)_1\begin{cases}
        \min\left\{ \partial_tv - rv + \sup_{b,g}\mathcal L^{b,g}[v]; \phi(t) - \partial_xv; v - F(x)\right\}=0 \\
    v(T,x,\lambda) = G(x),\; (x,\lambda)\in \mathbb R\times \mathbb R_+.
\end{cases}$$
\begin{multline*}
    \mathcal L^{b,g}[v](t,x,\lambda, \emph{m}):=  \tfrac{1}{2}|\sigma(x,b)|^2\partial_{xx}v(t,x,\lambda) +\mu(x)\partial_xv(t,x,\lambda) - \mathcal{K}(t,x,\lambda, \emph{m},b,g)+ q\partial_\lambda v(\lambda_\infty - \lambda) \\
    + \lambda \int_{\mathbb{R}_+} [v(t-, X_{t-} + \chi( x,  z,g), \lambda+\varrho) - v(t, x, \lambda)] \ \emph{m}(dz). \\
\end{multline*}
\end{remark}

\begin{remark}
    An example of such functions $F,G,\mathcal K$ is given in the application Section \ref{secnumeric}.
\end{remark}

We introduce the dynamic version of the objective and value function below. 
\begin{multline*}
    J_t(x, \bar{\varsigma};\nu,\vartheta):=
    \mathbb E\Big[ e^{-r(\tau-t)} F(X^{t,x}_\tau)\mathbf 1_{\tau<T} 
    +e^{-r(T-t)}G(X_T)\mathbf 1_{\tau\geq T} 
    -\int_t^{\tau\wedge T}e^{-r(s-t)} \phi(s) d\xi_s \\
    -\int_t^{\tau\wedge T} e^{-r(s-t)}\mathcal{K}(s,X_s,\bar{\varsigma}_s, \text{m},\nu_s)ds 
    |X_t=x,\bar{\varsigma}_t=\bar{\varsigma}\Big]
\end{multline*}
\begin{equation*}
    V_t(x, \bar{\varsigma}):=\sup_{(\nu,\vartheta)\in\mathcal V\times   \mathcal{T}_{[t,T]}\times \Xi} J_t(x, \bar{\varsigma};\nu,\vartheta),
\end{equation*}

Extending \cite[Theorem 2.1]{an2010combined} to mixed stopping singular control with a Hawkes process, we have the following result.

\begin{theorem}[Supersolution]\label{thm:cont_supersolution}Assume that there exists a regular function $v\in \mathcal C^{1,2,1}([0,T]\times \mathbb R\times\mathbb{R}_+)$ once differentiable in time, twice continuously differentiable in space, and once differentiable in the intensity such that 
\begin{enumerate}
\item there exist $C< \infty$ and exponents $\ell, \ell', \mathfrak{r}, \mathfrak{r'}>0$ such that for all $t,x,\bar{\varsigma}$:
\begin{align*}
    \textbf{(Polv)} \ : \quad &|v(t,x,\bar{\varsigma})| \leq C\big(1 + |x|^\ell + \big\|\bar{\varsigma}^\top\big\|^\mathfrak{r}\big) \\
    \textbf{(Dvx)} \ : \quad &|\partial_x v(t,x,\bar{\varsigma})| \leq C\big(1 + |x|^{\ell-1} + \big\|\bar{\varsigma}^\top\big\|^\mathfrak{r}\big) \\
    \textbf{(Dvs)} \ : \quad &|\partial_{\varsigma^{(j)}}v(t,x,\bar{\varsigma})| \leq C|d_j|\Big(1 + |x|^{\ell'} + \big\|\bar{\varsigma}^\top\big\|^{\mathfrak{r}'}\Big) \quad \forall j\in \{1,\dots,n\},
\end{align*}

    \item $v(t,x, \bar{\varsigma}) \geq F(x)$ for all $(t, x, \bar{\varsigma})$
    \item $ \phi(t)\geq \partial_xv(t,x,\bar{\varsigma}) $ for all $(t,x,\bar{\varsigma})$
    \item $\partial_tv(t,x,\bar{\varsigma}) -rv(t,x,\bar{\varsigma}) + \sup_{b,g\in B\times \Gamma}\mathcal L^{b,g}[v](t,x,\bar{\varsigma},\emph{m}) \leq 0 $ for all $(t,x,\bar{\varsigma})\in [0,T]\times\mathbb R\times \mathbb R^n_+$,
    \item $v(T,x, \bar{\varsigma}) = G(x),$
\end{enumerate}
   then $v(t,x, \bar{\varsigma})\geq V_t(x, \bar{\varsigma})$ for any $(t,x,\bar{\varsigma})\in [0,T]\times \mathbb R\times \mathbb{R}_+^n$.
\end{theorem}

\begin{remark}
    The first assumption is set to ensure that any local martingale is a true martingale while assumptions 2 to 5 are used for the supersolution property. 
\end{remark}

\begin{proof}[Proof of Theorem \ref{thm:cont_supersolution}]
Fix an admissible control $(\beta,\gamma;\tau, \xi)\in \mathcal V\times \mathcal T_{[0,T]}\times \Xi$. By Applying It\^o's formula to $e^{-r(s-t)}v(s, X_s, \bar{\varsigma}_s)$ we get

\begin{align*}
    &v(t,x,\bar{\varsigma})\\
    &= \mathbb{E}[e^{-r(\tau\wedge T - t)}v(\tau\wedge T, X_{\tau\wedge T}, \bar{\varsigma}_{\tau\wedge T})]\\
    &- \mathbb{E}\Bigg[ \int_t^{\tau\wedge T}e^{-r(s-t)}\big(\partial_sv(s, X_s, \bar{\varsigma}_s)  - rv(s, X_s, \bar{\varsigma}_s) + \mathcal L^{\beta,\gamma}[v](s,X_s,\bar{\varsigma}_s\big)ds\Bigg] \\
    &-\mathbb E\Bigg[\int_t^{\tau\wedge T} e^{-r(s-t)}\mathcal{K}(s,X_s,\bar{\varsigma}_s, \text{m},\beta_s,\gamma_s)ds  \Bigg]\\
    & -\mathbb{E}\left[ \int_t^{\tau\wedge T} e^{-r(s-t)}\partial_xv(s, X_s, \bar{\varsigma}_s)\sigma(X_s,\beta_s) dW_s \right] - \mathbb{E}\left[\int_t^{\tau\wedge T}e^{-r(s-t)}\partial_xv(s, X_s, \bar{\varsigma}_s) d\xi_s^{(c)} \right] \\
    &  - \mathbb{E}\left[ \sum_{t\leq s\leq \tau\wedge T}e^{-r(s-t)}\big(v(s, X_{s-} + \Delta \xi_s, \bar{\varsigma}_{s-}) - v(s, X_{s-}, \bar{\varsigma}_{s-})\big) \right] \\
    &-\mathbb{E}\left[ \int_t^{\tau\wedge T} \int_{\mathbb{R}_+} e^{-r(s-t)}\underbrace{[v(s, X_{s-} + \chi(X_s, z,\gamma_s), \bar{\varsigma}_{s-} + \varrho(z)\textbf{e}_n) - v(s, X_{s-}, \bar{\varsigma}_{s-})]}_{H}\bar{N}(ds,dz) \right],
\end{align*}
where $\bar{N}$ is defined from Lemma \ref{lemmaDM}. From the the mean value theorem, there exists a process $\zeta$ with values between $X_{s-}+\Delta \xi_s$ and $X_{s-}$ such that
$$v(s, X_{s-} + \Delta\xi_s, \bar{\varsigma}_{s-}) - v(s, X_{s-}, \bar{\varsigma}_{s-}) = \Delta\xi_s \cdot \partial_xv(s, \zeta_s, \bar{\varsigma}_s).$$
Consequently the fifth line of the previous equality becomes the expectation of $\sum_{t\leq s\leq \tau\wedge T} e^{-r(s-t)}\Delta\xi_s \cdot \partial_xv(s, \zeta_s, \bar{\varsigma}_s)$. Next, we check the integrability of the last term $H$. Since $v$ is polynomial in $X$ and $\bar{\varsigma}$, by the first assumption, we deduce from Lemmas \ref{lm: lam_mom_bound}, \ref{lemmavarsigma}, and \ref{lm: X_moments} that $\int_{\mathbb{R}_+}|H| \text{m}(dz) <\infty$.
Additionally, by using \textbf{(Polv)} and \textbf{(L)} together with Lemma \ref{lm: X_moments}, we deduce that $Z$ defined by $Z_t:=\sigma(X_t,\beta_t) v(t,X_t,\bar\zeta_t)$ is in $\mathbb H^2$. Therefore
\begin{align*}
    v(t,x,\bar{\varsigma}) &= \mathbb{E}[e^{-r(\tau\wedge T - t)}v(\tau\wedge T, X_{\tau\wedge T}, \bar{\varsigma}_{\tau\wedge T})]\\
    &- \mathbb{E}\Bigg[ \int_t^{\tau\wedge T}e^{-r(s-t)}\big(\partial_sv(s, X_s, \bar{\varsigma}_s)  - rv(s, X_s, \bar{\varsigma}_s) + \mathcal L^{\beta,\gamma}[v](s,X_s,\bar{\varsigma}_s\big)ds\Bigg] \\
    &- \mathbb{E}\Bigg[\int_t^{\tau\wedge T} e^{-r(s-t)}\mathcal{K}(s,X_s,\bar{\varsigma}_s,\text{m},\nu_s)ds\Bigg]\\ 
    &  - \mathbb{E}\left[\int_t^{\tau\wedge T}e^{-r(s-t)}\partial_xv(s, X_s, \bar{\varsigma}_s) d\xi_s^{(c)} \right] \\
    &  - \mathbb{E}\left[ \sum_{t\leq s\leq \tau\wedge T}e^{-r(s-t)}\big(v(s, X_{s-} + \Delta \xi_s, \bar{\varsigma}_{s-}) - v(s, X_{s-}, \bar{\varsigma}_{s-})\big) \right].
\end{align*}

From assumptions 2 through 5 we then deduce that
$v(t,x, \bar{\varsigma})\geq \sup_{\beta,\gamma;\tau,\xi}J_t(x, \bar{\varsigma}; \beta,\gamma,\tau, \xi) = V_t(x, \bar{\varsigma})$.
\end{proof}

We now turn to the verification result. 

\begin{theorem}\label{thm:verif}
    Assume all the conditions of Theorem \ref{thm:cont_supersolution} are satisfied. Let $D = \{ (t,x,\bar{\varsigma}) \ : \ v(t,x) > F(x) \}$ be a continuation region. Suppose that there exists $\hat\nu=(\hat\beta,\hat\gamma)\in \mathcal V$ and $ \hat{\vartheta}=(\hat\tau,\hat\xi)$
    such that
    \begin{enumerate}[start=5]
        \item $\partial_tv(t,x,\bar{\varsigma}) -rv(t,x,\bar{\varsigma}) -\sum_{j=1}^nq_j\varsigma^{(j)}\partial_{\varsigma^{(j)}}v(t,x,\bar{\varsigma}) + \mathcal L^{\hat\nu}[v](t,x,\bar{\varsigma},\emph{m}) = 0$ for all $(t,x)\in D$,
        \item $(  \phi(t) - \partial_xv(t,x,\bar{\varsigma}))d\hat{\xi}^{(c)}_t = 0$ for all $t$, where $\hat{\xi}^{(c)}$ is the continuous part of $\hat{\xi}$,
        \item $\Delta_{\hat{\xi}}v(t_k, X_{t_k},\bar{\varsigma}_{t_k}) = \phi(t_k)\cdot\Delta\hat{\xi}(t_k)$ for all $t_k$ jumping times of $\hat{\xi}$,
        \item $v(T, x,\bar{\varsigma}) = G(x)$,
        \item $\hat{\tau} = \inf\{ t>0 : (t,x,\bar{\varsigma})\notin D \} < \infty$ a.s. for all $x\in \mathbb R, \ \bar{\varsigma}\in \mathbb R^n_+$,
        \item $\{ v(\tau, x,\bar{\varsigma}) \ : \ \tau\in\mathcal{T}, \tau\leq\hat{\tau} \}$ is uniformly integrable for all $x\in\mathbb R, \  \bar{\varsigma}\in \mathbb R^n_+$ .
    \end{enumerate}
    Then $V(s,x,\bar{\varsigma}) = v(s,x,\bar{\varsigma})$ for all $(s,x,\bar{\varsigma})\in[0,T]\times\mathbb{R}\times\mathbb{R}^n_+$ and $(\hat\nu,\hat{\vartheta})$ ar ethe optimal volatility, jumps, stopping and singular control strategies.
\end{theorem}

\begin{proof}
Let $D$ be a continuation region as defined in the theorem statement above and let assumptions 1-10. be satisfied.
Fix the controls $(\hat{\nu}, \hat{\xi})$ and $\hat{\tau}$ as defined in the theorem and assumption 9.
We recall from the previous proof and from It\^o's formula together with assumptions made on $v$ in Assumption 1 
\begin{align*}
    v(t,x, \bar{\varsigma}) &= \mathbb{E}\left[e^{-r(\hat{\tau} - t)}F(X_{\hat\tau})\textbf{1}_{\hat{\tau}<T} + e^{-r(T-t)}G(X_T)\textbf{1}_{\hat{\tau}=T} - \int_t^{\hat{\tau}\wedge T}e^{-r(s-t)}\mathcal{K}(s, X_s, \bar{\varsigma}_s, \text{m},\hat\nu_s)ds \right . \\
    & \quad \left . - \int_t^{\hat{\tau}\wedge T}e^{-r(s-t)}\phi(s) d\hat{\xi}_s^{(c)} - \sum_{t\leq s_k\leq \hat{\tau}\wedge T}e^{-r(s-t)}\phi(s_k)\Delta\hat{\xi}(s_k) \right]\\
    &= \mathbb{E}\left[e^{-r(\hat{\tau} - t)}F(X_{\hat\tau})\textbf{1}_{\hat{\tau}<T} + e^{-r(T-t)}G(X_T)\textbf{1}_{\hat{\tau}=T} - \int_t^{\hat{\tau}\wedge T}e^{-r(s-t)}\mathcal{K}(s, X_s, \bar{\varsigma}_s,\text{m},\hat{\nu}_s)ds \right . \\
    & \quad \left . - \int_t^{\hat{\tau}\wedge T}e^{-r(s-t)}\phi(s) d\hat{\xi}_s \right].
\end{align*}

Let $n\in\mathbb{R}_+$ and let $\Theta_n = \inf\{ t \ : \ (X_t,\bar\varsigma_t) \in [-n,n]^{1+n} \}$; that is, $\Theta_n$ is a stopping time in a localization method.
We define
\begin{align*}
    v_n(t,x, \bar{\varsigma}) &= \mathbb{E}\Big[e^{-r(\hat{\tau} - t)}F(X_{\hat\tau})\textbf{1}_{\hat{\tau}<T\wedge\Theta_n} + e^{-r(T\wedge\Theta_n-t)}G(X_T)\textbf{1}_{\hat{\tau}=T\wedge\Theta_n}\\
    &\qquad- \int_t^{\hat{\tau}\wedge T\wedge\Theta_n}e^{-r(s-t)}\mathcal{K}(s, X_s, \bar{\varsigma}_s,  \text{m},\hat\nu_s)ds + \int_t^{\hat{\tau}\wedge T\wedge\Theta_n}e^{-r(s-t)}\phi(s) d\hat{\xi}_s \Big]
\end{align*}
As we take $n\to\infty$, then $\Theta_n\to\infty$ as well, meaning $\tau\wedge T\wedge\Theta_n \to \tau\wedge T$ and $T\wedge\Theta_n \to T$. From Assumptions \textbf{(Polx)}, \textbf{(Poly)}, and \textbf{(Coeff)} together with the Lemmas \ref{lm: lam_mom_bound}, \ref{lemmavarsigma}, and \ref{lm: X_moments} and the dominated convergence theorem, we deduce that $\lim_{n\to\infty} v_n(t,x,\bar{\varsigma}) = J_t(x, \bar{\varsigma};  \hat{\vartheta})\leq V_t(x,\bar{\varsigma})$. By using Theorem \ref{thm:cont_supersolution}, we then deduce that $V_t(x, \bar{\varsigma}) = v(t,x, \bar{\varsigma})$ with optimizers $\hat\nu$ and $\hat{\vartheta}$.
\end{proof}
\subsection{Restriction to a bounded domain}\label{sec:bounded}

We now define a constrained state space for later use in our Markov decision process.
In order to constrain the state space, we must introduce a term to the original dynamics that reflects the process back into the interior when the process hits either the upper or lower boundary value.
That is, we want the corresponding two-sided reflection map.
Let $L\in[0, \infty)$.
From \cite{burdzy2009skorokhod} and definition 1.2 of \cite{kruk2007explicit}, the Skorokhod map on $[-L, L]$ with $L\geq 0$ fixed, can be defined as follows.

\begin{definition}[Skorokhod problem]\label{def:2side_Skorokhod_map}

    Given $\Psi \in \mathcal{D}[0, \infty)$ there exists a unique pair of functions $(\Psi, \psi,\tilde \psi ) \in \mathcal{D}[0, \infty) \times \mathcal{BV}[0, \infty)$ that satisfy the following properties:
    \begin{enumerate}
        \item For every $t\in[0, \infty)$, $\Psi(t) = \tilde\psi(t) + \psi(t) \in [-L, L]$.
        \item $\psi$  decomposes as $\psi = \psi_\ell - \psi_u$ where $\psi_\ell, \psi_u \in \mathcal{I}[0,\infty)$ and satisfy
            \begin{equation*}
                \int_0^\infty \textbf{1}_{ \psi(s) > -L }d\psi_\ell(s) = 0 
                \quad \text{and} \quad
                \int_0^\infty \textbf{1}_{ \psi(s) < L }d\psi_u(s) = 0 .
            \end{equation*}
    \end{enumerate}
\end{definition}

We define the reflection processes $R^+_t$, $R^-_t$, and $\mathcal R^i_t,\; 1\leq i\leq n$ as componentwise nondecreasing, c\`adl\`ag, $\{\mathcal{F}_t\}$-adapted processes that satisfy
\begin{equation*}
    \int_0^{\tau\wedge T} \textbf{1}_{X^L_t > -L }dR^+_t = 0 \quad \quad \int_0^{\tau\wedge T}\textbf{1}_{ X^L_t < L }dR^-_t = 0
\end{equation*}
$$\int_0^{\tau\wedge T}\textbf{1}_{ \varsigma^{(i),L}_t < L }d\mathcal{R}^{(i)}_t = 0 \quad \forall i\in\{1,\dots,n\},$$
where $(X^L, \varsigma^{(i),L})$ satisfy the following for any control $\nu=(\beta,\gamma)\in \mathcal V$

\begin{equation}\label{eq:constrained_dyn}
    \begin{cases}
        dX^L_t = \Big[\mu( X^L_t) + \int_{\mathbb{R}_+}\int_{\mathbb{R}_+}\chi(t, X^L_t, z,\gamma_t)\textbf{1}_{\theta\leq\lambda_\infty(t)+\bar{d}(\bar{\varsigma}^L_t)^\top}d\theta\text{m}(dz)\Big]dt \\
        \quad\quad\quad\quad \ + \sigma( X^L_t ,\beta_t)dW(t)+d\xi(t)+ dR^+_t - dR^-_t \\
        d\varsigma^{(1),L}_t = -q_1\varsigma^{(1),L}_t dt + \int_{\mathbb{R}_+}\int_{\mathbb{R}_+} \varrho(z)\textbf{1}_{\theta\leq \lambda_\infty(t) + \bar{d}(\bar{\varsigma}^L_t)^\top} \Pi(dt, d\theta, dz) - d\mathcal{R}^{(1)}_t\\
         \vdots \\
         d\varsigma^{(n),L}_t = -q_n\varsigma^{(n),L}_tdt + \int_{\mathbb{R}_+}\int_{\mathbb{R}_+} \varrho(z)\textbf{1}_{\theta\leq \lambda_\infty(t) + \bar{d}(\bar{\varsigma}^L_t)^\top}\Pi(dt, d\theta, dz) - d\mathcal{R}^{(n)}_t,\\
       N(dt,dz)=\int_0^\infty \mathbf 1_{\theta\leq \lambda_\infty(t) + \bar{d}(\bar{\varsigma}^L_t)^\top} \Pi(dt,d\theta,dz)
    \end{cases}
\end{equation}
We note that in the case of the singular control, $\xi$, and Hawkes process, $N$, the bounded domain affects the processes as follows. If a jump is such that the end value lies outside the domain (i.e., a jump to $y > L$), then the jump triggers the reflection process and reflects back into the domain.
Now, we define $\tau_L = \inf \{ t\geq 0 \ : \ (X_t\notin [-L, L])\cup(\bar{\varsigma}_t\notin[0,L]^n) \}$. We define $J_{0,L}(x, \bar{\varsigma}; \nu, \vartheta)$ and $V_{0,L}(x, \bar{\varsigma})$ as follows:
\begin{multline*}\label{eq:Jl}
    J_{0,L}(x, \bar{\varsigma}; \nu,\vartheta) = \mathbb{E} \left[ 
    e^{-r\tau}F(X^L_\tau) \textbf{1}_{\{\tau < T \} \cap \{ \tau < \tau_L \}} 
    - e^{-r(T\wedge \tau_L)\}}G(X^L_{T\wedge \tau_L})\textbf{1}_{\{\tau= T \} \cup \{ \tau_L < \tau \}} \right . \\
    \left . - \int_0^{\tau\wedge T\wedge \tau_L}e^{-rs}\phi(s)d\xi_s 
    - \int_0^{\tau\wedge T\wedge\tau_L}e^{-rs}\mathcal{K}(s, X^L_s, \bar{\varsigma}^L_s, \text{m},\nu_s) ds \right]
\end{multline*}
    
\begin{equation}\label{eq:Vl}
    V_{0,L}(x, \bar{\varsigma}) = \sup_{(\nu,\vartheta)\in \mathcal V\times \mathcal T_{[0,T]}\times \Xi} J_{0,L}(x, \bar{\varsigma}; \nu,\vartheta) .
\end{equation}

We now want to show the convergence of the constrained state space problem to the unconstrained problem.
\begin{proposition}\label{prop:VL_to_V}
     $V_{0,L}(x,\bar{\varsigma})$ converges to $V_0(x,\bar{\varsigma})$ uniformly on compact subsets of $\mathbb{R}\times\mathbb{R}_+^n$ as $L\to\infty$.
\end{proposition}

\begin{proof}
Note that $\tau_L \to \infty$ as $L \to \infty$.
Additionally, $X^L$ and $\bar{\varsigma}^L$ converge pointwise and almost surely to $X$ and $\bar{\varsigma}$, respectively. By the dominated convergence theorem, Assumptions \textbf{(Polyx), (Poly), (Coef)}, and Lemma \ref{lm: X_moments} we get $J_{0,L}(x,\bar{\varsigma};\nu,\vartheta) \to J_0(x,\bar{\varsigma};\nu,\vartheta)$ as $L \to \infty$.
Consequently, \[\forall\epsilon,\; \exists L^*,\; \text{such that } \forall L\geq L^*, |J_{0,L}(x,\bar{\varsigma};\nu,\vartheta) - J_0(x,\bar{\varsigma};\nu,\vartheta)|\leq\epsilon.\] We now turn to the proof that $V_{0,L}(x,\bar{\varsigma}) \to V_0(x,\bar{\varsigma})$ as $L \to \infty$.
Let $\nu^*,\vartheta^*$ be optimal in \eqref{eq:Vl}. First, we show that $V_{0,L}(x,\bar{\varsigma}) \geq V_0(x,\bar{\varsigma}) - \epsilon$.
Recall that by the definition of a supremum, $\forall\epsilon>0$, $J_{0,L}(x,\bar{\varsigma};\nu^*, \vartheta^*) + \epsilon \geq V_{0,L}(x,\bar{\varsigma})$.
Additionally, by definition of optimality, $J_{0,L} (x,\bar{\varsigma}; \nu^*,\vartheta^*) \geq J_{0,L}(x,\bar{\varsigma};\nu^*, \vartheta)$ for all $\nu,\vartheta$.
Hence, since $J_{0,L}(x,\bar{\varsigma};\nu, \vartheta)$ converges to $J_0(x,\bar{\varsigma};\nu, \vartheta)$, $J_{0,L}(x,\bar{\varsigma}; \nu^*, \vartheta^*) \geq J_{0,L}(x,\bar{\varsigma};\nu^*, \vartheta)\geq J_0(x,\bar{\varsigma};\nu,\vartheta)-\epsilon \ \ \forall\nu,\vartheta$.
Further, from the definition of a supremum, we see 
\[V_{0,L}(x,\bar{\varsigma}) = J_{0,L}(x,\bar{\varsigma};\nu^*,\vartheta^*) \geq J_0(x,\bar{\varsigma}; \nu,\vartheta) - \epsilon = V_0(x,\bar{\varsigma}) - \epsilon \text{ for any $\nu, \vartheta$.}\] Now, we must prove that $V_0(x,\bar{\varsigma}) \geq V_{0,L}(x,\bar{\varsigma}) - \epsilon$. From the convergence of $J_{0,L}(x,\bar{\varsigma};\nu,\vartheta)$ to $J_0(x,\bar{\varsigma};\nu,\vartheta)$, we get $J_0(x,\bar{\varsigma};\nu, \vartheta) \geq J_{0,L}(x,\bar{\varsigma};\nu, \vartheta) - \epsilon$ for all $\nu,\vartheta$. Thus, $J_0(x,\bar{\varsigma};\nu^*,\vartheta^*) \geq J_0(x,\bar{\varsigma};\nu, \vartheta)$ for any $\nu, \vartheta$ and $J_0(x,\bar{\varsigma};\nu^*, \vartheta^*) + \epsilon \geq V_0(x,\bar{\varsigma}) $ for all $ \epsilon > 0$.
So, $J_0(x,\bar{\varsigma}; \nu^*,\vartheta^*) \geq J_0(x,\bar{\varsigma}; \nu, \vartheta) \geq J_{0,L}(x,\bar{\varsigma};\nu, \vartheta) - \epsilon$ for all $\nu,\vartheta$.
Then, $V_0(x,\bar{\varsigma}) = J_0(x,\bar{\varsigma}; \nu^*,\vartheta^*) \geq J_{0,L}(x,\bar{\varsigma}; \nu, \vartheta) - \epsilon = V_{0,L}(x,\bar{\varsigma}) - \epsilon$.
This leads us to conclude that $V_{0,L}(x,\bar{\varsigma}) \to V_0(x,\bar{\varsigma})$ as $L \to \infty$.
\end{proof}

\begin{remark}
    We note that since $\lambda_\infty(t) \leq \Lambda_\infty \ \forall t\in[0,T]$, by assumption \ref{as: L_inft_bound}, and $\bar{\varsigma}_L(t)\in[0,L]^n$, there exists $\Lambda_L < \infty$ such that $\lambda_L(t) = \lambda_\infty(t) + \bar{d}\bar{\varsigma}_L(t) \leq \Lambda_L$ for all $t\in[0,T]$. We note that specifically, one could set $\Lambda_L = \Lambda_\infty + nL\max_{k}d_k$.
\end{remark}

\section{Discrete Time Markov Chain Model}\label{sec:MDP}

The variational partial differential equation \textbf{(VPDE)} provides a characterization of the value function as the exact continuous-time solution of the control problem, see \cite{bensoussan2011applications}. However, the presence of a gradient constraint, arising from the singular control structure, renders the equation highly nonlinear and degenerate, typically leading to a free-boundary problem. From a numerical standpoint, this class of variational inequalities poses substantial difficulties: classical finite-difference or finite-element schemes must simultaneously approximate the solution and identify the active constraint region, and convergence guarantees generally rely on monotone and stable discretizations. In practice, this makes direct numerical resolution of the VPDE either computationally prohibitive or infeasible, especially in higher dimensions.
As a consequence, we adopt a discrete-time approximation framework to construct a numerically tractable scheme. Specifically, we rely on the Markov chain approximation method introduced by Kushner \cite{kushner1990numerical}, \cite{kushner1992numerical}, \cite{kushner1991numerical}. The core idea is to approximate the controlled diffusion by a suitably designed discrete-time Markov chain whose local transition moments match those of the original continuous-time process. This weak approximation preserves the key probabilistic structure of the dynamics while avoiding the direct discretization of the gradient constraint.
Building on this approximation, the original continuous-time control problem is reformulated as a discrete-time Markov decision process defined on the state space of the approximating chain. The associated dynamic programming equation is then well-defined and can be solved numerically using standard value iteration or policy iteration techniques, yielding a convergent approximation of the value function and optimal control as the discretization step tends to zero.

\subsection{Hawkes Process Discretization}\label{sec:hawkesdisc}
In order to define the discre-time Markov chain model, we must first define a discretization of the Hawkes process. Let $\bar{\varsigma}^h_i = (\varsigma_i^{(1),h},\dots, \varsigma_i^{(n),h})$ so that the discrete intensity is $\lambda_i^h := \lambda_\infty(t_n) + \bar{d}\bar{\varsigma}_i^{h\top}$. Since the Hawkes process kernel is a sum of exponentials, the process can be simulated exactly using the method introduced by Dassios and Zhang in \cite{dassios2013exact}. Let $\eta_j$ be the time of the $j^\text{th}$ jump. Then the exact update of $\bar{\varsigma}_i^h$ is computed via
$$\varsigma_{i+1}^{(k),h} = \varsigma_{i}^{(k),h}e^{-q_k\Delta t^h(x)} + \sum_{\eta_j\in[t_i, t_{i+1})} \varsigma_{\eta_j}^{(k),h},$$
and the corresponding intensity process is $\lambda_{i+1}^h = \lambda_\infty(t_{i+1}) + \bar{d}\bar{\varsigma}^h_{i+1}$, where 
 $\Delta t^h(x)$ is a time step defined rigorously later in the next section. The locally consistent expectation-based update is
\begin{equation*}
\mathbb{E}\big[\varsigma_{i+1}^{(k),h}\big] = \varsigma_i^{(k),h} e^{-q_k \Delta t^h(x)} + \mathbb{E}[\varrho(z)]\, \lambda_n^h \Delta t^h(x)\, \frac{1-e^{-q_k \Delta t^h(x)}}{q_k \Delta t^h(x)},
\end{equation*}
or equivalently
\begin{equation*}
    \mathbb{E}\big[\bar{\varsigma}_{i+1}^h\big] = \underbrace{\bar{\varsigma}_i^h\odot e^{-\bar{q}\Delta t^h(x)}}_{\bar{\varsigma}^\mathrm{decay}} + \mathbb{E}[\varrho(z)]\lambda_i^h\Delta t^h(x)\frac{1 - e^{-\bar{q}\Delta t^h(x)}}{\bar{q}\Delta t^h(x)},
\end{equation*}
where $\odot$, the exponential, and the division are all defined as Hadamard products. The continuous-valued $\bar{\varsigma}$-targets obtained above must be projected onto a finite $\mathbf{S}$-grid used by the chain and defined formally in Definition \ref{def:Sgrid}. For each component perform two-point linear splitting: if a target $\varsigma^{(k)}$ lies between grid points $s_{j}$ and $s_{j+1}$, assign weights $w_j^{(k)}$ with $ w_{j+1}^{(k)}=1-w_j^{(k)}$
so that the expected component equals the target. For the $n$-dimensional vector $\bar{\varsigma}$ we use the product-splitting rule.
Specifically, define the linear projection 
% 1. One-dimensional linear projection
$$\mathcal{P}^{\mathrm{lin}}(\varsigma^{(k)} \mid s_j, s_{j+1}) =
\begin{cases}
w_j^{(k)} := \frac{s_{j+1} - \varsigma^{(k)}}{s_{j+1}-s_j}, & \text{assigned to } s_j,\\[1mm]
w_{j+1}^{(k)} := \frac{\varsigma^{(k)} - s_j}{s_{j+1}-s_j}, & \text{assigned to } s_{j+1}.
\end{cases}$$
For our $n$-dimensional vector $\bar{\varsigma}$, we can use the product of the one-dimensional linear projections
$$\mathcal{P}_{\mathbf S}^{\text{no jump}}(\mathbf j' \mid \bar{\varsigma}) := \prod_{k=1}^n 
\mathcal{P}^{\mathrm{lin}}\Big(\varsigma^{(k)} \mid s^{(k)}_{j'_k}, s^{(k)}_{j'_k+1}\Big),$$
where \(s^{(k)}_{j'_k}, s^{(k)}_{j'_k+1}\) are the grid points bracketing \(\varsigma^{(k)}\) and $\mathbf{j}'$ is the vector containing the indices of the project grid points.  
This produces at most $2^n$ nonzero probabilities corresponding to the corners of the hypercube surrounding \(\bar{\varsigma}\), and preserves the expectation:
\[
\mathbb{E}_{\mathcal{P}_{\mathbf S}^{\text{no jump}}}[\bar{\varsigma}] = \bar{\varsigma}.
\]

Adding the marks in the modeling, we defined the transition probabilities for $\bar{\varsigma}$ by
% 3. Projection for k-jump events
\[
\mathcal{P}_{\mathbf S}^{(z)}(\mathbf j' \mid \bar{\varsigma}, z) := \prod_{k=1}^n 
\mathcal{P}^{\mathrm{lin}}\Big(\varsigma^{(k)} + \varrho(z) \;\Big|\; s^{(k)}_{j'_k}, s^{(k)}_{j'_k+1}\Big).
\]  
This also preserves the expectation:
\[
\mathbb{E}_{\mathcal{P}_{\mathbf S}^{(z)}}[\bar{\varsigma}] = \bar{\varsigma} + \varrho(z)\textbf{e}_n.
\]

\subsection{Discrete-time controlled Markov chain driven by Hawkes jumps}\label{sec:MCdisc}
First, we design the underlying Markov chain that will be used to approximate $X_t$. Let $h > 0$ be a discretization parameter. 
Using our constrained system dynamics from \eqref{eq:constrained_dyn}, we will assume without loss of generality that $L$ is an integer multiple of $h$. 
We define the state space of the Markov chain as $\mathbb{S}_L^h = \{x = kh \in\mathbb{R} : k\in\mathbb{Z} \}\cap[-L,L]$.
Let $\{ X_{i}^h : i \in\mathbb{N} \}$ be a discrete-time controlled Markov chain with state space $\mathbb{S}_L^h$. 
% \hl{Due to the existence of the reflection process, we introduce an extended state space, as in} \cite{amarjit_budhiraja2007convergent}, $S_{h}^+ = \{x = kh \in\mathbb{R} : k\in\mathbb{Z} \}\cap[-\ell - h,\ell + h]$.
Define $\Delta X_{i}^t = X_{i+1}^h - X_{i}^h$. At each time step, we can choose to either inject money via the singular control or else let the process diffuse. 
Define $\pi_{i}^h\in\{0,1, +, -\}$ as the action taken at step $i$ where
\begin{itemize}
    \item $\pi_i^h = 0$: diffusion,
    \item $\pi_i^h = 1$: injection of size $h$, 
    \item $\pi_i^h = +$ and $\pi_i^h=-$ correspond to  reflections via reflection process $R^+$ and $R^-$ respectively,
    \item $\pi_i^h = k_-$ corresponds to reflections of $\varsigma^{k}$ via reflection process $\mathcal R^{(k)}$, for $k\in \{1,\dots,n\}.$
\end{itemize}
Denote the injection amount by $\Delta \xi_{i}^h = h$. This choice of amount is standard in the literature and necessary for the rescaling procedure, \cite[Section 8.3]{kushner1992numerical}. Consequently, for $\pi_i^h=1$, we have $\Delta X_i^h = \Delta \xi_i^h = h$. For completeness, we define $\{Y_{i}^h\}$ in the same way as $\{X_{i}^h\}$ but it represents the process without the singular control, as defined in \eqref{eq:non-sing_dyn}. We detail how the system evolves given different possible actions $\pi$. For steps where $\pi_{i}^h=0$, the increment $\Delta X_{n}^h$ should diffuse as follows
$$ \Delta X_i^h \approx \int_{t_i^h}^{t_{i+1}^h} \mu(X_t) dt + \int_{t_i^h}^{t_{i+1}^h} \sigma(X_t,\beta_t) dW_t + \int_{t_i^h}^{t_{i+1}^h} \int_0^{\Lambda_L} \int_{\mathbb{R}_+} \chi( X_t, z,\gamma_t) \mathbf{1}_{\theta \le \lambda_t} \Pi(dt,d\theta,dz), $$
over small intervals.
We note that for the steps where $\pi_i^h\in\{-,+\}$, the process immediately reflects back into the state space $\mathbb{S}_L^h$ with a step of size $h$. That is, if $\pi_i^h=-$ then $ \Delta X_i^h = -h$ and if $\pi_i^h = +$ then $\Delta X_i^h = h$.
In general, for any $i$, one can write 
$$\Delta X_i^h = \Delta X_i^h \mathbf{1}_{\pi_i^h=0} + \Delta X_i^h \mathbf{1}_{\pi_i^h=1} - \Delta X_i^h \mathbf{1}_{\pi_i^h=-} + \Delta X_i^h \mathbf{1}_{\pi_i^h=+}.$$

\begin{definition}[Admissible Control $\pi^h$]
    Let $\pi^h = (\pi_1^h, \pi_2^h, \dots)$ and $\mathcal{F}_i^h = \sigma\{ X_j^h, \pi_j^h, j \le i\}$, the $\sigma$-algebra generated by $X^h$ and $\pi^h$. 
A control $\pi^h$ is admissible if:
\begin{itemize}
    \item $\pi_i^h$ is $\mathcal{F}_i^h$-adapted,
    \item For all $x \in \mathbb{S}_L^h$, $\mathbb{P}(X_{i+1}^h = x \mid \mathcal{F}_i^h) = p^h(X_i^h,x\mid \pi_i^h)$,
    \item $X_i^h \in \mathbb{S}_L^h$ for all $i$.
\end{itemize}
The set of admissible controls starting at $x$ is denoted $\mathcal{A}_x^h$.
\end{definition}

With the notation and underlying processes of the Markov chain now specified, we turn to the construction of the chain itself, namely the definition of its transition probabilities. These probabilities are chosen so that the discrete-time process ${X_i^h}$ is locally consistent with the underlying diffusion. More precisely, we require the following conditions to hold:

\begin{align}
    \mathbb{E}^{h,0}_{x,i}\left[\Delta X^h_{i}\right] &= \mu(x)\Delta t^h(x) + o(\Delta t^h(x)) \label{eq:local_consistency_E}\\
    \text{Var}^{h,0}_{x,i} (\Delta X^h_{i}) &= \sigma( x,\beta_{t_i^h})\Delta t^h(x) + o(\Delta t^h(x)) \label{eq:local_consistency_cov}
\end{align}
where $\mathbb{E}^{h,0}_{x,n}[\cdot]$ denotes the expectation with initial condition $x$, discretization $h$, current step $i$, and current action $\pi_i^h = 0$.

\begin{remark}
We note that for constrained processes with more intricate reflection behavior such as state-dependent reflection directions or magnitudes, additional local consistency conditions are required; see \cite[Section 5.7.2]{kushner1992numerical}. In the present setting, however, the reflection structure is simple, and, as in \cite{budhiraja2007convergent}, these additional conditions are automatically satisfied by construction.
\end{remark}

Now, we define the transition probabilities and time interval for the Markov chain.
To ensure the time discretization is independent of the controls, we use the method developed by Kushner and Dupis in section 5.2.1 of \cite{kushner1992numerical}.
Define 
$$
D(x) := \max_{\beta \in \mathcal{A}(x)} \sigma^2(x,\beta), \quad 
Q_h(x) := D(x) + h \mu(x), \quad 
\Delta t^h(x) := \frac{h^2}{Q_h(x)}.$$
Then diffusion probabilities for the Markov chain are defined thusly:
\begin{align*}
p_D^h(x,x+h|\beta) &= \frac{\frac12 \sigma^2(x,\beta) + h\, \mu^+(x)}{Q_h(x)}, \\
p_D^h(x,x-h|\beta) &= \frac{\frac12 \sigma^2(x,\beta) + h\, \mu^-(x)}{Q_h(x)}, \\
p_D^h(x,x|\beta) &= 1 - p_D^h(x,x+h) - p_D^h(x,x-h).
\end{align*}

To incorporate jumps we must include information from Section \ref{sec:hawkesdisc}. Since the Hawkes intensity evolves as a sum of exponentials driven by past marks, the conditional probability of a jump in the next interval is 
$\lambda_i^h \Delta t^h(x) + o(\Delta t^h(x))$ \cite{foschi2021measuring}.
Let
\[
p_{1,n} = e^{-\lambda_n^h \Delta t^h(x)} \lambda_n^h \Delta t^h(x)
\]
be the probability of 1 jump in a time step, as defined by the Poisson density.

\begin{definition}[Discrete jump probabilities $\mathcal{Q}_u$]
\label{def:Sgrid}Let $(\mathbb{R}_+, \mathcal B({\mathbb{R}_+}))$ be the mark space with mark law $\emph{m}(dz)$, 
and let $\chi_h(x,z,\gamma)$ be the jump map producing increments in the target space $(\mathbf{S}, \mathcal B_{\mathbf{S}})$. 
Partition $\mathbf{S}$ into $U$ disjoint measurable cells, $\{ C_1, C_2, \dots, C_U \}$, such that
\[ 
C_u \cap C_v = \varnothing \text{ for } u \neq v, \quad \bigcup_{u=1}^U C_u = \mathbf{S}.
\]
Each cell $C_u$ represents a set of possible jump increments of the process.
For a given spatial transition $x \to y$, define $u(x,y)$ as the unique index satisfying
\[
y - x \in C_{u(x,y)}.
\]
Then, the discrete jump probability associated with this transition is
\[
\mathcal{Q}_{u(x,y)}(x)
:= \emph{m}\big( \underbrace{\{ z \in \mathbb{R}_+ : \chi_h(x,z,\gamma) \in C_{u(x,y)} \}}_{\mathcal{Z}_{u(x,y)}(x)}\big),
\]
that is, the probability that a mark $z$ produces a jump increment lying in the same cell as $y-x$.
These probabilities satisfy
\[
\mathcal{Q}_u(x) \in [0,1], \qquad \sum_{u=1}^U \mathcal{Q}_u(x) = 1.
\]
\end{definition}
We now turn to the discretization of $\text{m}$ by selecting representative marks $\{z_\upsilon\}_{\upsilon=1}^\Upsilon$ where $\Upsilon>0$  
and weights $\{q_\upsilon(x)\}_{\upsilon=1}^\Upsilon$ are defined  below.

\begin{definition}[Voronoi Discretization of the Mark Measure]
Given a set of representative marks $\{z_\upsilon\}_{\upsilon=1}^\Upsilon \subset \mathfrak P$, let $\{V_\upsilon\}_{\upsilon=1}^\Upsilon$ be a measurable partition of the compact support $\mathfrak P$ such that $z_\upsilon \in V_\upsilon$ for all $\upsilon = 1, \dots, \Upsilon$. The discretized measure $\emph{m}^h$ is defined by:
\begin{equation*}
   \emph{m}^h(dz) := \sum_{\upsilon=1}^\Upsilon q_\upsilon(x) \delta_{z_\upsilon}(dz),
\end{equation*}
where $\delta_{z_\upsilon}$ is the Dirac measure at $z_\upsilon$ and the weights are given by the measure of the Voronoi cells:
\begin{equation*}
    q_\upsilon(x) = \emph{m}(V_\upsilon), \quad \upsilon = 1, \dots, \Upsilon.
\end{equation*}
By construction, $\sum_{\upsilon=1}^\Upsilon q_\upsilon(x) =\emph{m}(\mathfrak P) = 1$.
\end{definition}

Define the index set $\mathcal I_u := \{\upsilon : z_\upsilon \in \mathcal Z_u(x)\}$.
Then
$$\mathcal{Q}_u(x) = \sum_{\upsilon \in \mathcal I_u} q_\upsilon(x),
\qquad
\text{m}^h(dz) 
:= \sum_{\upsilon \in \mathcal I_u} 
q_\upsilon(x)\, \delta_{z_\upsilon}(dz).$$
Finally, for any integrable function $f$,
$$\int_{\mathcal Z_u(s,x)} f(z)\, \text{m}^h(dz) = \sum_{\upsilon \in \mathcal I_u} q_\upsilon(x)\, f(z_\upsilon).$$

We now turn to the one-step transition probabilities for the evolution of $(X^h,\bar{\varsigma}^h)$ defined by 

\begin{align*}
p^h\big((x,\mathbf j),(y,\mathbf j')\big)
&= p_{0,n}\ p_D^h(x,y)\ \mathcal P_{\mathbf S}^{\mathrm{no\ jump}}\!\big(\mathbf j' \mid \bar\varsigma^\mathrm{decay}\big) \\
&+ p_{1,n}\!
\int_{\mathcal Z_{u(x,y)}(s,x)}\!
\mathcal P_{\mathbf S}^{(z)}\big(\mathbf j' \mid \bar\varsigma,z\big)\,\emph{m}(dz) + o(\Delta t^h(x)).
\end{align*}
Or, a more practical discrete version
\begin{align}
\nonumber p^h\big((x,\mathbf j),(y,\mathbf j')\big)
&= p_{0,n}\; p_D^h(x,y)\; \mathcal{P}_{\mathbf S}^{\mathrm{no\ jump}}\!\big(\mathbf j' \mid \bar\varsigma^{\mathrm{decay}}\big) \notag \\
&\quad +\ p_{1,n}\ \sum_{\upsilon \,:\, \chi_h(x,z_\upsilon,\gamma)\in C_{u(x,y)}} 
       q_\upsilon(x)\; \mathcal P_{\mathbf S}^{(z_\upsilon)}\!\big(\mathbf j' \mid \bar\varsigma,z_\upsilon\big) + o(\Delta t^h(x)),\label{eq:one_step_trans}
\end{align}
where $\chi_h$ denotes any uniformly bounded approximation of the jump function 
$\chi$ satisfying $|\chi_h - \chi|\to 0$ as $h\to 0$. The singular control and reflection actions have deterministic outcomes, so their transition probabilities are
$$p^h(x,x+h|\pi=1) = 1, \quad 
p^h(L,L-h|\pi=-) = 1, \quad 
p^h(-L,-L+h|\pi=+) = 1.$$

\begin{remark}
    To alleviate notations, we omit the dependence with respect to $\varsigma$ but we keep in mind that 
    $p(x,x|\pi=k_-)=1$ corresponds to a transition on $\varsigma^k$ from $\varsigma^k$ to $\varsigma^k-h$, for $k\leq n$.
\end{remark}
For each $x$, define a family of interpolation intervals $\Delta t^h(x)$, as in the previously defined equation: $\Delta t^h(x) := \frac{h^2}{Q_h(x)}$.
Let's define
$$t^h_{0}=0, \quad \Delta t^h_{j}=\Delta t^h(X^h_{j}, \pi^h_{j}), \quad t^h_{i} = \sum_{j=0}^{i-1}\Delta t^h_{j}.$$   

\begin{remark}
    
Note that by section 5.6.2 of \cite{kushner1992numerical}, the controlled Markov chain $X_n^h$ is locally consistent as soon as $\Delta t^h(x)\to 0$ uniformly in $x$ and the controls as $h\to 0$ such that
\begin{enumerate}
    \item $p_D^h(x,y|\pi,\beta)$ is locally consistent with  \eqref{eq:local_consistency_E} and \eqref{eq:local_consistency_cov}
    \item there is $\psi^h = o(\Delta t^h(x))$ such that $p^h(x,y|\pi,\nu)$ can be written as in \eqref{eq:one_step_trans}.
\end{enumerate}
\end{remark}

We now define the reward and value functions for the controlled Markov chain. Let $t_N^h = T$ and let $\mathfrak{T} = \sup\{ i : t_i^h \leq \tau^h\}$ be the time step corresponding to the optimal stopping. 
\begin{multline*}
    J_0^h(x, \bar{\varsigma}; \nu^h,\vartheta^h, \pi^h) = \mathbb{E}\left[- \sum_{i=0}^{n\wedge \mathfrak{T}} e^{-rt_i^h}\left( \mathcal{K}(t_i^h, X_i^h, \bar{\varsigma}_i^h,   \text{m},\nu_i^h,)\Delta t_i^h + \phi(t_i^h)h\textbf{1}_{ \pi_i^h=1} \right) \right. \\
    \left. +  e^{-rT}G(X^h_T)\textbf{1}_{ n=N} + e^{-rt_\mathfrak{T}^h}F(X^h_\mathfrak{T})\textbf{1}_{n\geq \mathfrak{T}}\right]   
\end{multline*}

\begin{equation}\label{eq:Vh}
    V_0^h(x, \bar{\varsigma}) = \sup_{\nu^h, \vartheta^h, \pi^h} J_0^h(x, \bar{\varsigma}; \nu^h, \vartheta^h, \pi^h).
\end{equation}

\begin{theorem}[Discrete-time DPP with optimal stopping]\label{thm:DPP_discrete} Define, for every $i\in \{0, \dots, N-1\}$, the continuation operator 
\[
(\mathcal{L}_i^h V)(x,\bar\varsigma):=
\sup_{\nu,\vartheta}\left\{
-\mathcal{K}(t_i^h,x,\bar\varsigma,\emph{m},\nu^h)\Delta t_i^h
-\phi(t_i^h)h\,\mathbf{1}_{\{\pi=1\}}
+ e^{-r\Delta t_i^h}\mathbb{E}\!\left[V(x_{i+1}^h,\bar\varsigma_{i+1}^h)\mid x_i^h=x,\bar\varsigma_i^h=\bar\varsigma\right]
\right\}.
\]
The value functions satisfy the following backward-dynamic programming equation:
$$\begin{cases}
    \label{eq:Snell_DPP}
V_i^h(x,\bar\varsigma)=\max\Big\{\, e^{-r\Delta t_i^h}F(x), (\mathcal{L}_i^h V)(x,\bar\varsigma)\Big\}, \forall i\in\{0,\dots, N-1\}\\
V_N^h(x,\bar\varsigma)=e^{-r(T-t_N^h)}G(x)
\end{cases}$$
where the conditional expectation is taken under the controlled transition at time $i$.
\end{theorem}
The proof of the theorem above is a slight extension of
\cite{bertsekas2012dynamic,bertsekas1996stochastic} to the case of a discretized
Hawkes-driven state process. Whenever the maximum in the dynamic programming
equation is attained at time step $i$ by the stopping payoff $F$, the optimal
stopping time is given by $\tau^{h}=t_i^h$. Otherwise, the system continues to
evolve according to the controlled dynamics, and the optimal control
$\hat{\nu}^h$ is selected so as to attain the supremum in the continuation
value.

\section{Convergence of the discrete to the continuous time optimizations}\label{sec:cv}
We now turn to the heart of this study and the convergence of value function $V^h$ in a discrete time framework to $V$ for the continuous time model. We starts to show that the MDP developed in Section \ref{sec:MDP} converges to the controlled SDE \eqref{SDE:singular}. We provide more particularly a tightness result  fundamental to derive the convergence of the value functions and the optimizers.

\subsection{Convergence of the MDP to the singular SDE}\label{sec:MDPtoSDE}
The convergence of the Markov decision process to the singularly controlled
stochastic differential equation is established in three main steps. First, we
construct a continuous-time interpolation of the discrete-time controlled
process and relate it to the original MDP dynamics. Second, we introduce an
appropriate time-rescaling procedure and analyze its key properties, ensuring
that the rescaled dynamics are consistent with the limiting singular control
regime. Finally, we prove that the family of processes associated with the MDP
is tight in the Skorokhod space, which allows us to see any limit point as a
candidate solution of the target singularly controlled SDE.
\subsubsection{Continuous time interpolation}
We construct the continuous time interpolation to be piecewise constant on $[t^h_{i}, t^h_{i+1}), \ i\geq 0$. 
First, we need to define the discrete time process. 
Let $\xi^h_{0} = B^h_{0} = M^h_{0} =H_0^h=R_0^h 0,$ and define:
\[
    \xi^h_{i}= \sum_{j=0}^{i-1}\Delta \xi^h_{j}, \;
    B^h_{i} = \sum_{j=0}^{i-1}\textbf{1}_{\pi^h_{j}=0}\mathbb{E}^h_{j} [\Delta X^h_{j}], \;
    M^h_{i} = \sum_{j=0}^{i-1}\textbf{1}_{\pi^h_{j}=0}\left(\Delta Y^h_{j} - \mathbb{E}^h_{j} [\Delta Y^h_{j}]\right), \]
    \[ 
    H_i^h = \sum_{j:\eta_j^h<i} \chi_h( X_{\eta_j^h}^h, \gamma_{\eta_j^h}^h, z_j),
\]
where $\eta_j^h$ is the discretized jumping time of $X$ and 
\[ R^h_{k,i} = \sum_{j=0}^{i-1}\textbf{1}_{\pi^h_{j}=k}\Delta X^h_{j}, \quad k\in\{+,-\},\; \mathcal R^{(k),h}_{i} = \sum_{j=0}^{i-1}\textbf{1}_{\pi^h_{j}=k_-}\Delta \varsigma^{(k),h}_{j},\; k\in \{1,\dots,n\}. \]

With the discrete time processes defined, we can now define the constant interpolations as:
\[\
        X^h_t:=X^h_{i^h(t)},\; B^h_t:=B^h_{i^h(t)}, \; M^h_t:=M^h_{i^h(t)}, \; H^h_t:=H^h_{i^h(t)}\]
     \[  \varsigma^{(k),h}_t := \varsigma^{(k),h}_{i^h(t)}, \;\beta^h_t := \beta^h_{i^h(t)},\; 
        \gamma^h_t := \gamma^h_{i^h(t)},\;
        \xi^h_t:=\xi^h_{i^h(t)},\;  R_{j,t}^h:=R^h_{j, i^h(t)}, \; \mathcal R_{t}^{(k),h}:=\mathcal R^{(k),h}_{i^h(t)},
\]where $i^h(t) = \max\{i \ : \ t^h_{i}\leq t\}, \ t\geq 0$.
We set $\mathcal{F}^h_{i^h(t)} = \sigma\{ X^h_s, \xi^h_s \ : \ s\leq t\} = \mathcal{F}^h_t$. Note that

 \[X^h_t = x + B^h_t + M^h_t + \xi^h_t + H^h_t + R^{h}_{+,t} - R^{h}_{-,t}+\varepsilon^h_t,\]
 
where $ \varepsilon^h_t$ is $\mathcal{F}^h_t$-adapted and 
$\lim_{ h \to 0 }\sup_{t\in[0, T]} \mathbb{E}|\varepsilon^h_t|=0.$
\begin{remark}
    Note that $\varsigma$ is already discretized from Section \ref{sec:hawkesdisc}. We will work under the same time interpolation with $\varsigma$ as previously defined but we omit to write the exact definition to alleviate the notation. 
\end{remark}
\subsubsection{Time Rescaling}
To prove the convergence of the discrete solution to the continuous solution, time rescaling is introduced and will later be implemented. 
We define the rescaled time increments $\{ \Delta\hat{t}^h_i \ : \ i\in\mathbb{N} \}$ via
$$\Delta \hat{t}^h_i = \Delta \hat{t}^h_i\textbf{1}_{ \pi^h_i = 0} + h\textbf{1}_{ \pi^h_i \neq 0}$$
with $\hat{t}_0 = 0$ and $\hat{t}_i = \sum_{j=0}^{i-1} \Delta \hat{t}^h_j$ for $i\geq 1$.
This time rescaling can be interpreted as a stretch of the time scale by $h$ at jump steps. Similarly to \cite[Definition 4.1]{budhiraja2007convergent}, we define the time rescaled process as follows. 

\begin{definition}\label{def:rescale_hatT}
     The rescaled time process $\hat{T}^h(\cdot)$ is the unique continuous nondecreasing process satisfying the following:
     \begin{enumerate}[label=(\alph*)]
        \item $\hat{T}^h(0) = 0$,
        \item the derivative of $\hat{T}^h(\cdot)$ on $(\hat{t}^h_i, \hat{t}^h_{i+1})$ is 1 if $\pi^h_i = 0$,
        \item the derivative of $\hat{T}^h(\cdot)$ on $(\hat{t}^h_i, \hat{t}^h_{i+1})$ is 0 if $\pi^h_i \in \{-, +, 1,1_-,\dots,n_-\}$.
     \end{enumerate}
     We define $\mathcal{S}^h:=\hat{T}^h(T)$.
\end{definition}

\begin{remark}\label{Tequicontinuous}
    Note that $\hat{T}^h$ is bounded, 1-Lipschitz and, consequently, uniformly equicontinuous.
\end{remark}

Let $\hat{i}^h(t) := \max \{ i \ : \ \hat{t}^h_i \leq t \}$ for $t\geq 0$. 
Let $\hat{X}^h_t:= X^h_{\hat{T}^h(t)}$ be the rescaled and interpolated process.
We define $\hat{\xi}^h, \hat{H}^h, \hat{B}^h, \hat{M}^h, \hat{\beta}^h, \hat{\gamma}^h, \hat{R}_j^h$ and $\hat{\mathcal{F}}^h$ similarly.
We get
\begin{align*}
    \hat{X}^h_t &= x + \int_0^t \mu(\hat{X}^h_s )d\hat{T}^h(s) + \int_0^t \sigma( \hat{X}^h_s, \hat{\beta}^h_s)d\hat{W}^h_s + \hat{\xi}^h_t  +\hat{H}^h_t+ \hat{R}_{+,t}^h - \hat{R}_{-,t}^h+ \hat{\varepsilon}^h_t \notag \\
    &= x + \hat{B}^h_t + \hat{M}^h_t + \hat{\xi}^h_t +\hat{H}^h_t+ \hat{R}_{+,t}^h - \hat{R}_{-,t}^h + \hat{\varepsilon}^h_t,
\end{align*}
where $\hat{\varepsilon}^h_t$ is $\hat{\mathcal{F}^h(t)}$-adapted and 
$$\lim_{h\to 0}\sup_{t\in [0,S^h]} \mathbb{E} |\hat{\varepsilon}^h_t| = 0 .$$ We define $\hat{\tau}^h = \hat{T}^h(\tau^h) \leq \mathcal{S}^h$. We now give some properties of the inverse function of the time rescaling process we will use later.  
\begin{lemma}\label{lm:measure_theory}
    Let $\hat{G} \ : \ [0,T] \to [0,\mathcal{S}^h]$ be a continuous and nondecreasing function with $\hat{G}(0)=0$ and $\hat{G}(T)=\mathcal S^h$.
    Define the generalized inverse $G \ : \ [0,\mathcal{S}^h] \to [0,T]$ as $G(t) = \inf\{ s \ : \ \hat{G}(s) > t \}$.
    Then for all bounded and measurable functions $g \ : \ [0,T] \to \mathbb{R}$,
$$\int_{0}^{G(t)} g(s)d\hat{G}(s) = \int_{0}^t g(G(s))ds .$$
\end{lemma}
 
\begin{proof}
Define the Lebesgue-Stieltjes measure $\mu$ on $[0,T]$ as $\mu((a,b]) = \hat{G}(b) - \hat{G}(a)$ for $a<b$. Additionally define the push-forward measure $\nu := \mu\circ G^{-1}$ on $[0,\mathcal{S}^h]$. For any $u\in[0,\mathcal{S}^h]$,
$$\nu((0,u]) \;=\; \mu\big(G^{-1}((0,u])\big) 
= \mu\big(\{s\in[0,T]:\ \hat{G}(s)\in(0,u]\}\big).$$
Since $\hat{G}$ is continuous and nondecreasing, the set $\{s:\hat{G}(s)\le u\}$ has supremum $s^*=\sup\{s:\hat{G}(s)\le u\}$, and consequently
$$ \mu\big(\{s:\hat{G}(s)\in(0,u]\}\big) 
= \hat{G}(s^*)-\hat{G}(0) = u. $$
Thus $\nu((0,u])=u$ for all $u\in[0,\mathcal{S}^h]$, which shows that $\nu$ coincides with Lebesgue measure on $[0,\mathcal{S}^h]$. Now, for bounded measurable $g$,
$$\int_{0}^{G(u)} g(s)\,d\hat{G}(s)
= \int_{[0,G(u)]} g(s)\,\mu(ds)
= \int_{(0,u]} g(G(v))\,\nu(dv)
= \int_{0}^{u} g(G(v))\,dv. $$
\end{proof}

\begin{lemma}[Uniform endpoint convergence under equicontinuity]\label{lem:uniform_endpoint}

For every $\varepsilon>0, h
_0>0$ there exists $\delta>0$ (independent of $h$) such that
$$\sup_{0<h<h_0 \ }\sup_{|t-T|<\delta} \big|\hat T^h(t)-\mathcal{S}^h\big| < \varepsilon.$$
In particular, for each fixed $h$, $\hat T^h(t)\to \mathcal{S}^h$ as $t\to T$, and the convergence is uniform in $h\in(0,h_0)$ in the sense above.
\end{lemma}

\begin{proof} First, remember from Remark \ref{Tequicontinuous} that the family $\{\hat T^h\}_{0<h<h_0}$ is equicontinuous on $[0,T]$, then for any $h_0>0$
there exists a modulus $\omega(\delta)$ with $\omega(\delta)\downarrow0$ as $\delta\downarrow0$ such that
$$\sup_{0<h<h_0 \ }\sup_{|t-s|\le\delta} \big|\hat T^h(t)-\hat T^h(s)\big| \le \omega(\delta).$$
Fix $\varepsilon>0$. Choose $\delta>0$ so small that $\omega(\delta)<\varepsilon$.
For any $h\in(0,h_0)$ and any $t\in(T-\delta,T]$ we have, by monotonicity and equicontinuity,
\[
0 \le \mathcal S^h - \hat T^h(t) = \hat T^h(T) - \hat T^h(t)
\le \sup_{|s-t|\le\delta} \big|\hat T^h(s)-\hat T^h(t)\big|
\le \omega(\delta) < \varepsilon.
\]
Hence $\sup_{0<h<h_0}\sup_{|t-T|<\delta} |\hat T^h(t)-\mathcal{S}^h|<\varepsilon$, as required.
\end{proof}
\begin{lemma}\label{lm:mathfrakS} $\sup_h \mathbb{E}[\mathcal{S}^h]=$
$\sup_h \mathbb{E}[\hat{T}^h(T)]<\infty.$
\end{lemma}
\begin{proof}
    Since 
    $$\hat{T}^h(T) = \hat{t}_N = \sum_{i=0}^{N-1}\Delta t_i^h\textbf{1}_{\pi_i^h=0 } + h\sum_{i=0}^{N-1}\textbf{1}_{\pi_i^h\neq 0 } = T + h\sum_{i=0}^{N-1}\textbf{1}_{\pi_i^h\neq 0 },$$
    it suffices to find a uniform bound on $h\mathbb{E}\Big[\sum_{i=0}^{N-1}\textbf{1}_{\pi_i^h\neq 0 }\Big]$. We know that each time $\pi_i^h=1$, there is a cost $\phi(t_i^h)h$, while for the reflections $\pi^h_i=+,-,1_-,\dots,n_-$ there are no costs. For any admissible control, we note that
    $$J_0^h(x,\bar{\varsigma}; \cdot) \leq - \mathbb{E}\Bigg[\sum_{i=0}^{N-1} e^{-rt_i^h}\phi(t_i^h)h\textbf{1}_{\pi_i^h=1}\Bigg] + C_0$$
    where $C_0$ absorbs the bounded running and terminal costs. Now, taking a supremum over the controls, we directly show that there exists a constant $C>0$ such that
      \[h\mathbb{E}\Bigg[\sum_{i=0}^{N-1}\textbf{1}_{\pi_i^h=1 }\Bigg]\leq C<\infty.\]

    Since the constant upper bound given by the right hand side is independent of $h$, the bound is uniform.
    Thus, $\mathbb{E}[\hat{T}^h(T)] = T + h\mathbb{E}\Big[\sum_{i=0}^{N-1}\textbf{1}_{\pi_i^h=1 }\Big] \leq T + C$. So, $\sup_h \mathbb{E}[\hat{T}^h(T)] \leq T + C < \infty.$
\end{proof}

\subsubsection{Tightness and convergence of the MDP to the singular SDE}
In this section, we focus on the convergence of the MDP previously introduced in Section \ref{sec:MCdisc} to the solution to the singular controlled SDE system \eqref{SDE:singular}. The convergence will hold in a weak sense in Skorokhod space that we recall below. 

\begin{definition}[Skorokhod space]
Let $T>0$ and $E$ be a Polish space.
The \emph{Skorokhod space} $D([0,T];E)$ is the set of all c\`adl\`ag functions from $[0,T]$ into $E$,
that is, functions which are right-continuous and admit left limits at every $t\in(0,T]$.
\end{definition}

In order to get the tightness of the controls $\hat\beta^h,\hat\gamma^h$ we need additional properties for defining the admissibility. The Markovian structure of the problem naturally leads us to restrict the study to feedback controls.

\begin{definition}[Feedback controls]
   We define the set $\mathcal V^{F}$ as the set of $\nu:=(\beta,\gamma)$ with values in $\mathcal V$ such that there exists continuous mapping $\psi^b:[0,T]\times \mathbb 
   R\longrightarrow B$ and $\psi^g:[0,T]\times \mathbb 
   R\longrightarrow \Gamma$ such that $\beta_t=\psi^b(t,\hat X^h_t),\; \gamma_t=\psi^g(t,\hat X^h_t)$. 
\end{definition} 
\begin{remark}\label{rm:tight}
    As a consequence of this definition, if we restrict the study to control $\nu^h:=(\beta^h,\gamma^h)\in \mathcal V^F$, then $\hat\nu^h:=(\hat\beta^h,\hat\gamma^h)\in \mathcal V^h$ and as soon as $(\hat X^h)_h$ is tight in $D([0,T];\mathbb R)$ then by continuity of the feedback functions $\psi^b,\psi^g$, the family $\hat\nu^h$ is also tight in $D([0,T];B\times\Gamma)$.
\end{remark}

The main result of the section is given below and provides tightness properties of the processes involved into the problem.

\begin{proposition}
Let $\hat\nu^h=(\hat\beta^h,\hat\gamma^h)\in \mathcal V^F$ and define
\[
\hat{\mathcal{H}}^h
:= \big(
\hat{T}^h,
\hat{\bar{\varsigma}}^h,\hat{\lambda}^h,\hat{N}^h, \hat{X}^h,\hat{H}^h,\hat{B}^h,\hat{M}^h,\hat{R}^{+,h}, \hat{R}^{-,h};\hat\tau^h,\hat{\xi}^h,\hat{\beta}^h,\hat{\gamma}^h,\emph{m}^h
\big).
\]
Then the family of random elements $
\big\{\hat{\mathcal{H}}^h\,:\, h>0\big\}$
is tight in the product space
\begin{align*}
&[0,\mathcal S^h]\times  D([0,\mathcal S^h]; \mathbb R^n_+\times \mathbb R_+\times \mathbb R_+\times \mathbb R\times\mathbb R\times \mathbb R\times \mathbb R\times \mathbb R_+\times \mathbb R_+)\\
&\times  [0,\mathcal S^h]
\times D([0,\mathcal S^h];\mathbb R\times B\times \Gamma)\times \mathcal M(\mathbb R_+).
\end{align*}

\end{proposition}

\begin{proof}
First, we will prove the tightness of $\hat{\bar{\varsigma}}^h$ using Aldous' tightness criterion: for every $\eta>0$, 
\begin{equation}\label{aldous}\lim_{\delta\downarrow 0}\;
\sup_{h>0}\;
\sup_{\hat{\tau}\le \mathcal S^h}
\mathbb{P}\!\left(
\big|\hat\varsigma^{(k),h}_{\hat{\tau}+\delta}-\hat\varsigma^{(k),h}_{\hat{\tau}}\big|>\eta
\right)
=0,\end{equation}
where the supremum in $\hat{\tau}$ is over all $\mathcal{F}^h$-stopping times bounded by $S$. Recall that
\begin{equation}\label{eq:varsigmaproof}
    \hat{\varsigma}^{(k),h}_t = \hat{\varsigma}^{(k),h}_se^{-q_k(t-s)} + \int_s^t\int_{\mathbb{R}_+} \varrho(z)e^{-q_k(t-u)}\hat{N}^h(du, dz).
\end{equation}
Therefore,
$$\mathbb{E}\big[\hat\varsigma^{(k),h}_t\big] = \mathbb{E}\big[\varsigma^{(k),h}_0\big]e^{-q_kt} + \mathbb{E}\Bigg[\int_0^t\int_{\mathbb{R}_+}\varrho(z)\text{m}^h(dz)du\Bigg].$$
From Assumption \textbf{(R)} we get
\begin{align*}
    \mathbb{E}\big[\hat\varsigma^{(k),h}(t)\big] &\leq  \mathbb{E}\big[\varsigma^{(k),h}_0\big] + \bar{\varrho}\int_0^t\mathbb{E}[\lambda_\infty(u) + \bar{d}\cdot\mathbb{E}\big[\hat{\bar{\varsigma}}^h_u]\big]du.
\end{align*}
Note that $(\mathbb E\big[\hat\varsigma^{(k),h}_t\big])_{k\leq n}$ satisfies a system of Gr\"onwall–Volterra inequalities. Consequently there exists a constant $C_1$ such that
$$\sup_h\sup_{t\in[0,\mathcal{S}^h]} \mathbb{E}\Big[\big|\hat\varsigma^{(k),h}_t\big|\Big] \leq C_1 < \infty.$$
From Markov's inequality, we then deduce that for any time $t$ and $\zeta>0$, there exists a constant $\kappa_1$ large enough so that
$$\sup_h\mathbb{P}(|\hat\varsigma^{(k),h}_t|>\kappa_1)\leq \zeta.$$

Note that for any $0<u\le\delta$ and $\hat{\tau}\leq \mathcal S^h$ we have, from \eqref{eq:varsigmaproof},
$$\hat\varsigma^{(k),h}_{\hat{\tau}+u}-\hat\varsigma^{(k),h}_{\hat{\tau}}
=
\hat\varsigma^{(k),h}_{\hat{\tau}}\big(e^{-q_k u}-1\big)
\;+\;
D^{(k),h}_{\hat{\tau},\hat{\tau}+u}
\;+\;
M^{(k),h}_{\hat{\tau},\hat{\tau}+u},$$
where $D^{(k),h}$ is the compensator/drift and $M^{(k),h}$ is the martingale terms of the increment $\hat\varsigma^{(k),h}_{\hat{\tau}+u}-\hat\varsigma^{(k),h}_{\hat{\tau}}$. We study the convergence of these three terms separately. 
\begin{enumerate}[label=(\arabic*)]
    \item First, recall that $\mathbb{E}\!\left[\;\big|\hat\varsigma^{(k),h}_{\hat{\tau}}\big|\;\right] \le C_{1}.$
Hence
$$\mathbb{E}\!\left[
\big|\hat\varsigma^{(k),h}_{\hat{\tau}}\big|\cdot\big|e^{-q_k u}-1\big|
\right]
\le
C_{1}\,\big|1-e^{-q_k u}\big|\to 0 \text{ as $u\downarrow 0$ uniformly in $h$ and $\hat{\tau}$ }.$$
\item  Next, from Assumption \textbf{(R)}, we have
$$\mathbb{E}\!\left[\;\big|D^{(k),h}_{\hat{\tau},\hat{\tau}+u}\big|\;\right]
\le
\bar\varrho\;\mathbb{E}\!\left[\int_{\hat{\tau}}^{\hat{\tau}+u}\lambda^h_s\,ds\right].$$
By the uniform first-moment bound on $\bar\varsigma^h$ and linearity of $\lambda^h$,
$$\sup_{h>0}\sup_{s\le \mathcal S^h}\mathbb{E}[\lambda^h(s)] <\infty.$$
Thus, there exists a constant $C_{2}$ such that
$$\mathbb{E}\!\left[\;\big|D^{(k),h}_{\hat{\tau},\hat{\tau}+u}\big|\;\right]
\le
C_{2}\,u.$$
Consequently, $\mathbb{E}\!\left[\;\big|D^{(k),h}_{\hat{\tau},\hat{\tau}+u}\big|\;\right]\to 0$ as $u\downarrow 0$ uniformly in $h$ and $\hat{\tau}$.
\item The predictable quadratic variation of $M^{(k),h}$ satisfies
$$\big\langle M^{(k),h}\big\rangle_{\hat{\tau},\hat{\tau}+u}
=
\int_{\hat{\tau}}^{\hat{\tau}+u}\!\!\int_{\mathbb{R}_+}
\big(\varrho^{(k)}(z)\big)^{2}
e^{-2q_k(\hat{\tau}+u-s)}
\,\lambda^h_s \text{m}(dz)\,ds.$$
Using $\sup_{s\le \mathcal S^h}\mathbb{E}[\lambda^h_s]<\infty$
and $\mathbb{E}[(\varrho^{(k)}(z))^{2}]<\infty$, we deduce that there exists a positive constant $C_{3}$ such that
$$\mathbb{E}\!\left[\,\big\langle M^{(k),h}\big\rangle_{\hat{\tau},\hat{\tau}+u}\,\right]
\le
C_{3}\,u.$$
From Burkholder-Davis-Gundy inequality, we get
$$\mathbb{E}\!\left[\;\big|M^{(k),h}_{\hat{\tau},\hat{\tau}+u}\big|\;\right]
\le
\mathbb{E}\!\left[\big\langle M^{(k),h}\big\rangle_{\hat{\tau},\hat{\tau}+u}\right]^{1/2}
\le
C_{3}\,u^{1/2}.$$
\end{enumerate}
Combining the three points above, for any $0<u\le\delta$,
\[\sup_{h>0}\;\sup_{\hat{\tau}\le S}
\mathbb{E}\!\left[
\big|\hat\varsigma^{(k),h}_{\hat{\tau}+u}-\hat\varsigma^{(k),h}_{\hat{\tau}}\big|
\right] \to 0 \text{ as $u\downarrow 0$}\]
and the convergence is uniform in $\hat{\tau}$ and $h$. Equation \eqref{aldous} follows from Markov's inequality. Then, according to \cite[Corollary 3.7.4]{ethier2009markov} the component process $\hat\varsigma^{(k),h}$ is tight in $D([0,\mathcal{S}];\mathbb{R}_+)$. It implies the tightness of the vector-valued process $\hat{\bar{\varsigma}}^{\,h}$ in $D([0,\mathcal{S}];\mathbb{R}^n_+)$.\vspace{0.5em}

Since $\hat{\lambda}^h$ is a fixed linear combination of $\hat{\bar{\varsigma}}^h$, the tightness of $\hat{\lambda}^h$ is a direct consequence of the tightness of $\hat\varsigma^{(k),h}$.\vspace{0.5em}

The proof of the tightness of $\{\hat{R}^{+,h}, \hat R^{-,h}\}$ is similar than \cite[Theorem 5.3]{kushner1991numerical}.\vspace{0.5em}

 To prove the tightness of $\hat X^h,\hat B^h,\hat M^h,$ and $\hat H^h$, we must first introduce the continuous interpolation of the solution $\hat{Y}$ to the SDE without singular control \eqref{eq:non-sing_dyn}, with $\hat{Y}^h_t:=\hat{Y}^h_{i^h(t)}$. so that
\[\hat{Y}^h_t = Y_0 + \hat{B}^h_t + \hat{M}^h_t + \hat{H}^h_t + \hat{R}^{+,h}_{t}- \hat{R}^{-,h}_{t}.\] 
Then, there exists a constant $c>0$ such that
\begin{align*}
    \mathbb{E}\left[ \sup_{t\leq T} \left\lvert \hat{Y}^h_t \right\rvert^2 \right] &\leq c\mathbb{E}\left[ |Y_0|^2 + \sup_{t\leq T} \left\lvert \int_0^{t} \mu(\hat{Y}^h_s)d\hat{T}^h(s) \right\rvert^2\right]  \\
    &\quad + c\mathbb{E}\left[ \sup_{t\leq T} \left\lvert \int_0^{t} \sigma (\hat{Y}^h_s, \hat{\beta}^h_s)d\hat{W}^h_s \right\rvert^2 \right] \\
    &\quad + c\mathbb{E}\Bigg[\sup_{t\leq T}\left\lvert\int_0^{t}\int_{\mathbb{R}_+}\chi( \hat{Y}^h_s, \hat{\gamma}^h_s,z) \Big(\lambda_\infty(s)+\bar{d}\hat{\bar{\varsigma}}^\top_s \Big)\text{m}^h(dz)d\hat{T}(s) \right\rvert^2 \Bigg] \\
    &\quad + c\mathbb{E}\Bigg[\sup_{t\leq T} \left\lvert\hat{R}^{+,h}_{t}\right\rvert^2 +\left\lvert\hat{R}^{-,h}_{t}\right\rvert^2\Bigg] +E_T^h,\\
\end{align*}
where $E$ is a term induced by the interpolation such that $\lim_{h\to 0}E^h_T = 0 .$ By using Assumption \textbf{(L)} together with the tigthness of $\hat R^h_+,\hat R^h_-$, boundedness of $\hat\beta^h$ and $\hat\gamma^h$, and Burkholder-Davis-Gundy for the martingale term, there exists a constant $\tilde c>0$ such that 

    $$\mathbb{E}\left[ \sup_{t\leq T} |\hat{Y}^h_t|^2 \right] \leq \tilde c\Big(1+\mathbb{E}\left[ \int_0^T \sup_{u\leq s} |\hat{Y}^h_u|^2d\hat{T}^h(s) \right] \Big).$$
Recall that $\hat{T}(\cdot)$ is an adapted process,  we deduce from Fubini's theorem that
    $$\mathbb{E}\left[ \sup_{t\leq T} |\hat{Y}^h_t|^2 \right] \leq  \tilde c\Big(1+\int_0^T \mathbb{E}\left[\sup_{u\leq s} |\hat{Y}^h_u|^2 \right] d\mathbb{E}[\hat{T}^h(s)]\Big) \leq \tilde c\Big(1+\int_0^T \mathbb{E}\left[\sup_{u\leq s} |\hat{Y}^h_u|^2 \right] ds\Big) .$$
Notice that $\tilde c$ depend on $\mathcal{S}^h$ and thus on $\sup_h \mathbb{E}[\hat{T}^h(T)]$ but not directly on $h$. Applying Gronwall's lemma, we get 
$$\sup_h\mathbb{E}\Bigg[ \sup_{t\leq T} |\hat{Y}^h_t|^2 \Bigg] < \infty.$$
As a consequence of de la Vallée Poussin criterion, $\{\hat{Y}^h\}$ is uniformly integrable. Thus, $\{\hat{Y}^h\}$ is tight in $D([0,\mathcal S^h];\mathbb R)$. By the definition of the time rescaling, $|\hat{\xi}^h_{t + s} - \hat{\xi}^h_t| \leq |s| + o(h)$. 
Consequently, $\{ \hat{\xi}^h(\cdot)\}$ is tight. 
Since $\{\hat{T}^h(\cdot)\}$ is bounded, it is uniformly integrable and consequently it is tight.\vspace{0.5em}

Note that the tightness of $\hat{N}^h$ follows from the fact that the mean of the number of jumps on any interval $[t, t+s]$ is a linear function of $\hat\lambda^h$ which is tight. Hence, $\hat{N}^h$ is tight, so is $\hat{H}^h$.\vspace{0.5em}

Note that ${\text{m}}^h$ is tight by \cite[Theorem 3.2]{parthasarathy2005probability} because the compact support property we assumed for ${\text{m}}^h$ is inherited by its discretization. The tightness of the other terms also follows from \cite[Theorem 5.3]{kushner1991numerical} combined with \cite[Lemma 23.12]{kallenberg2021foundations}.\vspace{0.5em}

Therefore, since $\hat{B}^h, \hat{M}^h, \hat{\xi}^h,\hat{H}^h$, and $\hat{R}^{\pm,h}$ are tight and $\hat{X}^h$ is the sum of these tight processes, $\{ \hat{X}^h\}$ is also tight. The tightness of the stopping time $\hat\tau^h$ in $[0,T]$ is a direct consequence of the finite time horizon setting. Finally, $\hat\nu^h:=(\hat\beta^h,\hat\gamma^h)$ is tight from Remark \ref{rm:tight}.
\end{proof}

As a consequence of this result together with Prokhorov’s theorem, we have the following corollary.
    \begin{corollary}\label{subseqcv}
        There exists a non-increasing sequence $(h_n)_n$ converging to $0$ when $n$ goes to $\infty$ and a limit random  element $\mathcal H$ such that 
        \[\lim\limits_{n\to\infty} \hat{\mathcal H}^{h_n}=\hat{\mathcal H}:= \big(
\hat{T},
\hat{\bar{\varsigma}},\hat{\lambda},\hat{N}, \hat{X},\hat{H},\hat{B},\hat{M},\hat{R}^{+}, \hat{R}^{-};\hat\tau,\hat{\xi},\hat{\beta},\hat{\gamma},\hat{\emph{m}}
\big)\; a.s.. \]

    \end{corollary}
We now identify this limit and study the convergence of the MDP $\hat X^h$ to the solution to the SDE \eqref{SDE:singular}. We extend the proof framework laid out in \cite[Section 5]{budhiraja2007convergent} to our model and get the following result.
\begin{theorem}\label{thm:limit_point_prop}
    For each fixed sequence $\hat{\beta},\hat{\gamma}\in \mathcal V^F$, the limit point $\hat{\mathcal{H}}$ has the following properties:
    \begin{enumerate}
        \item $\hat{T}$ is nondecreasing and Lipschitz continuous with Lipschitz coefficient 1.
        \item For all $t\in [0,T]$, $\hat{B}_t = \int_0^{t} \mu( \hat{X}_s)d\hat{T}(s)$.
        \item $(\hat{M}_t)_{t\leq T}$ is a continuous $(\hat{\mathcal{F}}_t)_{t\leq T}$-martingale with quadratic variation given by 
        \begin{equation}\label{QV}\langle \hat{M} \rangle_t = \int_0^t \left\lvert\sigma(\hat{X}_s, \hat{\beta}_s) \right\rvert^2 d\hat{T}(s), t\in [0,T].\end{equation}
        \item $\hat{\xi}$ is nondecreasing and continuous.
        \item $\hat{R}^+, \hat{R}^-$ are nondecreasing continuous processes which satisfy 
            \begin{equation*}
                \int_0^{\hat\tau\wedge T} \textbf{1}_{X^L_t > -L }d\hat{R}_t^+ = 0 \quad \quad \int_0^{\tau\wedge T}\textbf{1}_{ X^L_t < L }d\hat{R}^-_t= 0
            \end{equation*}
        \item For all $t\in[0,T]$, $\hat{H}_t = \int_0^{t}\int_{\mathbb{R}_+}\int_{\mathbb{R}_+}\chi(\hat{X}_s, \hat{\gamma}_s, z)\textbf{1}_{\theta\leq \lambda_\infty(s) + \bar{d}\bar{\varsigma}^\top_s}\Pi(ds,d\theta, dz)$.
        \item $\hat{X}$ satisfies $\mathbb{P}(\hat{X}_t\in \mathbb{S}_L)=1$ for all $t\in[0,T]$ and
            \begin{equation*}
                \hat{X}_t = x + \hat{B}_t + \hat{M}_t +\hat{\xi}_t +\hat{H}_t+\hat R^+_t-\hat R^-_t.
            \end{equation*}
    \end{enumerate}
\end{theorem}

\begin{proof}

\begin{enumerate}
\item From Definition \ref{def:rescale_hatT}, the derivative of $\hat{T} \geq 0$ over the entire domain, so $\hat{T}$ is nondecreasing.
Additionally, since $\frac{\hat{T}(y) - \hat{T}(x)}{y - x} \in \{ 0, 1\}$ for $y$ close to $x$ (i.e., adjacent), it is obvious that $|\hat{T}(y) - \hat{T}(x)| \leq |y - x|$ for all such $y$.
Thus, $\hat{T}$ is Lipschitz continuous with coefficient 1.

\item Since $(\hat{X}^h, \hat{T}^h) \to (\hat{X}, \hat{T})$ almost surely, we note that $\int_0^{t} \mu( \hat{X}^h_s)d\hat{T}^h(s) \to \int_0^{t} \mu( \hat{X}_s)d\hat{T}(s)$ almost surely.
As a consequence of the dominated convergence theorem, since $\mu$ satisfies \textbf{(L)} together with Lemma \ref{lm: X_moments} we get the convergence in 2.

\item We show $\hat{M}$ has continuous paths and is an $( \hat{\mathcal{F}}_t)_{t\leq T}$-martingale with quadratic variation \eqref{QV}. From \eqref{eq:local_consistency_E} and \eqref{eq:local_consistency_cov} we deduce that there exists $\zeta \in (0,\infty)$ such that for all $h\geq 0$, $u\in(0,T]$, \[\sup_{t\leq u} | \hat{M}^h_t - \hat{M}^h_{t-}|\leq 2\zeta h.\]
Consequently, $\int_0^T e^{-u} (1 \wedge \sup_{t\leq u} | \hat{M}^h_t - \hat{M}^h_{t-}|)du \leq 2\zeta h$, for $h$ small enough. 
Therefore,
    $$\lim_{h\to 0} \int_0^T e^{-u} (1 \wedge \sup_{t\leq u} | \hat{M}^h_t - \hat{M}^h_{t-}|)du = 0.$$
Combining this convergence with $\hat{M}^h \to \hat{M}$, we deduce from \cite[Theorem 3.10.2]{ethier2009markov} that $\hat{M}$ has continuous paths.
We now note that the quadratic variation of $\hat{M}^h$ is
    \begin{equation*}\label{eq:pf_quad_var}
        \langle \hat{M}^h\rangle (t) = \int_0^{t } \left \lvert \sigma(\hat{X}^h_s, \hat{\beta}^h_s )\right \rvert^2 d\hat{T}(s) + \hat{\varepsilon}^h_t,
    \end{equation*}
with $\hat{\varepsilon}^h_t$ satisfying 
    \begin{equation*}
        \lim_{h\to 0} \mathbb{E}\left[ \sup_{u\leq T} |\hat{\varepsilon}^h_u|^\ell \right] = 0
    \end{equation*}
for all $\ell\geq 1$.
From Burkholder-Davis-Gundy inequality, \textbf{(L)} together with Lemma \ref{lm: X_moments}, there exists some constant $\alpha>0$ such that:
    \begin{equation*}
        \mathbb{E}[] | \hat{M}^h_t|^2] \leq \alpha \left[  T^2 + \mathbb{E}[\sup_{0 \leq u \leq t} |\hat{\varepsilon}^h_t|^2] \right].
    \end{equation*}
So $\{ (\hat{M}^h_t), h >0 \}$ is uniformly integrable and from [Theorem 3.10.2 in \cite{ethier2009markov}] $\hat{M}$ is consequently an $\{ \hat{\mathcal{F}}_t\}$-martingale with quadratic variation given by \eqref{QV}.

\item  The fourth property follows from $\hat{\xi}^h$ being nondecreasing and $| \hat{\xi}^h_{t+s} - \hat{\xi}^h_t | \leq s + h$.

\item Since $\mathbb{P}( X^h_{n+1}=x|\mathcal{F}^h_{n}) = \mathbb{P}( X^h_{n+1}=x|X^h_{n}, \mathcal{F}^h_{n}) = p^h(X^h_{n},x|\pi^h_{n})$, $(\hat{X}, \hat{R})$ in the reflection definitions can be replaced with $(\hat{X}^h, \hat{R}^h)$. Hence, since $\hat{X}^h\in [-L, L]$ the statement follows from the definition of $\hat{X}^h$ and Skorohod map properties in Definition \ref{def:2side_Skorokhod_map}.

\item This point directly follows from discretization of the $\Pi$-measure stochastic integral noting that the Voronoi-based discretization of $\text{m}^h$ weakly converges to $\text{m}$ and that $\chi$ is continuous together with the tightness properties of $\hat{X}^h$ and $\hat\gamma^h$.
\item As a consequence of all the previous points, we get the desired decomposition for the limit point $\hat X$. 
\end{enumerate}
\end{proof}

We want to transfer this result to the primal processes and especially the convergence of $X^h$ toward $X$. To achieve this goal, we define the inverse time process as follow
\[\overline{T}(t) = \inf\{ s \ : \ \hat{T}(s) > t \}.\]
We first give some properties of this process and refer to \cite[Equation (46)]{budhiraja2007convergent} for the proof of these results. 
\begin{lemma}[Properties of $\overline T$]\label{lemmabarT}For any $t\in [0,T]$, the process $\overline T$ is nondecreasing, right-continuous, converges to $T$ when $t$ goes to $T$ and satisfies the following properties 
\begin{enumerate}
\item $t\leq \overline{T}(t) \leq T$,
\item $\hat{T}(\overline{T}(t))=t$ and $\overline{T}(\hat{T}(t)) \geq t$
\item $\hat{T}(s)\in [0,t] \Leftrightarrow s\in[0, \overline{T}(t)].$
\end{enumerate}
\end{lemma}
We now similarly introduce the generalized inverse function of $\hat T^h$, the discretized version of $\hat T$, as follows
$$\bar{T}^h(u):=\inf\{s\in[0,T]:\ \hat T^h(s)>u\}.$$ This process satisfies the same properties of $\bar T$ presented above in Lemma \ref{lemmabarT}. Additionally, we have the following convergence result for this generalized inverse.
\begin{lemma}[Convergence properties of the generalized inverse]\label{lem:inverse_convergence}

Suppose that $\mathcal{S}^h\to \mathcal{S}$ as $h\downarrow0$ and that $\hat T^h\to\hat T$ uniformly on $[0,T]$.  
Then:
\begin{enumerate}
    \item For every $u\in[0,\mathcal{S}]$ which is \emph{not} a discontinuity point of $\bar{T}$ (equivalently a point where $\hat T$ is strictly increasing on some neighborhood), we have
    $$\bar{T}^h(u)\to \bar{T}(u)\quad\text{as }h\downarrow0.$$
    \item If, moreover, $\hat T$ is strictly increasing (hence $\bar{T}$ is continuous) on $[0,\mathcal{S}]$, then $\bar{T}^h\to \bar{T}$ uniformly on $[0,\mathcal{S}]$ (after extending $\bar{T}^h$ by $\bar{T}^h(u)=T$ for $u\in(\mathcal{S}^h,\mathcal{S}]$ when $\mathcal{S}^h<\mathcal{S}$).
\end{enumerate}
\end{lemma}

\begin{proof}
\begin{enumerate}

\item  Fix $u\in[0,\mathcal{S}]$ which is not a discontinuity point of $\bar{T}$. Since $\hat T^h\to\hat T$ uniformly and $\mathcal{S}^h\to \mathcal{S}$, for any $\varepsilon>0$ there exists $h_1>0$ such that for all $0<h<h_1$ we have $\sup_{s\in[0,T]}|\hat T^h(s)-\hat T(s)|<\varepsilon$ and $|\mathcal{S}^h-\mathcal{S}|<\varepsilon$. By the standard definition of the inverse of a monotone function, one can show
$$\hat T\big(\bar{T}(u)-\omega\big)\le u \le \hat T\big(\bar{T}(u)+\omega\big)$$
for all sufficiently small $\omega>0$. Uniform convergence of $\hat T^h$ implies the inequalities
$$\hat T^h\big(\bar{T}(u)-\omega\big)\le u + o(1),\qquad
\hat T^h\big(\bar{T}(u)+\omega\big)\ge u - o(1),$$
and these inequalities imply $\bar{T}^h(u)\in [\bar{T}(u)-\omega,\,\bar{T}(u)+\omega]$ for all small $h$. Letting $\omega\downarrow0$ yields $\bar{T}^h(u)\to \bar{T}(u)$.

\item  If $\hat T$ is strictly increasing then $\bar{T}$ is continuous on $[0,\mathcal{S}]$. Uniform convergence $\hat T^h\to\hat T$ and $\mathcal{S}^h\to \mathcal{S}$ then imply equicontinuity and uniform convergence of the inverses. Concretely, for any $\varepsilon>0$ choose $\delta>0$ so that $|s-s'|<\delta$ implies $|\hat T(s)-\hat T(s')|<\varepsilon$, then use uniform convergence to transfer this modulus to $\hat T^h$ uniformly in $h$, and the standard inverse-function inequality (monotonicity-based) to conclude the desired uniform convergence of $\bar{T}^h$ to $\bar{T}$ on $[0,\mathcal{S}]$.
\end{enumerate}
\end{proof}
This lemma induces the convergence of the compressed process $\tilde X_t^h:=\hat X^h_{\bar T(t)}$, as below.
\begin{lemma}[Convergence of compressed discrete processes]\label{cor:compressed_convergence}
Suppose that $\hat X^h \to \hat X$ uniformly on $[0,\mathcal S]$. Then $\tilde X^h \to \tilde X:= \hat X\circ \bar{T}
\quad\text{uniformly on }[0,T].$
\end{lemma}

\begin{proof}
From Lemma \ref{lem:inverse_convergence}, $\bar{T}^h\to \bar{T}$ uniformly. By assumption, $\hat{X}^h\to \hat{X}$ uniformly on the stretched domain. Standard composition results for uniform convergence of continuous functions imply
$
\hat X^h \circ \bar{T}^h \to \hat X \circ \bar{T},$
uniformly on $[0,T]$.
\end{proof}
Finally, the limit process $\tilde X$ as a c\`adl\`ag semi-martingale, admits a unique decomposition of the form
\begin{equation}\label{sdetildeX}\tilde X_t=x+\tilde B_t+\tilde M_t+\tilde\xi_t+\tilde H_t+\tilde R_t^+-\tilde R_t^-,\end{equation}
where $\tilde{B}$ is an absolutely continuous predictable process, $\tilde{\xi}$ and $\tilde{H}$ are nondecreasing, adapted, with finite variations, $\tilde{M}$ is a local martignale and $\tilde{R}^+,\tilde{R}^-$ solves the Skorokhod problem. We have the following identification theorem.

\begin{theorem}\label{limitthm} Considering the semi-martingale decomposition \eqref{sdetildeX} we have
    \begin{enumerate}
        \item $\tilde B_t = \int_0^{t} \mu( \tilde X_s)ds $.
        \item $\tilde M$ is a continuous $\{ \mathcal{F}_t \}$-martingale with quadratic variation
        \begin{equation*}
            \langle \tilde M \rangle_t = \int_0^{t} |\sigma( \tilde X_s, \beta_s)|^2ds .
        \end{equation*}
        There exists an enlargement of the probability space $(\Omega, \mathcal{F}, \mathbb{P})$ and of the filtration $\{ \mathcal{F}_t\}$ that supports a Wiener process $W$, which is a martingale with respect to the enlarged filtration and such that 
        \begin{equation*}
            \tilde M_t = \int_0^{t} \sigma(\tilde X_s, \tilde\beta_s)dW_s.
        \end{equation*}
      \item $\tilde{R}^+, \tilde{R}^-$ are nondecreasing continuous processes which satisfy 
            \begin{equation*}
                \int_0^{\hat\tau\wedge T} \textbf{1}_{\tilde X_t > -L }d\tilde{R}_t^+ = 0 \quad \quad \int_0^{\tau\wedge T}\textbf{1}_{ \tilde X_t < L }d\tilde{R}^-_t= 0
            \end{equation*}
        \item For all $t\in[0,T]$, $\tilde{H}_t = \int_0^{t}\int_{\mathbb{R}_+}\int_{\mathbb{R}_+}\chi(\tilde{X}_s, {\gamma}_s, z)\textbf{1}_{\theta\leq \lambda_\infty(s) + \bar{d}\bar{\varsigma}^\top_s}\Pi(ds,d\theta, dz)$.
    \end{enumerate}
\end{theorem}

\begin{proof}[Proof of Theorem \ref{limitthm}]
First, note that
\begin{equation*}
    \int \textbf{1}_{ [\bar{T}(\tau-), \bar{T}(\tau)]} d\hat{T}(s) = 0. 
\end{equation*}
Then, from Theorem \ref{thm:limit_point_prop} and Lemma \ref{lm:measure_theory}, we immediately get Property 1. Now recall that $\hat{M}$ is a continuous $\hat{\mathcal{F}}_t$-martingale.
Since $\hat{M}$ has continuous path and $\bar{T}(t)\leq T$ a.s., we see that $\lim_{n\to \infty}\hat{M}_{\bar{T}(t)\wedge n} = \hat{M}_{\bar{T}(t)} = \tilde M_t$ a.s. leading to Property 2 by Assumption \textbf{(L)} and Lemma \ref{lm: X_moments}. Properties 3 and 4 follow by similar argument than Theorem \ref{thm:limit_point_prop} by using Lemma \ref{lm:measure_theory} and the properties of $\chi$.
\end{proof}
As a direct consequence of \eqref{sdetildeX} and the uniqueness of the solution to \eqref{eq:constrained_dyn} we have the final corollary showing, in particular, the convergence of $X^h$ to $X$ and summarized in Figure \ref{pic:convergenceX}.  
\begin{corollary}\label{cor:cv}
For any $(\beta,\gamma,\tau,\xi)\in \mathcal V\times \Xi$ we have
\[\tilde{\mathcal X}_t=\mathcal X_t,\; \text{ for any } t\leq T,\quad \mathcal X=(X,H,B,M,H,R^+,R^-).\]
\end{corollary}

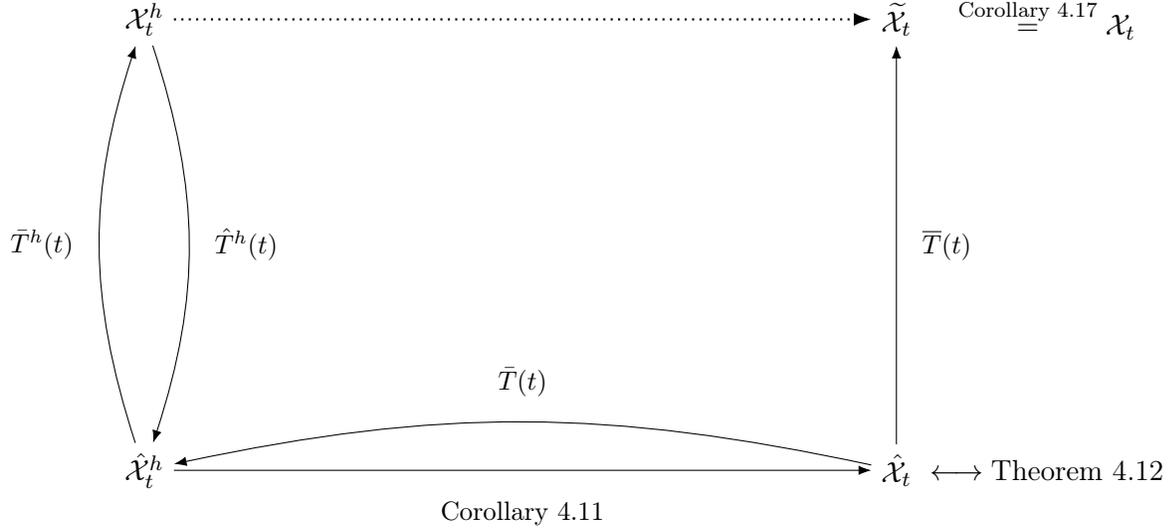
\begin{figure}[h]
\begin{tikzpicture}[
  >=Latex,
  every node/.style={font=\normalsize},
  lab/.style={font=\small},
  scale=1
]

% --------------------
% Nodes (wide layout)
% --------------------
\node (Xcalh)  at (0,6)   {$\mathcal X_t^h$};
\node (Xhath)  at (0,0)   {$\hat{\mathcal X}_t^h$};
\node (cal X)  at (12,6)   {$\overset{\text{Corollary 4.17}}{=}\mathcal X_t$};
\node (cal X)  at (12,0)   {$\longleftrightarrow \text{Theorem 4.12}$};
\node (Xtildeh) at (10,0) {$\hat{\mathcal X}_t$};
\node (Xtilde)  at (10,6) {$ \widetilde{\mathcal X}_t$};

\draw[->]
  (Xtildeh) -- node[lab, right=6pt] {$\overline T(t)$} (Xtilde);
% --------------------
% Left vertical arrows (gentle curves)
% --------------------
\draw[->, bend left=18]
  (Xhath) to
  node[lab, left=6pt] {$\bar T^h(t)$}
  (Xcalh);

\draw[->, bend left=18]
  (Xcalh) to
  node[lab, right=6pt] {$\hat T^h(t)$}
  (Xhath);

% --------------------
% Bottom horizontal arrow
% --------------------
\draw[->]
  (Xhath) --
  node[lab, below=8pt] {Corollary 4.11}
  (Xtildeh);

% --------------------
% Bottom return arrow
% --------------------
\draw[->, bend right=12]
  (Xtildeh) to
  node[lab, above=6pt] {$\bar T(t)$}
  (Xhath);

% --------------------
% Labels under nodes
% --------------------

% --------------------
% Top dotted double arrow
% --------------------
\draw[->, dotted, thick]
  (Xcalh) -- (Xtilde);

\end{tikzpicture}
\caption{Convergence of $\mathcal X^h$ where $\mathcal X$ represents any generic process $X,H,B,M,H,R^+,R^-$}\label{pic:convergenceX}
\end{figure}

\subsection{Convergence of the value functions}
\label{sec:convvalue}
We now turn to the convergence of the discrete time optimization problem \eqref{eq:Vh} to its continuous time version \eqref{controlpb-continuous} with respect to the time discretization parameter $h$.

\begin{theorem}[Convergence of the value function]\label{cvvaluefct}
Let $V_{0,L}(x, \bar{\varsigma})$ be the value function of the continuous bounded control problem defined in \eqref{eq:Vl}, and let $V_{0,L}^h(x^h, \bar{\varsigma}^h)$ be the value function of the discrete approximation defined in \eqref{eq:Vh} where $(x_h, \bar{\varsigma}_h)$ is a sequence of initial conditions in the discrete domain converging to $(x, \bar{\varsigma})$ as $h \to 0$. Then, there exists a subsequence $\{h_n\}_{n \ge 1}$ with $h_n \to 0$ as $n \to \infty$, such that the discrete value function converges to the continuous value function:
\begin{equation*}
    \lim_{n\to \infty} V_{0,L}^{h_n}(x_{h_n},\bar{\varsigma}_{h_n}) = V_{0,L}(x,\bar{\varsigma}).
\end{equation*}
\end{theorem}

\begin{proof}
We aim to prove that, up to a subsequence, the value function of the discrete bounded problem converges to the value function of the continuous bounded problem. Let $\{h_n\}$ be the subsequence for which the weak convergence of the processes has been established in Corollary \ref{subseqcv}. For simplicity, we denote the subsequence as $h$. We first recall from Corollary \ref{subseqcv} that the tuple of state and control processes converges weakly to a continuous limit process:
    \begin{equation*}
        (\bar{\varsigma}^h, X^h, u^h) \Rightarrow (\bar{\varsigma}^L, X^L, u),
    \end{equation*}
    where the limit processes satisfy the continuous bounded dynamics. Since $F, G, \phi,$ and $ \mathcal{K}$ satisfy \textbf{(Polx),(Poly)}, and \textbf{(Coef)}, by the dominated convergence theorem, we have the convergence of the expected payoffs:
    \begin{equation*}
        \lim_{h \to 0} J_{0,L}^h(x_h, \bar{\varsigma}_h; u^h) = J_{0,L}(x, \bar{\varsigma}; u).
    \end{equation*}The proof proceeds in two steps by establishing the upper and lower bounds.

\paragraph{Step 1:  $\limsup_{h \to 0} V_{0,L}^h \leq V_{0,L}$.}
For each $h$, let $u^h = (\nu^h, \tau^h, \xi^h)$ be an $\varepsilon$-optimal control sequence for the discrete problem starting at $(x_h, \bar{\varsigma}_h)$ defined by
    \begin{equation*}
        V_{0,L}^h(x_h, \bar{\varsigma}_h) \leq J_{0,L}^h(x_h, \bar{\varsigma}_h; u^h) + h.
    \end{equation*}

    The limit control $u$ belongs to the set of admissible controls for the continuous bounded problem. Therefore, the payoff achieved by this limit control cannot exceed the supremum (the value function):
    \begin{equation*}
        J_{0,L}(x, \bar{\varsigma}; u) \leq \sup_{\nu,\tau,\xi} J_{0,L}(x, \bar{\varsigma}; \nu,\tau,\xi) = V_{0,L}(x, \bar{\varsigma}).
    \end{equation*}

    Taking the limit supremum as $h \to 0$:
    \begin{equation*}
        \limsup_{h \to 0} V_{0,L}^h(x_h, \bar{\varsigma}_h) \leq \lim_{h \to 0} (J_{0,L}^h(u^h) + h) = J_{0,L}(u) \leq V_{0,L}(x, \bar{\varsigma}).
    \end{equation*}

\paragraph{Step 2: $\liminf_{h \to 0} V_{0,L}^h \geq V_{0,L}$.}
    Fix $\varepsilon > 0$. Let $u^* = (\nu^*, \tau^*, \xi^*)$ be an $\varepsilon$-optimal admissible control for the continuous bounded problem such that:
    \begin{equation*}
        J_{0,L}(x, \bar{\varsigma}; u^*) \geq V_{0,L}(x, \bar{\varsigma}) - \varepsilon.
    \end{equation*}

    We construct a sequence of discrete admissible controls $\hat{u}^h = (\hat{\nu}^h, \hat{\tau}^h, \hat{\xi}^h)$ approximating $u^*$ by setting $\hat{\nu}^h_i = \nu^*_{t_i}$, defining the discrete impulse $\hat{\xi}^h$ piecewisely such that its cumulative sum approximates $\xi^*$ and setting $\hat{\tau}^h = \min \{ t_k^h : t_k^h \geq \tau^* \}$.
Note that the state process $\hat{X}^h$ driven by $\hat{u}^h$ converges strongly (in $L^2$) to the optimal continuous trajectory $X^*$. Consequently, 
    \begin{equation*}
        \lim_{h \to 0} J_{0,L}^h(x_h, \bar{\varsigma}_h; \hat{u}^h) = J_{0,L}(x, \bar{\varsigma}; u^*).
    \end{equation*}
 Since now $\hat{u}^h$ is an admissible control for the discrete problem such that
    \begin{equation*}
        V_{0,L}^h(x_h, \bar{\varsigma}_h) \geq J_{0,L}^h(x_h, \bar{\varsigma}_h; \hat{u}^h),
    \end{equation*}

as $h \to 0$ we get
    \begin{equation*}
        \liminf_{h \to 0} V_{0,L}^h(x_h, \bar{\varsigma}_h) \geq \lim_{h \to 0} J_{0,L}^h(x_h, \bar{\varsigma}_h; \hat{u}^h) = J_{0,L}(x,\bar{\varsigma},u^*) \geq V_{0,L}(x, \bar{\varsigma}) - \varepsilon.
    \end{equation*}
    Since $\varepsilon$ is arbitrary, we have $\liminf_{h \to 0} V_{0,L}^h \geq V_{0,L}$.

\vspace{1em}
Combining the results from Step 1 and Step 2, we obtain:
\begin{equation*}
    V_{0,L}(x, \bar{\varsigma}) \leq \liminf_{h \to 0} V_{0,L}^h(x_h, \bar{\varsigma}_h) \leq \limsup_{h \to 0} V_{0,L}^h(x_h, \bar{\varsigma}_h) \leq V_{0,L}(x, \bar{\varsigma}).
\end{equation*}\end{proof}

By combining Theorem \ref{cvvaluefct} with Proposition \ref{prop:VL_to_V} we get the final convergence result. 
\begin{corollary}\label{cor:cvV}
    Let $(x_h, \bar{\varsigma}_h)$ be a sequence of initial conditions in the discrete domain converging to $(x, \bar{\varsigma})$ as $h \to 0$. Then, there exists a subsequence $\{h_n\}_{n \ge 1}$ with $h_n \to 0$ as $n \to \infty$, such that the discrete value function converges to the continuous value function:
\begin{equation*}
  \lim_{L\to\infty}  \lim_{n\to \infty} V_{0,L}^{h_n}(x_{h_n},\bar{\varsigma}_{h_n}) = V_{0}(x,\bar{\varsigma}).
\end{equation*}
\end{corollary}

\section{Application to energy project management with cyber risk: a controlled Ornstein–Ulhenbeck driven by Hawkes process}\label{secnumeric}

We turn to the application of our result motivated by cyber risk for power plant investment project. In this section, we model the project’s profitability state $X_t$ as a controlled mean-reverting factor. Mean reversion reflects the tendency of electricity-market fundamentals (and in particular de-seasonalized prices/margins) to revert due to physical balancing and supply response, for which Ornstein–Uhlenbeck specifications are standard see for example \cite{benth2008stochastic, lucia2002electricity}. We consider capital injection in this project modeled through a singular control $\xi$. To capture extreme events such as cascades outages, scarcity episodes or cyberattack on power grids (e.g. \cite{abraham2025cyber, case2016analysis}), we include accident-driven downward jumps as in mean-reverting jump-diffusion models for power prices, see for example \cite{cartea2005pricing}. modeled with a Hawkes process. In order to emphasize the impact of the singular control and the stopping time, we will fix $\beta=1$ and $\gamma=0$. The dynamics of $X$ considered is thus given by 
\begin{equation*}
    \begin{cases}
    &dX_t= (\alpha - \delta X_t)dt + \sigma dW_t + d\xi_t - dN_t \\
    &d\lambda_t=  -q(\lambda_t - a)dt + dN_t\\
     &N(dt,dz)=\int_0^\infty \mathbf 1_{\theta\leq \lambda_t } \Pi(dt,d\theta,dz),
\end{cases}
    \end{equation*}
where $\alpha,a,q,\delta,\sigma$ are positive constant. We note that in the equation for the intensity, $q$ is the decay rate and $a$ is the baseline intensity. Note that because we do no longer control the volatiltiy and the severity of attacks to focus on the singular-stopping problem, the cost function  $\mathcal{K}(s,x,\beta,\gamma,\lambda,\text{m})$ is zero.\newline

We model the stopping decision as sell-back flexibility. This acts as an insured floor, giving the manager the right to liquidate the project for a fixed insured value $I$. Mathematically, this is equivalent to holding an American put option on the project value X with strike I. If the project is never stopped, the manager receives the terminal value $G(X_T)$. Thus, the strategy acts as a hedge: as soon as the project value drops significantly below I, the manager exercises the option to minimize losses. We thus assume that at any time $\tau$, the manager may exit or sell the project and receive $I>0$ (insured floor, or liquidation value), seen as an American put option with strike $I$ on $X$. We set $F(X_\tau)=(I-X_\tau)^+$. If the project is not stopped and remains viable, the manager receives $G(X_T)=\frac{X_T^{1-\eta}}{\eta}$ at time $T$.\vspace{1em}

In the following, we set the  following values for the parameters in all simulations:
\begin{table}[h]
    \centering
    \begin{tabular}{cccccccccc}
    $r=0$ &
        $\alpha = 0.5$ & $\delta = 1$ &
        $\sigma = 0.1$ & $\eta = 0.1$  & $k=1$ & $a = 1$ & $T=1$
    \end{tabular}
\end{table}

\subsection{Convergence of the Hawkes-OU with a deterministic capital injection.}
First, we visually study the convergence of the probability model defined for the discrete-time Markov chain in Section \ref{sec:MCdisc}. Since we are just looking at the evolution of $X_t$, rather than the value function itself, we fix a constant jump size for $\xi=0.5$, which occurs at time $\tfrac{T}{2}$. 

\begin{figure}[H]
\centering 
    \centering
    \includegraphics[width=0.75\linewidth]{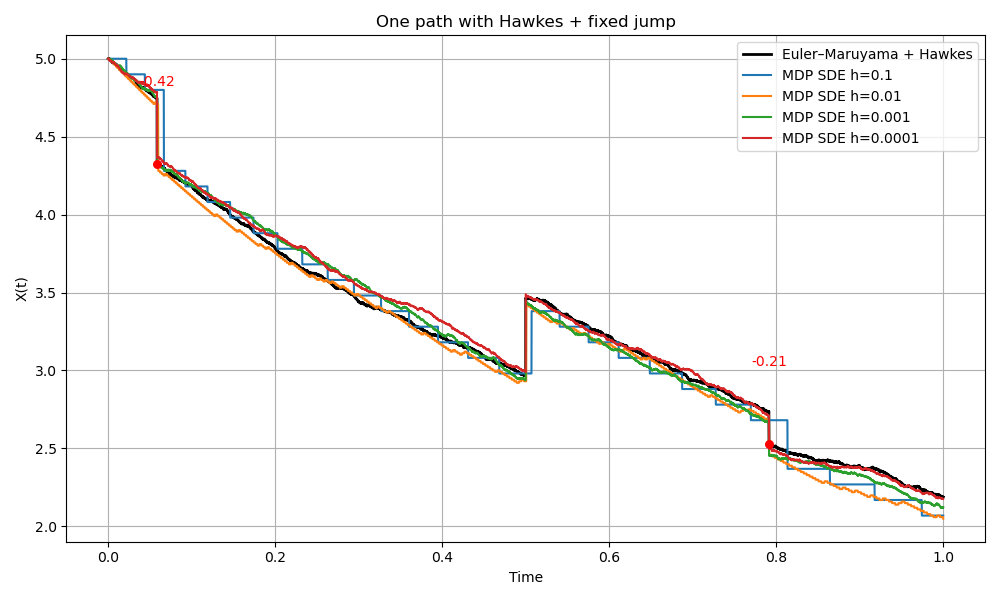}
\caption{Exact OU diffusion vs. DTMC approximations.}    \label{fig:SDE_DTMC_Hawkes}
\end{figure}

In Figure \ref{fig:SDE_DTMC_Hawkes}, we illustrate Corollary \ref{cor:cv} and see that when $h$ decreases towards 0, the discrete-time Markov chain approximation approaches the exact simulation of the Ornstein–Uhlenbeck process.

\subsection{Convergence of the discrete time optimization toward the continuous time}
We now illustrate the result of Theorem \ref{cor:cvV} and the convergence of the value function associated with the DTMC toward the value function in the continuous time setting for different values of $h$. Figure \ref{fig:conv_V} shows the convergence of $V^h$ with respect to both $x_0$ and $\lambda_0$ for different value of $h$ while Figure \ref{fig:tau_L_nosing} presents cross sections with respect to a fix value of $x_0$ or $\lambda_0$. We see that the value converges when $h$ goes to zero validating Corollary \ref{cor:cvV}. Notice that for $x_0$ big enough and close to 1, the value function seems to go to a ceiling value of $\frac1\eta$. It is explained by the fact that in this case, with this choice of parameters, the project is viable and never stops and the dominant term is $\frac{X_T^{1-\eta}}{\eta}\approx \frac1\eta$ since $X$ is capped at one in our simulation. We notice that the value of the project increases with respect to $x_0$ showing that the optimal injection of capital leads to a viable project. At the opposite, when the value of $\lambda_0$ increases, the project is more subjected to accidents leading to a decrease of the value function. 
\begin{figure}[H]
    \centering
    \includegraphics[width=0.75\linewidth]{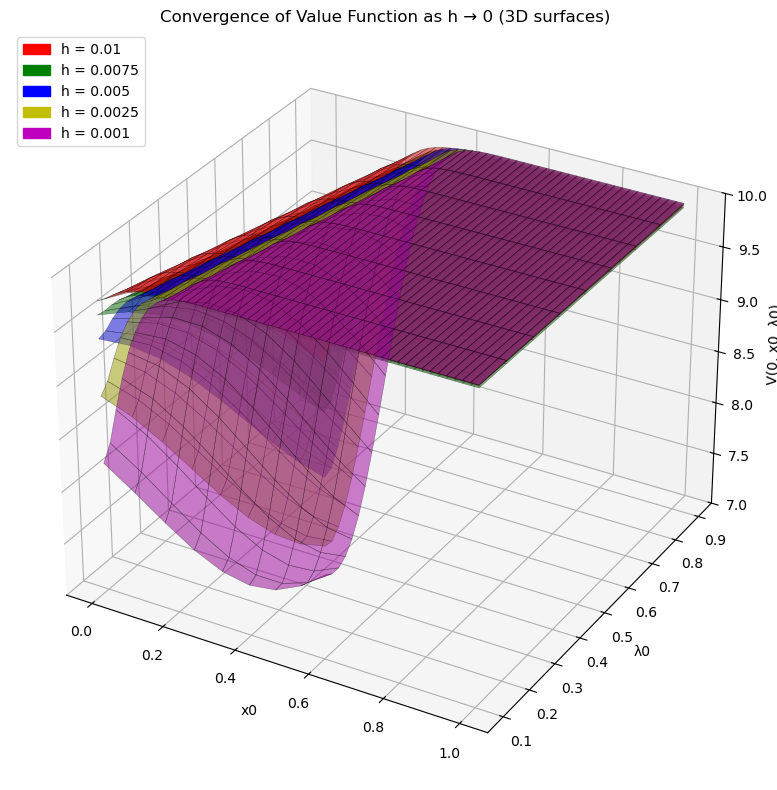}
    \caption{Value function convergence in $x_0$ and $\lambda_0$}
    \label{fig:conv_V}
\end{figure}

\begin{figure}[H]
\centering
\begin{subfigure}{0.48\textwidth}
    \centering
    \includegraphics[width=\linewidth]{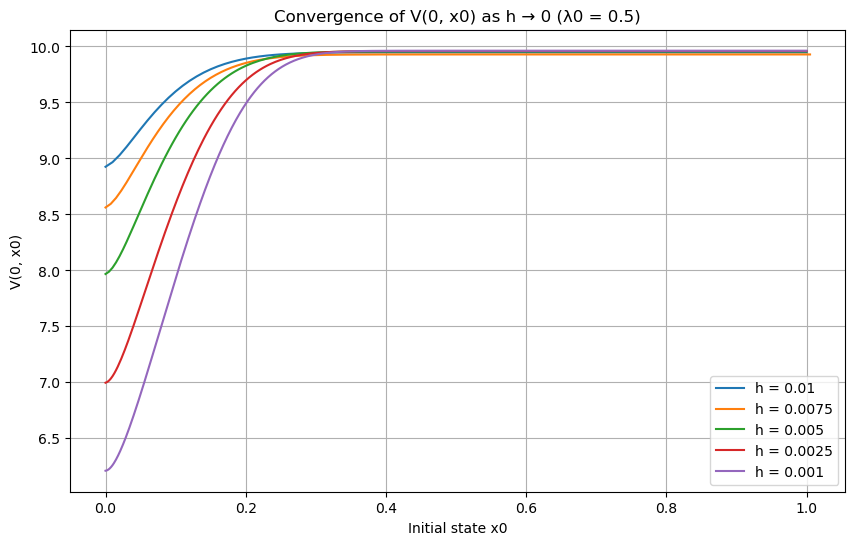}
    \caption{$V(0, x_0, \lambda_0=0.5)$}
    \label{fig:conv_h}
\end{subfigure}
\hfill
\begin{subfigure}{0.48\textwidth}
    \centering
    \includegraphics[width=\linewidth]{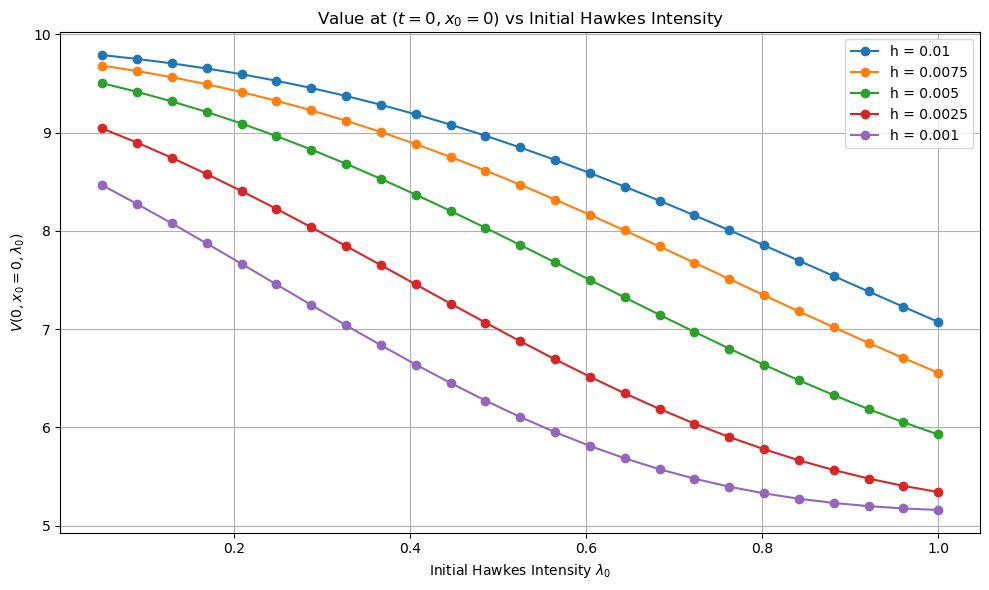}
    \caption{$V(0, x_0=0, \lambda_0)$}
    \label{fig:conv_L}
\end{subfigure}
\caption{Convergence of the value function with fixed initial parameters for different $h$.}
\end{figure}

\subsection{Optimal capital injection}
In Figure \ref{fig:singular} we display a sample path of $X$ and the value function $V^{h=0.005}_t(x_0=0,\lambda_0=1)$ alongside, with optimal capital injections (singular control) visualized on the x-axis. We observe that the manager injects capital heavily at the onset of the project to ensure viability and prevent early stopping, with activity increasing again as the project approaches maturity. This corresponds to a kind of rescue time at the beginning, then a wait and see period and finally a cost of potential ruin hedging at the end.

\begin{figure}[H]
    \centering
    \includegraphics[width=0.75\linewidth]{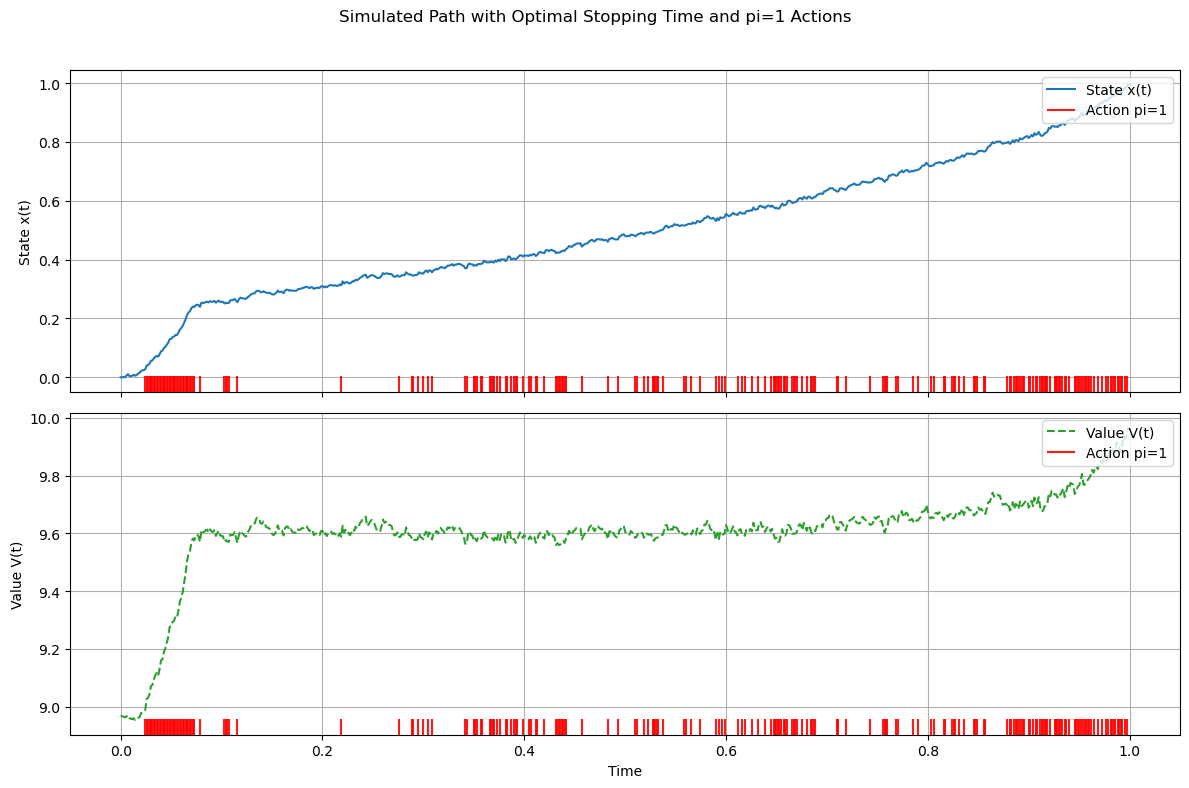}
    \caption{Top a path of $X, $bottom $V^{h=0.005}(x_0=0,\lambda_0=1)$ over time; red lines denote times of capital injection.}
    \label{fig:singular}
\end{figure}

\subsection{Singular control and optimal stopping}
Figure \ref{graphtaustar} shows how the optimal stopping time for the MDP when $h=0.005$ changes as the initial Hawkes intensity increases for a fixed $x_0$. Specifically, Figure \ref{fig:tau_L_nosing} displays how the MDP stops earlier as the initial Hawkes intensity increases when there is no singular control applied during the process. The intuition behind this result is that a higher $\lambda_0$ means there will be more jumps in the Hawkes process, which corresponds to more decreases in the value function. Consequently, without the ability to offset these decreases via injection using singular control, the value function tends to hit the stopping condition much earlier. In Figures \ref{fig:tau_L_midsing} and \ref{fig:tau_L_sing}, we take high and low costs for the singular control (relative to the problem), respectively. In both cases, compared to the uncontrolled problem, the project remains viable until maturity in the majority of scenarios, even for high $\lambda_0$ thereby avoiding early exit. This underscores the benefit of strategic capital injections amplified when the cost of intervention is reduced.
\begin{figure}[H]
    \centering
    % --- First Row (2 Images) ---
    \begin{subfigure}[b]{0.48\textwidth}
        \centering
        \includegraphics[width=\linewidth]{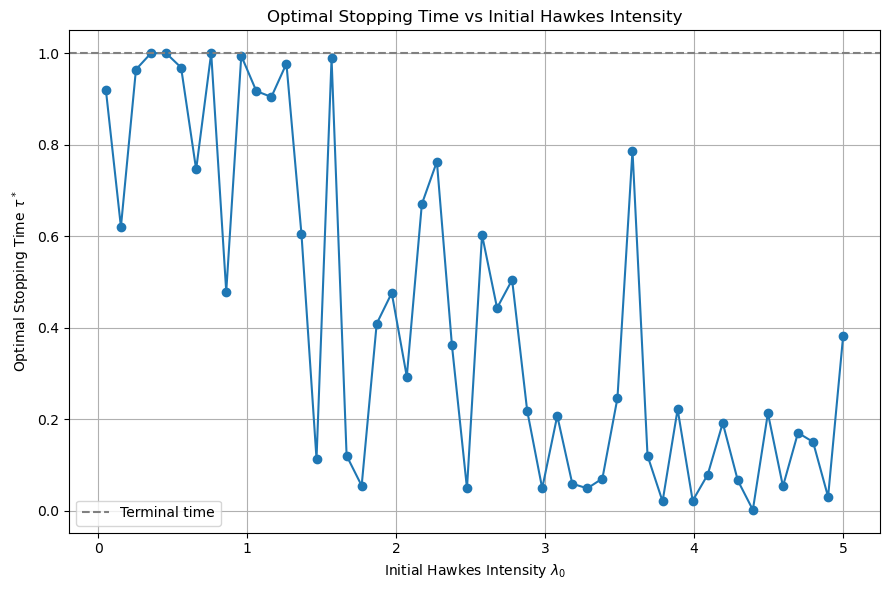}
        \caption{No singular control}
        \label{fig:tau_L_nosing}
    \end{subfigure}

    % --- Vertical Spacing ---
    \par\bigskip 
    
    % --- Second Row (1 Centered Image) ---

     \begin{subfigure}[b]{0.48\textwidth}
        \centering
        \includegraphics[width=\linewidth]{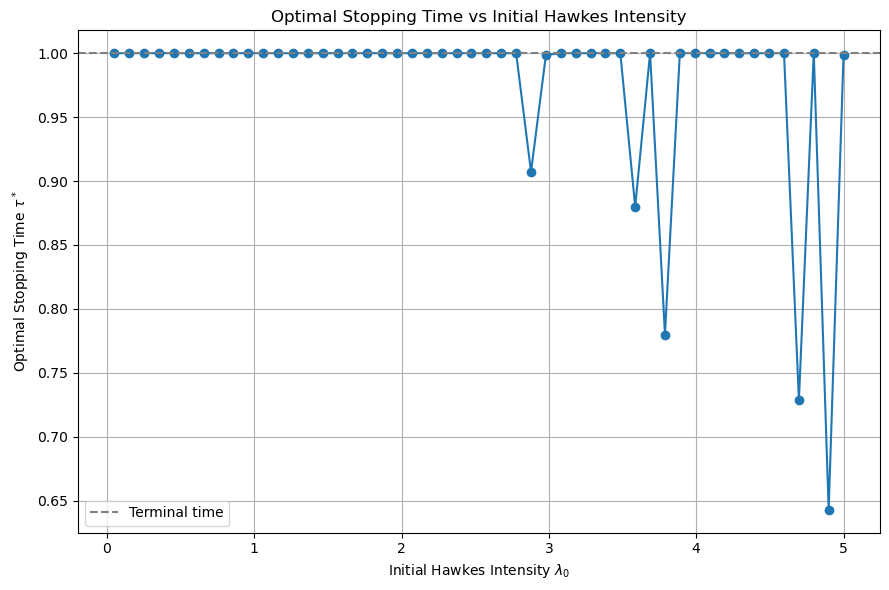}
        \caption{$\phi = 3$}
        \label{fig:tau_L_midsing}
    \end{subfigure}
    \hfill
    \begin{subfigure}[b]{0.48\textwidth}
        \centering
        \includegraphics[width=\linewidth]{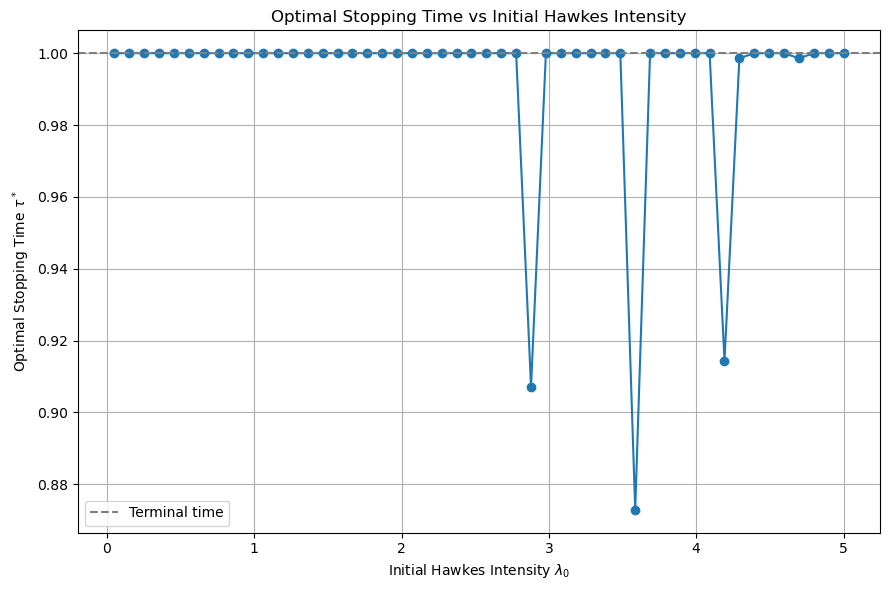}
        \caption{$\phi = 0.01$}
        \label{fig:tau_L_sing}
    \end{subfigure}
    
    \caption{Optimal stopping time $\tau^*$ vs. initial Hawkes intensity $\lambda_0$ with and without singular control.}\label{graphtaustar}
\end{figure}

\bibliographystyle{siam}
\bibliography{P1}

@article{aid2015explicit,
  title={Explicit investment rules with time-to-build and uncertainty},
  author={A{\"\i}d, Ren{\'e} and Federico, Salvatore and Pham, Huy{\^e}n and Villeneuve, Bertrand},
  journal={Journal of Economic Dynamics and Control},
  volume={51},
  pages={240--256},
  year={2015},
  publisher={Elsevier}
}

@article{alfonsi2013capacitary,
  title={Capacitary measures for completely monotone kernels via singular control},
  author={Alfonsi, Aur{\'e}lien and Schied, Alexander},
  journal={SIAM Journal on Control and Optimization},
  volume={51},
  number={2},
  pages={1758--1780},
  year={2013},
  publisher={SIAM}
}

@article{alvarez2001singular,
  title={Singular stochastic control, linear diffusions, and optimal stopping: A class of solvable problems},
  author={Alvarez, Luis HR},
  journal={SIAM Journal on Control and Optimization},
  volume={39},
  number={6},
  pages={1697--1710},
  year={2001},
  publisher={SIAM}
}

@article{alvarez2009singular,
  title={On singular stochastic control and optimal stopping of spectrally negative jump diffusions},
  author={Alvarez, Luis HR and Rakkolainen, Teppo A},
  journal={Stochastics: An International Journal of Probability and Stochastics Processes},
  volume={81},
  number={1},
  pages={55--78},
  year={2009},
  publisher={Taylor \& Francis}
}

@article{an2010combined,
  title={Combined optimal stopping and singular stochastic control},
  author={An, Ta Thi Kieu},
  journal={Stochastic Analysis and Applications},
  volume={28},
  number={3},
  pages={401--414},
  year={2010},
  publisher={Taylor \& Francis}
}

@article{bovo2025variational,
  title={Variational Inequalities on Unbounded Domains for Zero-Sum Singular Controller vs. Stopper Games},
  author={Bovo, Andrea and De Angelis, Tiziano and Issoglio, Elena},
  journal={Mathematics of Operations Research},
  volume={50},
  number={1},
  pages={277--312},
  year={2025},
  publisher={INFORMS}
}

@article{aubert2025optimal,
  title={Optimal dividend and capital injection under self-exciting claims},
  author={Aubert, Paulin and Chevalier, Etienne and Vath, Vathana Ly},
  journal={arXiv preprint arXiv:2511.19701},
  year={2025}
}

@article{case2016analysis,
  title={Analysis of the cyber attack on the Ukrainian power grid},
  author={Case, Defense Use},
  journal={Electricity information sharing and analysis center (E-ISAC)},
  volume={388},
  number={1-29},
  pages={3},
  year={2016},
  publisher={Washington, DC}
}

@article{abraham2025cyber,
  title={Cyber-Attacks on Energy Infrastructure—A Literature Overview and Perspectives on the Current Situation},
  author={Abraham, Doney and Houmb, Siv Hilde and Erdodi, Laszlo},
  journal={Applied Sciences},
  volume={15},
  number={17},
  pages={9233},
  year={2025},
  publisher={MDPI}
}

@article{lucia2002electricity,
  title={Electricity prices and power derivatives: Evidence from the Nordic power exchange},
  author={Lucia, Julio J and Schwartz, Eduardo S},
  journal={Review of derivatives research},
  volume={5},
  number={1},
  pages={5--50},
  year={2002},
  publisher={Springer}
}

@article{cartea2005pricing,
  title={Pricing in electricity markets: a mean reverting jump diffusion model with seasonality},
  author={Cartea, Alvaro and Figueroa, Marcelo G},
  journal={Applied Mathematical Finance},
  volume={12},
  number={4},
  pages={313--335},
  year={2005},
  publisher={Taylor \& Francis}
}

@book{benth2008stochastic,
  title={Stochastic modelling of electricity and related markets},
  author={Benth, Fred Espen and Benth, Jurate Saltyte and Koekebakker, Steen},
  volume={11},
  year={2008},
  publisher={World Scientific}
}

@book{bertsekas2012dynamic,
  title={Dynamic programming and optimal control: Volume I},
  author={Bertsekas, Dimitri},
  volume={4},
  year={2012},
  publisher={Athena scientific}
}

@book{bertsekas1996stochastic,
  title={Stochastic optimal control: the discrete-time case},
  author={Bertsekas, Dimitri and Shreve, Steven E},
  volume={5},
  year={1996},
  publisher={Athena Scientific}
}

@book{bensoussan2011applications,
  title={Applications of variational inequalities in stochastic control},
  author={Bensoussan, Alain and Lions, J-L},
  volume={12},
  year={2011},
  publisher={Elsevier}
}

@article{baldursson1996irreversible,
  title={Irreversible investment and industry equilibrium},
  author={Baldursson, Fridrik M and Karatzas, Ioannis},
  journal={Finance and stochastics},
  volume={1},
  number={1},
  pages={69--89},
  year={1996},
  publisher={Springer}
}

@article{baldwin2017contagion,
  title={Contagion in cyber security attacks},
  author={Baldwin, Adrian and Gheyas, Iffat and Ioannidis, Christos and Pym, David and Williams, Julian},
  journal={Journal of the Operational Research Society},
  volume={68},
  number={7},
  pages={780--791},
  year={2017},
  publisher={Taylor \& Francis}
}

@article{bather1967sequential,
  title={Sequential decisions in the control of a space-ship (finite fuel)},
  author={Bather, John and Chernoff, Herman},
  journal={Journal of Applied Probability},
  volume={4},
  number={3},
  pages={584--604},
  year={1967},
  publisher={Cambridge University Press}
}

@article{bielecki2022construction,
  title={Construction and simulation of generalized multivariate Hawkes processes},
  author={Bielecki, Tomasz R and Jakubowski, Jacek and Niew{\k{e}}g{\l}owski, Mariusz},
  journal={Methodology and Computing in Applied Probability},
  volume={24},
  number={4},
  pages={2865--2896},
  year={2022},
  publisher={Springer}
}

@article{benevs1980some,
  title={Some solvable stochastic control problemst},
  author={Bene{\v{s}}, V{\'a}clav E and Shepp, Larry A and Witsenhausen, Hans S},
  journal={Stochastics: An International Journal of Probability and Stochastic Processes},
  volume={4},
  number={1},
  pages={39--83},
  year={1980},
  publisher={Taylor \& Francis}
}

@article{bensoussan2024stochastic,
  title={Stochastic control for diffusions with self-exciting jumps: An overview},
  author={Bensoussan, Alain and Chevalier-Roignant, Benoit},
  journal={Mathematical Control and Related Fields},
  volume={14},
  pages={1452--1476},
  year={2024},
  publisher={Mathematical Control and Related Fields}
}

@article{bremaud1996stability,
  title={Stability of nonlinear Hawkes processes},
  author={Br{\'e}maud, Pierre and Massouli{\'e}, Laurent},
  journal={The Annals of Probability},
  pages={1563--1588},
  year={1996},
  publisher={JSTOR}
}

@article{budhiraja2007convergent,
  title={Convergent numerical scheme for singular stochastic control with state constraints in a portfolio selection problem},
  author={Budhiraja, Amarjit and Ross, Kevin},
  journal={SIAM Journal on Control and Optimization},
  volume={45},
  number={6},
  pages={2169--2206},
  year={2007},
  publisher={SIAM}
}

@article{burdzy2009skorokhod,
  title={The Skorokhod problem in a time-dependent interval},
  author={Burdzy, Krzysztof and Kang, Weining and Ramanan, Kavita},
  journal={Stochastic processes and their applications},
  volume={119},
  number={2},
  pages={428--452},
  year={2009},
  publisher={Elsevier}
}

@article{cadenillas1994stochastic,
  title={The stochastic maximum principle for a singular control problem},
  author={Cadenillas, Abel and Haussmann, Ulrich G},
  journal={Stochastics: An International Journal of Probability and Stochastic Processes},
  volume={49},
  number={3-4},
  pages={211--237},
  year={1994},
  publisher={Taylor \& Francis}
}

@article{ceci2004mixed,
  title={Mixed optimal stopping and stochastic control problems with semicontinuous final reward for diffusion processes},
  author={Ceci, Claudia and Bassan, Bruno},
  journal={Stochastics and Stochastic Reports},
  volume={76},
  number={4},
  pages={323--337},
  year={2004},
  publisher={Taylor \& Francis}
}

@article{dassios2013exact,
  title={Exact simulation of Hawkes process with exponentially decaying intensity},
  author={Dassios, Angelos and Zhao, Hongbiao},
  year={2013},
  journal={Electronic Communications in Probability},
  doi={10.1214/ECP.v18-2717}
}

@article{dufour2004singular,
  title={Singular stochastic control problems},
  author={Dufour, Fran{\c{c}}ois and Miller, Boris},
  journal={SIAM journal on control and optimization},
  volume={43},
  number={2},
  pages={708--730},
  year={2004},
  publisher={SIAM}
}

@article{dumitrescu2021approximation,
  title={Approximation schemes for mixed optimal stopping and control problems with nonlinear expectations and jumps},
  author={Dumitrescu, Roxana and Reisinger, Christoph and Zhang, Yufei},
  journal={Applied Mathematics \& Optimization},
  volume={83},
  number={3},
  pages={1387--1429},
  year={2021},
  publisher={Springer}
}

@article{foschi2021measuring,
  title={Measuring discrepancies between poisson and exponential hawkes processes},
  author={Foschi, Rachele},
  journal={Methodology and Computing in Applied Probability},
  volume={23},
  number={1},
  pages={219--239},
  year={2021},
  publisher={Springer}
}

@article{guo2008connections,
  title={Connections between singular control and optimal switching},
  author={Guo, Xin and Tomecek, Pascal},
  journal={SIAM Journal on Control and Optimization},
  volume={47},
  number={1},
  pages={421--443},
  year={2008},
  publisher={SIAM}
}

@article{guo2009class,
  title={A class of singular control problems and the smooth fit principle},
  author={Guo, Xin and Tomecek, Pascal},
  journal={SIAM Journal on Control and Optimization},
  volume={47},
  number={6},
  pages={3076--3099},
  year={2009},
  publisher={SIAM}
}

@article{harrison1987brownian,
  title={Brownian models of open queueing networks with homogeneous customer populations},
  author={Harrison, J Michael and Williams, Ruth J},
  journal={Stochastics: An International Journal of Probability and Stochastic Processes},
  volume={22},
  number={2},
  pages={77--115},
  year={1987},
  publisher={Taylor \& Francis}
}

@article{haussmann1995singularI,
  title={Singular optimal stochastic controls I: Existence},
  author={Haussmann, Ulrich G and Suo, Wulin},
  journal={SIAM Journal on Control and Optimization},
  volume={33},
  number={3},
  pages={916--936},
  year={1995},
  publisher={SIAM}
}

@article{haussmann1995singularII,
  title={Singular optimal stochastic controls II: Dynamic programming},
  author={Haussmann, Ulrich G and Suo, Wulin},
  journal={SIAM Journal on Control and Optimization},
  volume={33},
  number={3},
  pages={937--959},
  year={1995},
  publisher={SIAM}
}

@article{hawkes1971spectra,
  title={Spectra of some self-exciting and mutually exciting point processes},
  author={Hawkes, Alan G},
  journal={Biometrika},
  volume={58},
  number={1},
  pages={83--90},
  year={1971},
  publisher={Oxford University Press}
}

@article{hemmati2022identification,
  title={Identification of cyber-attack/outage/fault in zero-energy building with load and energy management strategies},
  author={Hemmati, Reza and Faraji, Hossien},
  journal={Journal of Energy Storage},
  volume={50},
  pages={104290},
  year={2022},
  publisher={Elsevier}
}

@article{hening2019harvesting,
  title={Harvesting of interacting stochastic populations},
  author={Hening, Alexandru and Tran, Ky Quan and Phan, Tien Trong and Yin, George},
  journal={Journal of Mathematical Biology},
  volume={79},
  number={2},
  pages={533--570},
  year={2019},
  publisher={Springer}
}

@article{hillairet2024chaotic,
  title={On the chaotic expansion for counting processes},
  author={Hillairet, Caroline and R{\'e}veillac, Anthony},
  journal={Electronic Journal of Probability},
  volume={29},
  pages={1--33},
  year={2024},
  publisher={The Institute of Mathematical Statistics and the Bernoulli Society}
}

@article{howard1972risk,
  title={Risk-sensitive Markov decision processes},
  author={Howard, Ronald A and Matheson, James E},
  journal={Management science},
  volume={18},
  number={7},
  pages={356--369},
  year={1972},
  publisher={INFORMS}
}

@article{iglehart1970multiple,
  title={Multiple channel queues in heavy traffic. I},
  author={Iglehart, Donald L and Whitt, Ward},
  journal={Advances in Applied Probability},
  volume={2},
  number={1},
  pages={150--177},
  year={1970},
  publisher={Cambridge University Press}
}

@article{jin2013numerical,
  title={Numerical methods for optimal dividend payment and investment strategies of regime-switching jump diffusion models with capital injections},
  author={Jin, Zhuo and Yang, Hailiang and Yin, G George},
  journal={Automatica},
  volume={49},
  number={8},
  pages={2317--2329},
  year={2013},
  publisher={Elsevier}
}

@article{jusselin2023scaling,
  title={Scaling limit for stochastic control problems in population dynamics},
  author={Jusselin, Paul and Mastrolia, Thibaut},
  journal={Applied Mathematics \& Optimization},
  volume={88},
  number={1},
  pages={14},
  year={2023},
  publisher={Springer}
}

@article{karatzas1983class,
  title={A class of singular stochastic control problems},
  author={Karatzas, Ioannis},
  journal={Advances in Applied Probability},
  volume={15},
  number={2},
  pages={225--254},
  year={1983},
  publisher={Cambridge University Press}
}

@article{karatzas2000finite,
  title={Finite-fuel singular control with discretionary stopping},
  author={Karatzas, Ioannis and Ocone, Daniel and Wang, Hui and Zervos, Mihail},
  journal={Stochastics: An International Journal of Probability and Stochastic Processes},
  volume={71},
  number={1-2},
  pages={1--50},
  year={2000},
  publisher={Taylor \& Francis}
}

@article{karoui1988probabilistic,
  title={Probabilistic aspects of finite-fuel, reflected follower problems},
  author={Karoui, Nicole El and Karatzas, Ioannis},
  journal={Acta Applicandae Mathematica},
  volume={11},
  number={3},
  pages={223--258},
  year={1988},
  publisher={Springer}
}

@article{khabou2025markov,
  title={Markov approximation for controlled Hawkes Jump-Diffusions with general kernels},
  author={Khabou, Mahmoud and Talbi, Mehdi},
  journal={arXiv preprint arXiv:2507.11294},
  year={2025}
}

@article{koch2021optimal,
  title={Optimal installation of solar panels with price impact: a solvable singular stochastic control problem},
  author={Koch, Torben and Vargiolu, Tiziano},
  journal={SIAM Journal on Control and Optimization},
  volume={59},
  number={4},
  pages={3068--3095},
  year={2021},
  publisher={SIAM}
}

@article{kruk2007explicit,
  title={An explicit formula for the Skorokhod map on [0, a]},
  author={Kruk, Lukasz and Lehoczky, John and Ramanan, Kavita and Shreve, Steven},
  journal={The Annals of Probability},
  year={2007}
}

@article{kushner1990numerical,
  title={Numerical methods for stochastic control problems in continuous time},
  author={Kushner, Harold J},
  journal={SIAM Journal on Control and Optimization},
  volume={28},
  number={5},
  pages={999--1048},
  year={1990},
  publisher={SIAM}
}

@article{kushner1991numerical,
  title={Numerical methods for stochastic singular control problems},
  author={Kushner, Harold J and Martins, Luiz Felipe},
  journal={SIAM journal on control and optimization},
  volume={29},
  number={6},
  pages={1443--1475},
  year={1991},
  publisher={SIAM}
}

@article{lewis1976simulation,
  title={Simulation of nonhomogeneous Poisson processes with log linear rate function},
  author={Lewis, Peter AW and Shedler, GS},
  journal={Biometrika},
  volume={63},
  number={3},
  pages={501--505},
  year={1976},
  publisher={Oxford University Press}
}

@article{martins1990routing,
  title={Routing and singular control for queueing networks in heavy traffic},
  author={Martins, Luiz Felipe and Kushner, Harold J},
  journal={SIAM journal on control and optimization},
  volume={28},
  number={5},
  pages={1209--1233},
  year={1990},
  publisher={SIAM}
}

@article{ogata1981lewis,
  title={On Lewis' simulation method for point processes},
  author={Ogata, Yosihiko},
  journal={IEEE transactions on information theory},
  volume={27},
  number={1},
  pages={23--31},
  year={1981},
  publisher={IEEE}
}

@article{reppen2020optimal,
  title={Optimal dividend policies with random profitability},
  author={Reppen, A Max and Rochet, Jean-Charles and Soner, H Mete},
  journal={Mathematical Finance},
  volume={30},
  number={1},
  pages={228--259},
  year={2020},
  publisher={Wiley Online Library}
}

@article{shreve1994optimal,
  title={Optimal investment and consumption with transaction costs},
  author={Shreve, Steven E and Soner, H Mete},
  journal={The Annals of Applied Probability},
  pages={609--692},
  year={1994},
  publisher={JSTOR}
}

@article{stergiopoulos2020cyber,
  title={Cyber-attacks on the oil \& gas sector: A survey on incident assessment and attack patterns},
  author={Stergiopoulos, George and Gritzalis, Dimitris A and Limnaios, Evangelos},
  journal={Ieee Access},
  volume={8},
  pages={128440--128475},
  year={2020},
  publisher={IEEE}
}

@article{tran2016numerical,
  title={Numerical methods for optimal harvesting strategies in random environments under partial observations},
  author={Tran, Ky and Yin, George},
  journal={Automatica},
  volume={70},
  pages={74--85},
  year={2016},
  publisher={Elsevier}
}

@article{yildiz2024enhancing,
  title={Enhancing load frequency control and cybersecurity in renewable energy microgrids: A fuel cell-based solution with non-integer control under cyber-attack},
  author={Y{\i}ld{\i}z, S{\"u}leyman and Yildirim, Burak and {\"O}zdemir, Mahmut Temel},
  journal={International Journal of Hydrogen Energy},
  volume={75},
  pages={438--449},
  year={2024},
  publisher={Elsevier}
}

@book{applebaum2009levy,
  title={L{\'e}vy processes and stochastic calculus},
  author={Applebaum, David},
  year={2009},
  publisher={Cambridge university press}
}

@book{bauerle2011markov,
  title={Markov decision processes with applications to finance},
  author={B{\"a}uerle, Nicole and Rieder, Ulrich},
  year={2011},
  publisher={Springer Science \& Business Media}
}

@book{ethier2009markov,
  title={Markov processes: characterization and convergence},
  author={Ethier, Stewart N and Kurtz, Thomas G},
  year={2009},
  publisher={John Wiley \& Sons}
}

@book{feinberg2012handbook,
  title={Handbook of Markov decision processes: methods and applications},
  author={Feinberg, Eugene A and Shwartz, Adam},
  volume={40},
  year={2012},
  publisher={Springer Science \& Business Media}
}

@book{fleming2006controlled,
  title={Controlled Markov processes and viscosity solutions},
  author={Fleming, Wendell H and Soner, H Mete},
  year={2006},
  publisher={Springer}
}

@book{gripenberg1990volterra,
    address = {Cambridge},
    series = {Encyclopedia of {Mathematics} and its {Applications}},
    title = {Volterra {Integral} and {Functional} {Equations}},
    isbn = {0-521-37289-5},
    language = {en},
    number = {34},
    publisher = {Cambridge University Press},
    author = {Gripenberg, G. and Londen, S.-O. and Staffans, O.},
    year = {1990},
}

@book{ikeda2014stochastic,
  title={Stochastic differential equations and diffusion processes},
  author={Ikeda, Nobuyuki and Watanabe, Shinzo},
  volume={24},
  year={2014},
  publisher={Elsevier}
}

@book{kallenberg2021foundations,
	address = {Switzerland},
	edition = {3rd},
	series = {Probability {Theory} and {Stochastic} {Modelling}},
	title = {Foundations of {Modern} {Probability}},
	volume = {2},
	isbn = {978-3-030-61870-4},
	language = {English},
	publisher = {Springer Nature},
	author = {Kallenberg, Olav},
	year = {2021},
}

@book{kushner1992numerical,
	edition = {1},
	series = {Stochastic {Modelling} and {Applied} {Probability}},
	title = {Numerical {Methods} for {Stochastic} {Control} {Problems} in {Continuous} {Time}},
	isbn = {978-1-4684-0443-2},
	language = {English},
	number = {24},
	publisher = {Springer-Verlag},
	author = {Kushner, Harold J. and Dupuis, Paul G.},
	year = {1992},
}

@book{parthasarathy2005probability,
  title={Probability measures on metric spaces},
  author={Parthasarathy, Kalyanapuram Rangachari},
  volume={352},
  year={2005},
  publisher={American Mathematical Soc.}
}

@book{puterman2014markov,
  title={Markov decision processes: discrete stochastic dynamic programming},
  author={Puterman, Martin L},
  year={2014},
  publisher={John Wiley \& Sons}
}

@incollection{kunita2004stochastic,
  title={Stochastic differential equations based on L{\'e}vy processes and stochastic flows of diffeomorphisms},
  author={Kunita, Hiroshi},
  booktitle={Real and Stochastic Analysis: New Perspectives},
  pages={305--373},
  year={2004},
  publisher={Springer}
}

@incollection{shreve1988introduction,
  title={An introduction to singular stochastic control},
  author={Shreve, Steven E},
  booktitle={Stochastic differential systems, stochastic control theory and applications},
  pages={513--528},
  year={1988},
  publisher={Springer}
}

@phdthesis{sun2024optimal,
  title={Optimal Singular Control Problems and Quantum-Inspired Algorithms in Finance},
  author={Sun, Chuhao},
  year={2024},
  school={University of Michigan--Ann Arbor},
  doi={https://dx.doi.org/10.7302/23776}
}
%\nocite{*}

\end{document}